\documentclass[a4paper,10pt,twoside]{amsart}

\usepackage[left=2cm,right=2cm,top=2cm,bottom=2cm]{geometry}
\usepackage{verbatim}
\usepackage[latin1]{inputenc}
\usepackage[OT1]{fontenc}
\usepackage[english]{babel}
\usepackage{amsfonts}
\usepackage{amsmath}
\usepackage{amssymb}
\usepackage{amsthm}
\usepackage{amsaddr}
\usepackage{enumerate}
\usepackage{color}
\usepackage{esint}
\usepackage{hyperref}
\usepackage{cite}
\usepackage{mathrsfs}
\usepackage{cleveref}
\usepackage{bbm}
\usepackage{graphicx}
\usepackage{cancel}

\usepackage{todonotes}

\DeclareMathOperator{\supp}{supp}

\DeclareMathOperator*{\EssSup}{ess\,sup}

\newcommand{\Div}{\mathrm{div}}

\newcommand{\ee}{\varepsilon}

\newcommand{\pre}{ \mathrm{p}}
\newcommand{\uu}{ u}

\newcommand{\I}{\mathcal{I}}

\newcommand{\EE}{\mathcal{E}}

\newcommand{\RR}{\mathbb{R}}

\newcommand{\NN}{\mathbb{N}}

\newcommand{\ssubset}{\subset\joinrel\subset}
\newcommand{\rightharpoonupstar}{\stackrel{\ast}{\rightharpoonup}}

\newtheorem{theorem}{Theorem}[section]

\newtheorem{prop}[theorem]{Proposition}
\newtheorem{lemma}[theorem]{Lemma}

\newtheorem{remark}[theorem]{Remark}

\DeclareFontFamily{OT1}{rsfs}{}
\DeclareFontShape{OT1}{rsfs}{m}{n}{ <-7> rsfs5 <7-10> rsfs7 <10-> rsfs10}{}
\DeclareMathAlphabet{\mycal}{OT1}{rsfs}{m}{n}

\newcommand\smallc{
  \mathchoice
    {{\scriptstyle\mathcal{C}}}
    {{\scriptstyle\mathcal{C}}}
    {{\scriptscriptstyle\mathcal{C}}}
    {\scalebox{.7}{$\scriptscriptstyle\mathcal{C}$}}
  }

\newcommand*\Lbar{\ooalign{${\displaystyle L}$\cr\hidewidth\raisebox{.3ex}{--}\hidewidth}}
\newcommand*\lbar{\ooalign{${\displaystyle L}$\cr\hidewidth\raisebox{.3ex}{\hspace{-0.1cm}--}\hidewidth}}

\begin{document}

\title[Colloidal Homogenisation of nematic-LC flows]{\textbf{Colloidal Homogenisation for the Hydrodynamics of Nematic Liquid Crystals}}

\author[
    F. De Anna\qquad 
    A. Schl\"omerkemper\qquad 
    A. Zarnsescu
]{
    Francesco De Anna$^1$ \qquad  
    Anja Schl\"omerkemper$^2$\qquad 
    Arghir Zarnescu$^{3}$
}


 \address{
 $\,^{1,\,2}$ Institute of Mathematics, University of W\"urzburg \\
 $\,^3$BCAM, Basque Center for Applied Mathematics\\
 $\,^3$IKERBASQUE, Basque Foundation for Science\\
 $\,^3$``Simion Stoilow'' Institute of the Romanian Academy
 \\
 \medskip
 $\,^1$francesco.deanna@uni-wuerzburg.de\\
 $\,^2$anja.schloemerkemper@uni-wuerzburg.de\\
 $\,^3$azarnescu@bcamath.org
}

\begin{abstract}

This paper analytically explores a simplified model for the hydrodynamics of nematic liquid crystal colloids. We integrate a Stokes equation for the velocity field with a Ginzburg-Landau transported heat flow for the director field. The study focuses on a bounded spatial domain containing periodically distributed colloidal particles, which impose no-anchoring conditions on the nematic liquid crystal. By progressively reducing the particle size to zero and simultaneously increasing the number of particles, we delve into the associated homogenisation problem. Our analysis uncovers a form of decoupling where the velocity field asymptotically satisfies a Darcy equation, independent of the director, while the director follows a gradient flow, unaffected by the velocity field. One of the most intricate aspects of the homogenisation process is the absence of an extension operator for the director field that preserves the uniform estimates related to the system's energy. We address this challenge with a novel variation of the Aubin-Lions lemma, specifically adapted for homogenisation problems.


\end{abstract}

\maketitle



\section{Introduction}

\noindent 
 Liquid-crystal colloids (LC colloids) have become a prominent topic in condensed matter physics, featuring a combination of a liquid crystal  as surrounding fluid  and a cluster of microscopically-sized particles. These particles are suspended and can be strategically engineered to leverage the molecular orientation of the liquid crystal, leading to the development of novel materials with improved properties \cite{WOS:A1997WP05600034,WOS:000374357900011,WOS:000283385400001,MR3790092}. 
 

\noindent
In recent years, the mathematical community has experienced a renewed interest in LC colloids, driven by a desire to accurately predict and quantify their key characteristics. This resurgence reflects a growing recognition of the importance and potential applications of these colloids in various fields, ranging from the realms of food and pharmaceuticals to the intricate area of nanoparticle transport. 
Some biomedical applications concern for instance cancer therapies, using nanoparticles as markers to monitor biological binding reactions and for targeted drug delivery (see for instance \cite{zhang2023liquid}).
In \cite{WOS:000374357900011}, additionally, Lavrentovich highlighted how active colloids can be manipulated by influencing the host LC, thanks to techniques such as LC-enabled electrokinetics and defect entrapment.

\noindent 
Despite the broad range of applications, the most mathematical studies have focused however uniquely on the analysis of a static LC host, where the colloidal configurations minimise some energy functionals subjected to homogenisation (cf.~for instance \cite{MR3717151,MR4314141,MR3519975,MR2112873,WOS:000346382700007,WOS:000412437400004}). To the best of our knowledge, there have been limited analytical findings for the case of the hydrodynamics of LC colloids.  
Our work provides a foundation in this regard by investigating the homogenisation problem of a  simplified version of the well-known Ericksen-Leslie equations (cf.~System \eqref{maineq1}-\eqref{in-cdt}), which describe the hydrodynamics of nematic LC upon the director formalism of Oseen and Frank.
More precisely, we consider the following partial differential equations, which couple a Stokes equation for the velocity field $u_\ee$ and a transported gradient flow for the director $d_\ee$:
\begin{alignat}{8}
    - \Delta u_\ee + \nabla \pre_\ee &= - (\nabla d_\ee)^T \Delta d_\ee 
    + F_\ee - \Delta H_\ee
    \qquad		&&(0,T)\times \Omega_\ee,						 	\label{maineq1}\\
	\Div\, u_\ee & = 0	\qquad		&&(0,T)\times \Omega_\ee,		\label{maineq2}\\
	\partial_t d_\ee + u_\ee \cdot  \nabla d_\ee - \Delta d_\ee &= -\big(|d_\ee|^2 -1\big)d_\ee		\qquad		&&(0,T)\times \Omega_\ee. \label{maineq3}
\end{alignat}
This system is a simplified version of a more comprehensive director model for nematic liquid crystals (cf.~also \Cref{rmk:pressure} and \cite{MR1329830}, where the Stokes equations are replaced by a Navier-Stokes counterpart) 
and which does not contain the so-called stretching terms (see for instance \cite{liu2008energetic}). The advantage of using this model is that it allows one to clearly  see the main mathematical challenges, see Section~\ref{sec:challenges}. We believe that some of the tools developed here will be relevant 
for more sophisticated versions of the liquid crystal systems.

\noindent We supplement the equations with initial conditions of the director (cf.~\eqref{in-cdt} below) and the following homogeneous boundary conditions
\begin{equation}\label{bdy-cdt}
        \uu_\ee = 0 \quad 
		\qquad\text{and}\qquad
		\frac{\partial d_\ee }{\partial \nu }= 0
        \qquad 
		\text{on}\quad(0,T)\times \partial \Omega_\ee,
\end{equation}
where $\nu= \nu (x)\in \mathbb R^2$ represents the normal vector at a given point $x\in \partial \Omega_\ee$. 
Details about all components of the equations are postponed to \Cref{sec:main-result}. The parameter $\ee>0$ represents the size of the colloidal particles that we assume uniformly and periodically distributed within a prescribed domain $\Omega$. 

\noindent 
In this article we establish that the family of solutions $(u_\ee, d_\ee)$ with $\ee>0$ exhibit a notable separation in their dynamics, as the scaling parameter $\ee$ tends to zero. Specifically, the solutions converge towards two detached profiles, denoted as $u = u(t,x)$ and $d= d(t,x)$, which satisfy two uncoupled sets of equations (cf.~the Darcy-type system \eqref{eq:Darcy} and  the gradient flow \eqref{eq:d}). This decoupling phenomenon suggests that the influence between the velocity field and the director field diminishes as $\ee$ approaches zero, resulting in distinct and independent behaviors for $u$ and $d$ (cf.~\Cref{sec:main-result} and \Cref{main_thm}). 


\noindent 
In the next two subsections, we first present a brief overview of the mathematical literature regarding LC colloids (cf.~\Cref{sec:overview-static-lc-colloids}), followed by an overview of the results for the homogenisation problem for Newtonian fluids (cf.~\Cref{sec:overview-homog-stokes}). The mathematical challenges related to our main contribution is then briefly described  in \Cref{sec:challenges}. It focuses on addressing the speciality of the interactions between LC fluids (or more broadly, complex fluids) and colloidal particles, an interaction which  introduces a further  layer of analytical complexity when compared with the existing literature. A comprehensive exposition of our main result is then provided in \Cref{sec:main-result}, followed by the proof in Sections \ref{sec:preliminary-results-homogenisation}-\ref{sec:director}.

\subsection{An overview of the analysis of static LC colloids}\label{sec:overview-static-lc-colloids}$\,$

\medskip\noindent 
As elucidated by Lavrentovich  (see Section 2 in \cite{WOS:000374357900011}), colloidal particles have been observed to induce non-trivial distortions in the surrounding liquid crystal when their size is comparable to, or significantly larger than, the de Gennes-Kleman length. This length represents a crucial ratio between the average Frank elastic constants and the surface anchoring strength of the colloid. Specifically, this regime corresponds to an intricate balance between elastic and surface energies, the latter maintaining an appropriate rescaling, as the number of inclusions increases. 
Consequently, two primary areas emerged within the mathematical community, focusing on addressing either the interaction between few particles (typically one or two colloids) or a large array of particles whose sizes tend to infinitesimal values.


\noindent 
Regarding few particles, a significant analytical result was achieved  by Alama, Bronsard, Golovaty and Lamy in \cite{MR4314141}. The authors examined a single spherical colloid particle within a nematic liquid crystal with strong radial anchoring conditions on the colloidal surface. Their findings highlighted how certain LCs adopt two distinct favored setups around a colloidal sphere: either molecules tend to align and create a singular point defect along an axis of symmetry, or they form a circular ring singularity perpendicular to a symmetry axis. 

\noindent 
For what concerns a large array of particles immersed within a LC host, the work by Berlyand, Cioranescu, and Golovaty in \cite{MR2112873} provided some initial analytical insights. Their investigation pertained to a simplified director model for nematic LCs within a spatial domain saturated with colloids. 
They coupled a Ginzburg-Landau-type energy density with appropriate surface energy terms and, by employing a mesocharacteristics method, they uncovered the homogenisation asymptotic of the corresponding Euler-Lagrange equations.

\noindent 
In \cite{WOS:000518700700003} the authors tackled the tensorial Landau-de Gennes model applied to colloids within a dilute regime. 
By focusing on colloids of identical shapes and a uniform spatial distribution, they unveiled 
that an homogenisation process and a strategic selection of anchoring-energy densities allow to arbitrarily tailor the coefficients of the bulk quartic potential of the Landau-de Gennes theory.
More specifically, the coefficients of the bulk quartic potential of the Landau-de Gennes theory can be arbitarly be tailored with precision via the homogenisation process (with an appropriate relation governing the colloidal surface energy).
Subsequently in \cite{WOS:000518220200006}, the same authors extended  their result further, by taking into account polydispersity in colloidal shapes. Their findings demonstrated that the homogenisation process replicates the impact of an electric field on the liquid crystal, through a careful design of colloidal shapes and a dilute regime. 

\subsection{Homogenisation of Stokes and Navier-Stokes equations}\label{sec:overview-homog-stokes}$\,$

\medskip\noindent 
While there has been limited exploration into the mathematical analysis of the hydrodynamics of LC colloids, extensive research has been dedicated to the analysis of Newtonian fluids interacting with a wide array of obstacles, such as the case of fluids in porous media. This topic is particularly pertinent to the homogenisation problem associated with a wide range of partial differential equations, including the Stokes, Navier-Stokes, and Euler equations, among others. In this overview, we briefly highlight some of the key results for the incompressible Stokes and Navier-Stokes equations which relate to our investigation. 

\noindent 
Thus far, mathematical results confirm the physical observations of porous media, particularly highlighting the significance of the distribution, shape, and especially the size of obstacles in the context of the limit process. The size of these obstacles emerges as a key factor, and it has been recognized that a particular range of lengths leads to a kind of trifurcation in the homogenised problem.  
Particularly insightful in these regards were the pioneering works of Allaire in \cite{MR1079190, MR1079189}. Allaire focused on a family of Stokes and Navier-Stokes equations, operating within a stationary regime (with no time dependence on the state variables) and featuring an expanding array of obstacles with Lipschitz shape  and uniform magnitude, periodically distributed in space.  In this context, setting $\ee>0$ as the distance between two adjacent obstacles and $a_\ee>0$ as the characteristic length of a single obstacle, Allaire's works led to the identification of three possible scenarios based on the ratio $\sigma_\ee = (\ee^3/a_\ee)^{1/2}$ in three dimensions or $\sigma_\ee = \ee |\ln(a_\ee /\ee)|^{1/2}$ in two dimensions (Allaire considered also higher dimensions). These scenarios were classified as
\begin{equation*}
    \lim_{\ee \to 0} \sigma_\ee = \sigma \in \mathbb R,\qquad 
    \lim_{\ee \to 0} \sigma_\ee = +\infty  \qquad 
    \lim_{\ee \to 0} \sigma_\ee = 0
\end{equation*}
corresponding to the regimes of critical size, smaller holes and larger holes, respectively. 

\noindent 
In the critical-size regime, Allaire's research led to the conclusion that the interaction between the stationary fluid and the obstacles is robust enough to introduce a new term of Brinkman type into the Stokes or Navier-Stokes equations. The Brinkman term, i.e.~an additional linear zero-order term for the velocity, accounts for the resistance to flow that occurs due to the presence of porous structures. 

\noindent 
In the regime of smaller holes, however, the obstacles diminish more rapidly and thus apply a weaker resistance to the flow. In particular, the limit problem coincides with the original Stokes or Navier-Stokes equations. 

\noindent 
Finally, as more obstacles are introduced in the larger holes regime, the velocity field gradually decreases and the flow is significantly impeded. Through an appropriate scaling approach, the Stokes or Navier-Stokes equations ultimately approach a behavior resembling that of a Darcy-type law. 

\noindent 
For a revisitation of Allaire's results through a Tartar's cell-problem approach, see also \cite{WOS:000566949600003}.

\noindent After Allaire's groundbreaking results, significant efforts were directed towards investigating similar scenarios in which the underlying Navier-Stokes equations also depend on time (the evolutionary Stokes equations being a relatively easier case). The first notable contribution in this context was made by Mikeli\'c in \cite{MR1131849}. In his study, Mikeli\'c made the assumption that the distance and the size of the obstacles or holes rescale with the same magnitude, corresponding to the larger holes regime identified by Allaire. Mikeli\'c's findings demonstrate that, under the conditions of smooth boundaries for the obstacles and non-homogeneous boundary conditions within the larger spatial domain (while still assuming a vanishing velocity field around the obstacles), the homogenised problem converges to the solution of a Darcy-type law. Notably, this convergence does not require any rescaling of the flow in $\ee>0$ (as it was with Allaire), since the boundary conditions in the limit process are non-homogeneous and they provide a non-vanishing limit of the velocity field as $\ee \to 0$. 

\noindent 
In the context of the critical regime within the Navier-Stokes equations, a significant development occurred later with the research conducted by Feireisl and collaborators in \cite{WOS:000368060900016}. Specifically, in the three-dimensional case, where the obstacles scale as $\epsilon^3$, their work demonstrated a result akin to Allaire's earlier findings. In essence, the homogenization process led to the emergence of an evolutionary Navier-Stokes equations with an additional Brinkman-type term. This achievement was accomplished with minimal assumptions regarding the regularity of the obstacles, requiring that the boundaries adhere to a cone-type relation.

\noindent 
The question that naturally arose was whether the scenarios observed by Allaire would also hold true in more complex settings, involving intricate physics like non-uniform distributions of obstacles or polydispersity. Remarkably, recent analytical findings indicate that even within these more intricate scenarios, the limit process continues to yield behaviors resembling Brinkman and Darcy-type models. For instance, in \cite{WOS:000692316800013}, Giunti demonstrated that a random distribution of larger spherical holes still leads to the emergence of Darcy's law in the asymptotic behavior of the velocity field. Additionally, Hillairet, as discussed in \cite{MR3851058}, derived the Brinkman system for the Stokes equations when the obstacles have spherical shapes, even in cases where the boundary conditions on these obstacles are not homogeneous.

\noindent 
While these findings may be deemed satisfactory when dealing with Newtonian fluids, the unique characteristics of LC colloids suggest the possibility of more intricate geometrical structures emerging at increasing colloidal particles. This paper offers valuable insights into this aspect. We still establish a Darcy law  for the velocity field, but we also derive a limiting equation for the alignment of molecules within the colloid, and our results indicate a significant kind of separation between these two state variables.

\noindent 
For further results concerning homogenisation and hydrodynamics we refer also to \cite{WOS:A1996WF15500005,WOS:000349810800007, MR3851058,WOS:000897817400002}.
\subsection{Challenges of homogenisation for complex fluids and our contributions}\label{sec:challenges}$\,$

\medskip\noindent 
Homogenisation for complex fluids of the type we consider, that involve an additional equation to describe the director in the non-Newtonian component, see \eqref{maineq3}, does not seem to have been addressed so far in the mathematical literature. 
Next we mention two difficulties related to the nonlinearity of this equation. 

\noindent 
A first specific challenge is produced by the boundary conditions for the director fields. Indeed, in contrast to the velocity field $u_{\varepsilon}$, a major challenge with examining the limit of $d_{\varepsilon}$ stems from the lack of homogeneous Dirichlet conditions on the surfaces of colloidal particles. While each $u_\varepsilon \in H^1_0(\Omega_\varepsilon)$ has a norm uniformly bounded in $\varepsilon > 0$, and its null extension $\tilde{u}_\varepsilon$ belongs to $H^1_0(\Omega)$ with norm still uniformly bounded, the director $d_\varepsilon$ is not generally zero at the colloidal particle surfaces, therefore it does not satisfy similar properties. This remains true even when replacing the no-anchoring conditions in \eqref{bdy-cdt} with what the physicists call strong or weak anchoring conditions, corresponding mathematically to inhomogeneougs Dirichlet and respectively inhomogeneous Robin boundary conditions. 

\noindent 
One initial approach to tackle this issue might involve introducing suitable extension operators. However, as indicated in \Cref{rmk:pressure}, to leverage the energy estimates in \eqref{energy-eq}, these operators should extend functions beyond $H^1(\Omega_\varepsilon)$, possibly to $W^{1,4}(\Omega_\varepsilon)$ or $H^2(\Omega_\varepsilon)$. Unfortunately, these extension operators come with continuity constants that diverge as $\varepsilon \to 0$, potentially hindering the transfer of energy estimates in \eqref{energy-eq} to the extended functions and making the overall strategy ineffective. 

\noindent In light of these challenges, we adopt a different strategy based on a concept of ``quasi-compactness'' for functions defined in $(0,T)\times \Omega_\varepsilon$ as $\varepsilon \to 0$. Our approach involves an adaptation of the Aubin-Lions lemma, specifically tailored for homogenization problems in periodically perforated domains (cf.~\Cref{sec:Aubing-Lions} with \Cref{lemma:compactness_Omega_eps} and \Cref{thm:Aubin-Lions-Linfty}). 
Essentially, this adapted lemma states that even if a sequence of functions in $(0,T)\times \Omega_\varepsilon$ cannot be extended to $(0,T)\times \Omega$ to meet the original Aubin-Lions criteria, a parallel outcome is achievable within $(0,T)\times \Omega_\varepsilon$. It allows the identification of a weakly convergent subsequence that,
while not strongly convergent, exhibits a key property typically associated with strong convergence: when multiplied to another weakly convergent sequence, it ensures weak convergence of the product towards the product of their respective limits. 
The product and the limit take into account the volume fraction between the unitary cell and colloidal particle.

\noindent 
Our adaptation of the Aubin-Lions lemma is inspired by a parallel result from Allaire and Murat (refer to Lemma 2.3 in \cite{MR1225439}), where they extended the Rellich-Kondrachov theorem to homogenisation contexts for functions in $H^1(\Omega_\varepsilon)$. Our result broadens this scope, accommodating functions that depends on time and are within more general Sobolev spaces.

\smallskip
\noindent 
A further challenge, specific to the coupled system concerns the different levels of regularity, for $u_\varepsilon$ and $d_\varepsilon$ with $d_\varepsilon$ having natural regularity one derivative higher than $u_\varepsilon$ as can be seen also from the energy law (see \eqref{energy-eq}). This has two implications: on one hand uniform estimates at the higher regularity are harder to get (also taking into account the issue with the boundary conditions previously mentioned) and also that this creates terms in the stress tensor that  are virtually uncontrollable by energy estimates. Indeed, after suitably absorbing some terms in the pressure gradient, 
one has only a control of the forcing term in the $L^1$-norm in time and space 
provided by $d$ in Stokes and this will require us to deal with an extension of the pressure only in $L^p$, with $p\in [1,2)$ (see Remark~\ref{rmk:pressure}).  

\subsection*{Structure of the paper} The forthcoming sections are organized as follows: in \Cref{sec:main-result} we articulate the technical aspects of our findings, while \Cref{sec:preliminary-results-homogenisation} revisits some preliminary tools commonly utilised in homogenisation problems. In \Cref{sec:Aubing-Lions}, we present our adaptation of the Aubin-Lions Lemma, crucial for addressing the director equation \eqref{maineq3}. \Cref{sec:pressure} initiates the analysis of the asymptotic behavior of the Stokes equation \eqref{maineq1}-\eqref{maineq2} for the velocity field, focusing on an appropriate extension of the pressure. \Cref{sec:Darcy} then delves into the Darcy law pertaining to the velocity's limit. Finally, \Cref{sec:director} explores the asymptotic limit of the director equation, frequently applying our adaptation of the Aubin-Lions lemma.

\section{Main result}\label{sec:main-result}
\noindent We first introduce the state variables that characterise our system and govern the behavior of the LC colloid. We set furthermore the domain and framework within which these quantities are defined. 

\noindent 
In this study, we mainly focus on two-dimensional LC colloids and, following a standard approach in homogenisation problems, we define a unitary configuration cell denoted as $Y:= [-1/2,1/2]^2 \subset \mathbb{R}^2$. Additionally,  we consider a closed connected domain $\mathcal T $, representing a region in space that configures the (unitary) geometry of a single colloidal particle within the liquid crystal system. 
$\mathcal T $ is a subset of $\mathring{Y} = ]-1/2, 1/2[^2$, with boundary $\partial \mathcal T$ of class $\mathcal{C}^\infty$. 
We set  $Y^* := Y\setminus \mathcal{T}$ as  the LC part of the colloidal system,  representing the region of the liquid-crystal host within the unitary cell $Y$. We assume that also $Y^*$ is connected.

\noindent 
The hosting LC is confined within a bounded domain $\Omega_\ee\subset \mathbb R^2$, depending on a positive parameter $\ee>0$. The given domain $ \Omega_\ee$ is characterized as the intersection of a larger domain $\Omega$ with boundary $\partial \Omega$ of class $C^\infty$ and a collection of smaller cells $Y_k^{*,\ee}$,  which are scaled-down versions of $Y^*$ using the  parameter $\varepsilon>0$. Specifically, for any $k \in \mathbb Z^2$, the following definitions apply:
\begin{equation*}
    x_k^\ee := \ee k,\qquad 
    Y_k^{*,\ee} := x_k^\ee + \ee Y^*,\qquad 
    \mathcal{T}_k^{\ee} := x_k^\ee + \ee \mathcal{T},\qquad 
    \Omega_\ee :=
    \Omega
    \setminus
    \bigcup\limits_{k\in \mathcal K}  
    \overline{\mathcal{T}_k^{\ee}},
\end{equation*}
where  $\mathcal K = \mathcal K _\ee \subset \mathbb Z^2$ is the set of any index $k\in \mathbb Z^2$ for which the cell $ Y_k^{*,\ee}$ is all included in the domain $\Omega$.  In other words, we selectively ``perforate'' only those cells that lie completely within $\Omega$, ensuring that no colloidal particle intersects the boundary $\partial \Omega$.

\noindent
The parameter $\ee>0$ serves as a scaling factor, allowing  the study of colloidal particles with small size. We are indeed interested in the asymptotics as $\ee\to 0$.


\smallskip
\noindent
The state variable describing the privileged molecular orientation of the liquid crystal is a two-dimensional vector field $d_\ee=d_\ee(t,x)\in \mathbb R^2$ (known as director), depending on time $t\in (0, T)$ (for a general lifespan $T>0$) and spatial variables $x\in \Omega_\ee$. Furthermore, we denote by $u_\ee = u_\ee(t,x) \in \mathbb{R}^2$ the velocity field of the liquid crystal. 

\smallskip
\noindent 
 In addition to the system \eqref{maineq1}--\eqref{maineq3} and \eqref{bdy-cdt} above, we specify the initial data for the director field $d_\ee : (0,T)\times \Omega \to  \mathbb R^2$ as
\begin{equation}\label{in-cdt}
    d_{\ee|t=0} = d_{\ee, \rm in} \in L^\infty(\Omega_\ee)\cap H^1(\Omega_\ee)
    \quad \text{with}\quad
    \|d_{\ee, in}\|_{L^\infty(\Omega_\ee)}\leq 1
\end{equation}
and we set two types of forcing terms, denoted as $F_\ee ,\, H_\ee :(0,T)\times \Omega \to \mathbb R^2$. These vector functions are defined across the whole time-space domain $(t,x)\in (0,T)\times \Omega$
and belong to
\begin{equation}\label{forcing-term-F}
    F_\ee \in L^2( (0,T) \times \Omega ) ,\qquad 
    H_\ee \in L^2(0,T;H^1_0(\Omega)) ,\quad 
    \text{with }H_\ee \equiv 0 \text{ in }\Omega\setminus \Omega_\ee.
\end{equation}
The last relation, in particular, implies also $H_\ee \in L^2(0,T, H^1_0(\Omega_\ee))$.
For a fixed value of the scaling parameter $\ee>0$, System \eqref{maineq1}-\eqref{in-cdt} admits a unique global-in-time Leray-type solution $(u_\ee,d_\ee)$. These solutions can be constructed using a standard argument based on a Galerkin method (cf.~for instance \cite{MR1329830}, where the Stokes equations are replaced by their Navier-Stokes counterpart). Leray-type solutions of \eqref{maineq1}-\eqref{in-cdt} are distributional solutions that belong to the function space 
\begin{equation}\label{fct-space:weak-solutions}
        u_\ee \in 
        L^2(0,T; H^1_0(\Omega_\ee)),\quad
        d_\ee \in L^\infty(0,T;H^1(\Omega_\ee))
        \cap 
        L^2(0,T; H^2(\Omega_\ee))
\end{equation}
and satisfy the energy inequality
\begin{equation}\label{energy-eq}
\begin{aligned}
  \int_{\Omega_\ee}
  \left( 
    \frac{|\nabla d_\ee|^2}{2}
    + 
    \frac{(|d_\ee|^2-1)^2}{4}
  \right)
  (t,x)
  d x
  +
  \int_0^t
  \!\!\!
  \int_{\Omega_\ee}  
  \left( 
  |\nabla \uu_\ee|^2 
  +
  \big| 
    \Delta d_\ee - (|d_\ee|^2-1)d_\ee
  \big|^2
  \right)
  (s,x)
  d xd s
  \\
  \leq 
  \int_{\Omega_\ee}
    \left( 
        \frac{|\nabla d_{\ee, \rm in}|^2}{2}
        +
        \frac{(|d_{\ee, \rm in}|^2-1)^2}{4}
    \right)(x)d x
    +
    \int_0^t \int_{\Omega_\ee} 
    \left(
    F_\ee\cdot  u_\ee
    +
    \nabla H_\ee: \nabla u_\ee 
    \right)(s,x)
    d x d s.
\end{aligned}
\end{equation}
 at any time $t\in (0,T)$. Moreover the pressure $\pre_\ee\in L^2(0,T;\Lbar^2_0(\Omega_\ee))$ of \eqref{maineq1} is implicitly defined by the incompressibility condition \eqref{maineq2} (cf.~also \Cref{rmk:pressure}), where $\Lbar^2_0(\Omega_\ee)$ denotes the closed subset of all square-integrable functions with null average: for any bounded Lebesgue-measurable subset $\Sigma\subseteq \mathbb R^d$ and any $p\in [1,\infty]$, we set
 \begin{equation*}
     \Lbar^p_0(\Sigma) := \bigg\{ f \in L^p(\Sigma)\quad \text{such that}\quad \int_{\Sigma} f(x) dx = 0\bigg\}.
 \end{equation*} 
 Our primary objective is to investigate and determine the asymptotic behavior of the solutions $(u_\ee, d_\ee)$ and the pressure $\pre_\ee$ (or rather some appropriate extension of $\nabla \pre_\ee$) as the scaling parameter $\ee\to 0$, namely when the characteristic length scale of the colloidal particles tend towards zero.

\smallskip
\noindent 
To state our result in details, we shall first introduce some notation and standard tools of homogenisation, that we will repeatedly employ throughout this paper. Given a general function $f:\Omega_\ee \to \mathbb R$, we denote by $\tilde{f}$ the unique extension $f:\Omega\to \mathbb R$, which is identically null $\tilde{f}(x) = 0$ for all $x \in \Omega \setminus \Omega_\ee$. Additionally, we define the following matrices $\mathbb{A} \in \mathbb{R}^{2\times2}$ and $\mathbb{B} \in \mathbb{R}^{2\times2}$, with components
\begin{equation}\label{eq:matrices-AB}
   \mathbb A_{ij}
   :=
   \frac{1}{|Y|}
   \int_{Y^*} \nabla  \psi_{i}(y) \cdot  \nabla  \psi_{j}(y)  dy
   ,
   \qquad
   \mathbb B_{ij}
   :=
   \frac{1}{|Y|}
   \int_{Y^*} 
   \nabla_y \omega^{i}(y) 
   :
   \nabla_y \omega^{j}(y) 
   dy,\qquad i,j = 1,2.
\end{equation}
For $i\in \{1,2\}$, the above scalar function $ \psi_i\in H^1(Y^*)$ and 2D-vector function $\omega^i \in H^1(Y^*)$ are solutions of the following auxiliary problems associated to the reference cell $Y^*$:
\begin{equation*}
    \left\{
        \begin{alignedat}{4}
		&- \Delta_y \psi_i(y) = 0\qquad 
            &&
            y \text{ in }Y^*,\\
		&\mu(y) \cdot \nabla_y \psi_i(y) = 0
            \qquad 
            &&
            y\text{ on }\partial \mathcal{T},
            \\
           & 
           ( \tilde \psi_i - y_i) 
           \text{ is $Y$-periodic and average free},
           \hspace{-1.5cm}
           &&
	 \end{alignedat}
    \right.
    \qquad 
    \left\{
        \begin{aligned}
	  &\begin{alignedat}{4}
		&- \Delta_y \omega^i + \nabla_y \pi^i = e^i\qquad 
            &&
            \text{in }Y^*,\\
		&\Div_y\, \omega^i = 0
            \qquad 
            &&
            \text{in }Y^*,
            \\
            &
            \omega^i = 0
            &&
            \text{on }\partial \mathcal{T},\\
            & 
           \tilde \omega^i
           \text{ is $Y$-periodic and average free.}
            \hspace{-1cm}
           &&
  \end{alignedat}
      \end{aligned}
      \right.
\end{equation*}
Here, $\mu = \mu(y)$ denotes the inward normal vector to the unitary particle $\mathcal{T}$, while $\{e^1, e^2\}$ is the canonical basis of $\mathbb R^2$. For the matrix $\mathbb A$, we refer to Allaire and Murat \cite{MR1225439} (where $\mathbb A$ is denoted as $\tilde A$  in (1.5)-(1.6)). Likewise, see Mikeli\'c \cite{MR1131849} for the permeability tensor $\mathbb B$ (expressed with its equivalent formulation $K_{ij} = \int_{Y^*}e^j\cdot \omega^i dy$ in (4.9)-(4.11)). 
From the assumption that $Y^*= Y\setminus \mathcal T$ is connected and has a smooth boundary, both matrices $\mathbb A$ and $\mathbb B$ are symmetric and positive definite. (For $\mathbb{A}$, see \cite{MR1225439}, and for $\mathbb{B}$, we refer to \cite[Chapter 7, Proposition~2.2]{Sanchez-Palencia:1391389}).

\noindent
Our main result states that as $\ee\to 0$, the Leray solutions $(u_\ee, d_\ee)$ of \eqref{maineq1} - \eqref{in-cdt} weakly converge to a profile $(u, d)$, which is solution of the following separated equations: 

\begin{minipage}{0.4\linewidth}
\begin{equation}\label{eq:Darcy}
    \begin{aligned}
    \left\{
    \begin{alignedat}{8}
        &  u + \mathbb B\nabla P = G \quad         
        &&(0,T)\times \Omega,\\
        & \Div\, u = 0 
        &&(0,T)\times  \Omega,\\
        &u \cdot \nu = 0
        &&(0,T)\times \partial \Omega,\\
    \end{alignedat}
    \right.
\end{aligned}
\end{equation}
\end{minipage}
\begin{minipage}{0.55\linewidth}
\begin{equation}\label{eq:d}
\begin{aligned}
    \left\{
    \begin{alignedat}{8}
        &\partial_t d  - 
            \Div
            \Big(
                \frac{1}{\theta} \mathbb{A} \nabla d
            \Big) = - (|d|^2-1)d \quad  
        &&(0,T)\times \Omega,\\
        &\nu \cdot \mathbb{A}\nabla d = 0
        &&(0,T)\times \partial \Omega,\\
        &d_{|t = 0} = d_{\rm in}
        &&\hspace{1.3cm} \Omega.\\
    \end{alignedat}
    \right.
\end{aligned}
\end{equation}
\end{minipage}

\smallskip
\noindent 
with $\theta = |Y^*|/|Y|$. The first equation embodies a Darcy-type model for the velocity field, while the second equation describes a gradient-flow for the director field. The two-dimensional vector function $G \in L^2((0,T)\times \Omega)$ depends on the limits of the forcing terms $F_\ee$ and $H_\ee$, while $d_{\rm in} \in H^1(\Omega)\cap L^\infty(\Omega)$ depends on the limit of the initial data $d_{\rm in, \ee}$ (cf.~\Cref{main_thm} for more details). 
Both systems \eqref{eq:Darcy} and \eqref{eq:d} possess unique weak solutions
\begin{equation}\label{function-spaces-for-u-and-d-at-the-limit}
    u \in L^2((0,T)\times \Omega)
    \quad \text{and}\quad 
    d \in L^\infty((0,T) \times \Omega)\cap 
    L^\infty(0,T; H^1(\Omega)) \cap L^2(0,T; H^2(\Omega)).
\end{equation}
A solution of \eqref{eq:Darcy} is due to the Lax-Milgram theorem, which ensures the existence and uniqueness of a pressure $P$ in $L^2(0,T;\lbar^2_0(\Omega)\cap H^1(\Omega))$ such that
\begin{equation*}
      \int_0^T \int_\Omega \mathbb B \nabla P \cdot \nabla \phi \, dx = \int_0^T \int_\Omega G \cdot \nabla \phi \, dx,
\end{equation*}
for any $\phi \in L^2(0,T;H^1(\Omega))$. Therefore, $u = G-\mathbb B \nabla P$ is the unique solution of \eqref{eq:Darcy}. Similarly, a unique weak solution $d$ of \eqref{eq:d} can be obtained through a standard Galerkin approach, coupled with the a-priori estimate
\begin{equation*}
    \frac{1}{2} \int_{\Omega}  \mathbb A\nabla d(t,x):\nabla d(t,x)dx  +
    \int_0^t\int_{\Omega} 
    \big|\Delta d(s,x) - (d(s,x)^2-1)d(s,x) \big|^2 dxds
    \leq 
     \frac{1}{2} 
     \int_{\Omega}  \mathbb A\nabla d_{\rm in }(x):\nabla d_{\rm in}(x)dx,    
\end{equation*}
at any $t\in (0,T)$ and the weak maximum principle of System \eqref{eq:d}. Our primary focus lies therefore in examining the weak convergence of the extensions $\tilde u_\ee/\ee$ and $\tilde d_\ee$ towards $u$ and $d$, respectively. Our main result reads as follows:
\begin{theorem}\label{main_thm}
Consider a family of initial data $d_{\ee, \rm in}$ as in \eqref{in-cdt} with $H^1(\Omega_\ee)$-norms uniformly bounded in $\ee>0$. Assume that the extensions $\tilde d_{\ee, \rm in}$ weakly-$\ast$ converges to a vector function $\theta\,d_{\rm in}$ in $ L^\infty(\Omega)$, with $\theta = |Y^*|/|Y|$. Assume that $H_\ee$ are uniformly bounded in $L^2(0,T;H^1_0(\Omega))$ and the following convergences hold true: 
\begin{equation*}
   \ee F_\ee \to F\quad  \text{strongly in }L^2((0,T)\times \Omega)\quad 
   \text{and} \quad \frac{H_{\ee}}{\ee}\rightharpoonup H\quad \text{weakly in } L^2((0,T)\times \Omega).
\end{equation*}
Define the two-dimensional vector field $G \in L^2((0,T)\times \Omega)$ with components
\begin{equation*}
    G_i = \bigg(\frac{1}{|Y|} \int_{Y^*}\omega^i(y)dy \bigg)\cdot F + H_i,\quad 
    \text{with }i = 1,2.
\end{equation*}
Then the following weak-convergence results hold true:
\begin{enumerate}[(i)]
    \item\label{part-i-main-thm} The family of functions $\tilde \uu_\ee$  converges strongly to zero in $L^2((0,T) \times \Omega)$, as $\ee \to 0$, while $\tilde u_\ee/\ee$ weakly converges in $L^2((0,T) \times \Omega)$
    towards the unique solution $u$ of \eqref{eq:Darcy}.
    \item There exists a family of pressures $P_\ee:(0,T)\times \Omega\to \mathbb R$  which are in $L^2(0,T;\lbar^p_0(\Omega))$, for any $p\in [1,2)$,  and satisfy $P_{\ee}(t,x) = \pre_\ee(t,x) + \alpha_\ee $   a.e.~$(t,x)\in (0,T)\times \Omega_\ee$, for constants $\alpha_\ee\in \mathbb R$. 
    Additionally, $\ee P_\ee$ are uniformly bounded in $L^2(0,T;\lbar^p_0(\Omega))$, for any $p\in [1,2)$, and they
    weakly converges to the pressure $P$ of \eqref{eq:Darcy}, as $\ee \to 0$.
    \item $\tilde d_\ee$ weakly-$\ast$ converges in $L^\infty(\mathbb (0,T)\times  \Omega)$ to the unique weak solution $d$ of \eqref{eq:d}.
\end{enumerate}
\end{theorem}
\noindent 
Some remarks about this statement and the overall result are in order.
\begin{remark}\label{rmk:d-in-is-also-in-H1}
    Given that the $H^1(\Omega_\ee)$-norm of the initial data $d_{\ee, \rm in}$ is uniformly bounded for all $\ee>0$, we can apply Lemma 2.3 from \cite{MR1225439}. From this, the weak-$\ast$ convergence $\tilde d_{\ee, \rm in} \rightharpoonupstar d_{\rm in}$ in $L^\infty(\Omega)$ leads to the fact that also $d_{\rm in}$ is in $H^1(\Omega)$. Therefore, System \eqref{eq:d} possesses a unique weak solution $d$, as characterized in \eqref{function-spaces-for-u-and-d-at-the-limit}.
\end{remark}

\begin{remark}\label{rmk:uniform-estimates-norms-solutinos}
    Given the assumptions on $d_{\ee,\rm in}$, $F_\ee$ and $H_\ee$ in Theorem \ref{main_thm}, and taking into account the energy inequality from \eqref{energy-eq}, we emphasize the existence of a  constant $\EE_0 >0$, independent of $\ee > 0$, such that 
    \begin{equation*}
    \begin{aligned}
        \| \nabla u_\ee \|_{L^2((0,T)\times \Omega_\ee)}  
        +
        \|
            \nabla
            d_\ee    
        \|_{L^\infty(0,T;L^2(\Omega_\ee))}
        &+
        \|
            \Delta
            d_\ee    
        \|_{L^2((0,T)\times \Omega_\ee)} +
        \\
        &+
        \ee 
        \| F_\ee \|_{L^2((0,T)\times \Omega)}
        +
        \|
            \nabla
            H_\ee    
        \|_{L^2((0,T)\times \Omega)}
        \leq 
        \EE_0.
    \end{aligned}
    \end{equation*}
    We will use repeatedly  this estimate throghout our analysis.  Moreover, the equation \eqref{maineq3} adheres to a weak maximum principle. Coupled with \eqref{in-cdt}, this guarantees  $\|d_\ee \|_{L^\infty((0,T)\times \Omega_\ee)}\leq 1 $, for any $\ee>0$.
\end{remark}

\begin{remark}\label{rmk:pressure}
    For a fixed $\ee>0$, a more popular form of the equation \eqref{maineq1} is rooted on a variation of the energy \eqref{energy-eq}, which recasts the Stokes equation into
    \begin{equation}\label{eq:eq-u-with-rhs-in-h-1}
        -\Delta  u_\ee + \nabla \tilde \pre_\ee = - \Div(\nabla d_\ee \odot \nabla d_\ee )+  F_\ee + \Delta H_\ee
        \quad (0,T) \times \Omega_\ee.
    \end{equation}
    Using a conventional Galerkin method, indeed, \eqref{eq:eq-u-with-rhs-in-h-1} together with \eqref{maineq2}-\eqref{maineq3} admits a unique weak solution $(u_\ee,d_\ee)$ as in \eqref{fct-space:weak-solutions} satisfying the energy inequality \eqref{energy-eq}. Morever, given that $d_\ee\in L^2(0,T;H^2(\Omega_\ee))\hookrightarrow  L^2(0,T;W^{1,4}(\Omega_\ee))$, the right-hand side~of  \eqref{eq:eq-u-with-rhs-in-h-1} is in $L^2(0,T;H^{-1}(\Omega_\ee))$, a fact that allows to select a pressure $\tilde \pre_\ee$ in  $L^2(0,T;  \lbar^2_0(\Omega_\ee))$. 

    \noindent
    Additionally, $(u_\ee,d_\ee)$ is also solution of System \eqref{maineq1}-\eqref{maineq2}, where the pressure $\pre_\ee$ is given by 
    \begin{equation}\label{pressure-2}
         \pre_\ee(t,x)  =\tilde\pre_\ee(t,x) + \frac{|\nabla d_\ee(t,x)|^2}{2} -
        \fint_{\Omega_\ee} \frac{|\nabla d_\ee(t,y)|^2}{2}dy,\quad \text{a.e}\quad (t,x)\in (0,T)\times \Omega_\ee.
    \end{equation} 
    This implies in particular that also the pressure $\pre_\ee$ in \eqref{maineq1} belongs to 
    $ L^2(0,T;  \lbar^2_0(\Omega_\ee))$. 
    Our primary concern, however, is not the regularity of $\pre_\ee$ (or $\tilde \pre_\ee$) for a fixed $\ee>0$, but about obtaining some uniform estimates as $\ee\to 0$. Equation \eqref{maineq1} is particularly advantageous in this regard: the forcing term $(\nabla d_\ee)^T\Delta d_\ee$ can be uniformly bounded in $L^2(0,T; L^1(\Omega_\ee))$, stemming directly from the energy estimate \eqref{energy-eq}. In \Cref{sec:pressure}, we will transfer these estimates to a suitable extension of the pressure in $L^2(0,T;\lbar^p_0(\Omega))$, with $p \in [1,2)$. However, achieving uniform bounds for  $\Div(\nabla d_\ee \odot \nabla d_\ee )$ in $L^2(0,T;H^{-1}(\Omega_\ee))$ would require an explicit expression of the constant of embedding $H^2(\Omega_\ee)\hookrightarrow W^{1,4}(\Omega_\ee)$, contingent on $\ee>0$. This becomes increasingly challenging as the domain $\Omega_\ee$ is punctuated by a rising number of obstacles. Consequently, despite the equivalence of \eqref{eq:eq-u-with-rhs-in-h-1} and \eqref{maineq1}, the estimates associated with \eqref{maineq1} are better suited for our objective.
\end{remark}

\begin{remark}
In the analysis of the momentum equation (\ref{maineq1}), our focus is on two distinct types of forcing terms: $F_\ee$ and $\Delta H_\ee$. Both of these terms play a crucial role in the asymptotic behavior as $\ee \to 0$ in the context of the Darcy equation (\ref{eq:Darcy}). However, their influence in this limit is governed by different assumptions. Specifically, $F_\ee$ requires rescaling by $1/\ee$ in its norm to maintain significance in the limit, as failing to do so would lead to its diminishing to zero. In other words, this scaling is necessary to prevent the colloids from overpowering the forcing term. In contrast, although $\Delta H_\ee$ is less physical, $H_\ee$ simply needs to remain uniformly bounded in the space $L^2(0,T;H^1_0(\Omega_\ee))$. This requirement aligns with the behavior of $\Delta H_\ee$ and $\Delta u_\ee$, which exhibit similar characteristics, thus suggesting a parallel between $u_\ee/\ee $ and $H_\ee/\ee$.

\noindent 
Additionally, it would be plausible to consider a third type of forcing term, represented as $\Div\, \mathbb T_\ee$, with $2\times 2$ tensor function  $ \mathbb T_\ee$ uniformly bounded in $L^2((0,T)\times \Omega)$. This term would weakly converge to a tensor function $\mathbb T$ (up to a subsequence). Nevertheless, the corresponding term in Equation \eqref{eq:Darcy} would eventually manifest a linear dependency on $\mathbb T$, as well as on the tensor $\int_{Y^*}\nabla_y \omega^i(y)dy$. Given the null Dirichlet conditions of $\omega^i$,$\int_{Y^*}\nabla_y \omega^i(y)dy$ is indetically null. 
Therefore, a forcing term in the form of $\Div\, \mathbb T_\ee$ would vanish in the asymptotic limit $\ee \to 0$. To maintain brevity and focus, we henceforth exclude this term from further discussion.
\end{remark}

\section{A preliminary toolbox of homogenisation}\label{sec:preliminary-results-homogenisation}

\noindent 
Before delving into the core elements of the proof for \Cref{main_thm}, we first revisit some foundational tools commonly used in problems associated with homogenisation. One pivotal tool is the concept of a restriction operator from Sobolev spaces $W^{1,q}(\Omega)$ onto $W^{1,q}(\Omega_\ee)$, with $q\in (1,\infty)$, which was elaborated upon by Mikeli\'{c} (cf.~Lemma 2.2 of \cite{MR1131849} and Lemma 3.2 of \cite{MR1434319}). Notably, this concept broadens the restriction operators initially presented by Tartar \cite{1571980074146392576} and Allaire \cite{MR1020348} for the special case $q=2$.
\begin{lemma}\label{lemma:restriction_operator}
        For any $q\in (1,\infty)$, there exists an operator $\mathcal R_{q,\ee}\in \mycal{L}(W^{1,q}_0(\Omega, \mathbb R^2),W^{1,q}_0(\Omega_\ee, \mathbb R^2))$, such that
	\begin{itemize}
		\item 	for any $w \in W^{1,q}_0(\Omega, \mathbb R^2)$, $\mathcal{R}_{q,\ee} w \equiv 0$ in $\Omega\setminus \Omega_\ee$, 
		\item 	for any $w \in W^{1,q}_0(\Omega, \mathbb R^2)$ with $w\equiv 0$ in $\Omega\setminus \Omega_\ee$, $\mathcal{R}_{q,\ee} w = w$, 
		\item 	$\Div (\mathcal{R}_{q,\ee} w) = 0$, for any $w \in W^{1,q}_0(\Omega,  \mathbb R^2)$ with $\Div\,w = 0$,
		\item 	the following inequalities are satisfied
		\begin{equation*}
		\begin{aligned}
		  \| R_{q,\ee} w \|_{L^q(\Omega_\ee)}&\leq 
            C_{q,\Omega}\Big( \| w \|_{L^q(\Omega)} + \ee \| \nabla w \|_{L^q(\Omega)}\Big) ,\\
			\| \nabla R_{q,\ee} w \|_{L^q(\Omega_\ee)}&\leq C_{q,\Omega}
            \Big( \frac{1}{\ee} \| w \|_{L^q(\Omega)} + \| \nabla w \|_{L^q(\Omega)}\Big),
		\end{aligned}	
		\end{equation*}
		for a suitable constant $C_{q,\Omega}>0$ that depends only on the domain $\Omega$ and $q\in (1,\infty)$.
	\end{itemize}
\end{lemma}
\noindent 
As in the studies by Allaire \cite{MR1079190,MR1079189}, this lemma will be instrumental in implicitly defining an extension $\nabla P_\ee$ of  $\nabla \pre_\ee$ across the entire domain $\Omega$ (cf.~\Cref{prop:pressure}). 
\begin{remark}\label{rmk:restriction-operators-coincide}
    Despite the abstract formulation of \Cref{lemma:restriction_operator}, the construction of the operators $\mathcal{R}_{q, \ee}$ is rather ``mechanical''. As an example, one might consider the unit cell $Y=[-1/2,1/2]^2$ and a sufficiently large smooth open set $\mathcal{B}  \subset Y$ that acts as a ``control volume'', with $\overline{\mathcal{T}} \subset \mathcal{B}$. For any function $\tilde \omega \in W^{1,q}(Y)^2$, we define $\mathcal{R}(\tilde \omega) =\tilde \omega$ outside the region $Y^*\setminus \mathcal{B}$. Inside the region $\mathcal{B}\setminus \mathcal{T}$, however, $\mathcal{R}(\tilde \omega)$ solves the Stokes equation given by:
    \begin{equation*}
    \begin{cases}
        -\Delta_y \mathcal{R}(\tilde \omega)+ \nabla_y Q = -\Delta \tilde \omega \quad 
        &\mathcal{B}\setminus \mathcal{T},\\
        \Div_y \mathcal{R}(\tilde \omega) =\Div_y \tilde \omega + \frac{1}{|\mathcal{C}|}\int_{\mathcal{T}} \Div_y \tilde \omega
        \quad 
        &\mathcal{B}\setminus \mathcal{T},\\
         \mathcal{R}(\tilde \omega) = \tilde \omega\quad 
        &\partial \mathcal{B},\\
         \mathcal{R}(\tilde \omega) = 0\quad 
        &\partial \mathcal{T}.
    \end{cases}
    \end{equation*}    
    Consequently, each operator $\mathcal{R}_{q, \ee}$ is constructed by taking a function $\omega \in W^{1,q}(\Omega)^2$, selecting a cell $ Y^{\ee}_k = x_k + \ee Y $, and then applying $\mathcal{R}$ to the re-scaled function $ \tilde \omega(y) = \omega(x_k+\ee y)$. The resulting function is then re-scaled back to its original variables  $ x \in  Y^{*,\ee}_k $. Further details are referred to Section 2.2 of \cite{MR1079189} for the case $q = 2$ and Remark 2.1 and Lemma 2.2 in \cite{MR1131849} for more general $q \in (1, \infty)$. Notably, we can assume that for any $1<q<q'<\infty$, 
    $\mathcal{R}_{q,\ee}(\omega) =  \mathcal{R}_{q',\ee}(\omega)$ in $W^{1,q}_0(\Omega_\ee)$, for any $\omega \in W^{1,q'}(\Omega)\hookrightarrow W^{1,q}(\Omega)$. The constants $C_{q,\Omega}$, however, are certainly different for any $q\in(1,\infty)$ and become unbounded as $q\to 1$ or $q\to \infty$.
\end{remark}
\noindent 
Another essential instrument for our analysis is the following Poincar\'e inequality within \(W^{1,q}_0(\Omega_\ee)\) (as delineated in Lemma 2.3 of \cite{MR1131849}). This inequality leverages the vanishing of functions on the boundary of each obstacle, intertwining the Poincar\'e constant with the scaling parameter $\ee>0$.
\begin{lemma}\label{Poincare_ineq}
For any $q\in (1, \infty)$, there exists a constant $C_{q, \Omega}>0$ depending only on $q$ and $\Omega$, such that 
	\begin{equation*}
		\| f \|_{L^q(\Omega_\ee)} \leq C_{q, \Omega}\ee \| \nabla f \|_{L^q(\Omega_\ee)}.
	\end{equation*}
for all $\ee>0$ and any function $f$ in $W^{1,q}_0(\Omega_\ee)$.
\end{lemma}
\noindent 
We anticipate that we will predominantly use \Cref{Poincare_ineq} to ascertain the vanishing of $\tilde u_\ee $ (cf.~\eqref{eq:estimate-of-tildeu-ee-Poinc}) and the asymptotic limit of $\tilde u_\ee/\ee$  (cf.~\eqref{eq:estimate-of-tildeu-ee-Poinc} and further estimate within \Cref{sec:Darcy}), as it weakly converges to a solution of the Darcy's law. Additionally, we will employ in our analysis the following Poincar\'e-Wirtinger inequality in the rescaled cells $Y_{k}^{*, \ee}$. This homogenisation version entangles once more the constant of the inequality with the scale parameter $\ee>0$.
\begin{lemma}\label{lemma:Poincare-Wirtinger-inequality}
    For any $1\leq q< \infty$, there exists a positive constant $C_{q, Y^*}$ which depends only on $q$ and the unitary cell $Y^*$, such that
    \begin{equation*}
        \Big\| 
            f - \fint_{ Y^{*, \ee}_k } f
        \Big\|_{L^q(Y^{*, \ee}_k)}\leq \ee \,C_{q, Y^*}  \| \nabla f \|_{L^q(Y^{*, \ee}_k)}, 
        \quad \text{with}\quad 
        \fint_{ Y^{*, \ee}_k } f = \frac{1}{|Y^{*, \ee}_k|}\int_{ Y^{*, \ee}_k} f(x) dx,
    \end{equation*}
     for any $\ee>0$, any cell $Y^{*, \ee}_k$ and for any function $f \in W^{1,q}(Y^{*, \ee}_k)$. 
\end{lemma}
\noindent 
As an aside for $q=2$ we refer to Lemma 2.2 in \cite{MR1225439}. For general $q \in (1,\infty)$ the proof involves rescaling the function $f$ to a function $g $ defined in the unitary cell $Y^*$ by $g(y) = f(x_k^\epsilon + \epsilon y)$. This is followed by applying the classical Poincar\'e-Wirtinger inequality in the unit cell $Y^*$. Rescaling back to the original variable $x \in Y^{*, \epsilon}_k$ then yields the desired result.

\smallskip \noindent 
Finally, we anticipate that we will establish a suitable extension of the pressure $\pre_\ee$ and several associated estimates using an argument rooted in the so-called Bogovskii's approach  (cf.~\Cref{prop:pressure} and \eqref{def:Fg}). We state the exsitence of the Bogovskii operator with the following lemma.
\begin{lemma}\label{lemma:Bogovskii}
    Let $\Omega\subset \mathbb R^n$ be a bounded domain in $\mathbb R^n$ with smooth boundary, $n\geq 2$. Given an average-free function $f$ in $\lbar^q_0(\Omega)$, with $q\in (1,\infty)$, there exists a function 
    $\psi_f \in W^{1,q}_0(\Omega)$ such that $\Div\, \psi_f = f$ and furthermore
    \begin{equation*}
        \| \nabla \psi_f \|_{L^q(\Omega)} \leq C_{q,\Omega} \| f \|_{L^q(\Omega)},
    \end{equation*}
    for a suitable constant $C_{q,\Omega}$ that depends only on $q$ and $\Omega$.
\end{lemma}
\noindent 
We refer to Theorem III.3.1 of \cite{Galdi2011} for the proof of this lemma, where the author considered the more general case of a domain that is the union of star-shaped sets (which is beyond the scope of this paper). Additionally, without loss of generality, we may assume that the constants denoted by $C_{q, \Omega}$ all coincide in the previous lemmas.

\noindent 
Having introduced the foundational tools and key lemmas crucial for our homogenisation analysis, we are now in a prime position to commence the proof of our main theorem. 

\section{Aubin-Lions Lemma Revisited: a Homogenisation Extension}\label{sec:Aubing-Lions}

\noindent 
A central result in our analysis is an adapted version of the well-known Aubin-Lions Lemma, tailored for functions on the  domain $(0,T)\times \Omega_\ee$, as $\ee \to 0$.  Specifically, we analyse a sequence of functions whose derivatives are uniformly bounded within this domain in $\ee>0$. This allows us to extract a subsequence, whose extension exhibit a property similar to ``strong convergence'' in the larger domain $(0,T)\times \Omega$.  

\noindent 
To articulate this result, we first recall that $\theta$ is the ratio $|Y^*|/|Y|$ and, for any function $f: (0,T)\times \Omega_\ee \rightarrow \mathbb{R}$,  $\tilde{f}$ denotes the zero-extension of $f$ in $(0,T)\times \Omega$. 
Additionally, we consider a sequence of positive values $(\ee_l)_{l \in \mathbb{N}}$, which monotonically converges to zero, and we simplify our notation with $\Omega_l := \Omega_{\ee_l}$. Our lemma is then formally stated as follows.
\begin{theorem}\label{lemma:compactness_Omega_eps}     
Let $q,r\in (1,\infty)$, $s \in [1,q]$ and $(f_l)_{l\in \NN}$ be a sequence of functions in
\begin{equation*}
    f_l \in
    L^\infty(0,T;L^q(\Omega_l)) 
    \quad 
    \text{with}
    \quad
    \nabla f_l \in 
    L^\infty(0,T;L^s(\Omega_l)) 
    \quad
    \text{and}
    \quad
    \partial_t f_l \in L^r(0,T;L^s(\Omega_l)),
\end{equation*}
the norms of which are uniformly bounded: there exists a constant $\mathbf{C}>0$ such that
\begin{equation}\label{est:Thm-Aubin-Lions-uniform-estimates}
    \| f_l              \|_{L^\infty(0,T;L^q(\Omega_{l}))} + 
    \| \nabla f_l       \|_{L^\infty(0,T;L^s(\Omega_{l}))} +
    \| \partial_t f_l   \|_{L^r(0,T;L^s(\Omega_l))}
    \leq \mathbf{C},\quad \text{for any}\quad  l \in \NN.
\end{equation}
There exist a subsequence $(f_{l_j})_{j\in \mathbb N}$ and a function $f \in L^\infty(0,T;L^q(\Omega))$ such that 
\begin{equation*}
\tilde f_{l_j}\stackrel{\ast}{\rightharpoonup} \theta f \quad \text{in}\quad  L^\infty(0,T;L^q(\Omega))
\end{equation*} 
and moreover the following property is satisfied: for any sequence $(g_j)_{j\in \mathbb N}$ 
with $g_j \in L^1(0,T; L^{\alpha}(\Omega_{l_j}))$, for some $\alpha \in (q', \infty]$, satisfying $\tilde g_j \rightharpoonup \theta g $ weakly in $L^1(0,T; L^{\alpha }(\Omega))$ (or weakly-$\ast$, if $\alpha = \infty$)
\begin{equation}\label{eq:lemma-compactness}
        \lim_{j \to \infty}
        \int_0^T 
        \int_{\Omega_{l_j}}
            \phi 
            f_{l_j}  
            g_j 
        = 
        \theta 
        \int_0^T 
        \int_{\Omega}
        \phi
        f  g
\end{equation}
for any $\phi \in \mathcal{C}^\infty_c([0,T]\times \Omega)$. Finally, if $s>1$, then $\nabla f\in L^\infty(0,T;L^s(\Omega))$ and  $\partial_t f \in L^r(0,T;L^s(\Omega))$.
\end{theorem} 
\begin{remark}\label{rmk:Aubin-Lions}
    The classical Aubin-Lions Lemma directly implies \Cref{lemma:compactness_Omega_eps} if each function $f_l$ satisfies the uniform estimates \eqref{est:Thm-Aubin-Lions-uniform-estimates} in the domain $(0,T)\times \Omega$. This implication is reinforced by the Rellich-Kondrachov Theorem, which asserts the compact embedding $W^{1,s}(\Omega)\ssubset L^s(\Omega)$. Consequently, Aubin-Lions ensures precompactness of the sequence $(f_{l})_{l \in \mathbb N}$ in $\mathcal{C}([0,T], L^s(\Omega))$. Furthermore, the uniform bounds in $L^\infty(0,T;L^q(\Omega))$ indicate its precompactness in $L^\infty(0,T; L^{\alpha'}(\Omega))$, for any $\alpha' \in [1,q)$. The novel aspect of \eqref{lemma:compactness_Omega_eps} lies in the requirement of uniform estimates solely within the domain $(0,T)\times \Omega_l$. This extends the work of Allaire and Murat (cf.~Lemma~2.3 in \cite{MR1225439}), where they provided a revised version of the Rellich-Kondrachov theorem for functions in $H^1(\Omega_l)$.

\end{remark}
    \noindent 
    Due to the technical nature of \Cref{lemma:compactness_Omega_eps}, readers primarily interested in the main system \eqref{maineq1}-\eqref{maineq3} may proceed directly to Section \ref{sec:pressure}, while this section comprehensively addresses the proof. Moreover, we first state a refined result under the additional assumption that the sequence $(f_l)_{l\in \mathbb N}$ is uniformly bounded in $L^\infty((0,T)\times \Omega_l)$.
\begin{prop}\label{thm:Aubin-Lions-Linfty}
     Assume that the hypotheses of \Cref{lemma:compactness_Omega_eps} are satisfied and 
     suppose additionally that $(f_l)_{l \in \mathbb N}$ is uniformly bounded in $L^\infty((0,T)\times \Omega_l)$. 
     Denote by $(f_{l_j})_{j\in \mathbb N}$ and $f$ the same subsequence and limit function as in \Cref{lemma:compactness_Omega_eps}. 
     Then $f$ belongs also to $L^\infty((0,T)\times \Omega)$ and for any $(g_j)_{j\in \mathbb N}$ with $g_j \in L^1(0,T;L^\alpha(\Omega_{l_j}))$ for some $\alpha \in (1, \infty]$, satisfying $\tilde g_j \rightharpoonup \theta g $ weakly in $L^1(0,T; L^{\alpha }(\Omega))$ (or weakly-$\ast$, if $\alpha = \infty$)
    \begin{equation}\label{eq:lemma-compactness-non-linear}
        \lim_{j\to \infty}
        \int_0^T 
        \int_{\Omega_{l_j}}
        \phi 
        f_{l_j}^m
        g_j  
        = 
        \theta 
        \int_0^T 
        \int_{\Omega}
        \phi
        f^m g
    \end{equation}
    for any $m\in \mathbb N$ and any $\phi \in \mathcal{C}^\infty_c([0,T]\times \Omega)$.
\end{prop}
\begin{proof}[Proof of \Cref{lemma:compactness_Omega_eps} and \Cref{thm:Aubin-Lions-Linfty}]

From the assumptions each $f_l$ belongs also to $W^{1,r}(0,T; L^s(\Omega_l))$, thus we can consider its representative in $\mathcal{C}([0,T], L^s(\Omega_l))$ without loss of generality. Moreover, from the estimates in \eqref{est:Thm-Aubin-Lions-uniform-estimates}, the sequence of extensions $(\tilde f_l)_{l\in \mathbb N}$ is uniformly bounded in $L^\infty(0,T;L^q(\Omega))$; therefore there exist a subsequence $(\tilde f_{l_j})_{j\in \mathbb N}$ and a function $f \in L^\infty(0,T;L^q(\Omega))$ such that $\tilde f_{l_j}\stackrel{\ast}{\rightharpoonup} \theta f$ weakly-$\ast$ in $L^\infty(0,T;L^q(\Omega))$. 

\smallskip 
\noindent 
We organise our proof into five major steps, summarized as follows:
\begin{enumerate}[(a)]
    \item We introduce a new sequence of functions $(\bar{f}_{l_j})_{j \in \mathbb N}$, which averages each function \(f_{l_j}\) within each cell  (cf.~\eqref{eq:def:barfl}). Under the claim that this sequence is precompact in an appropriate function space, we extract a subsequence that converges strongly to a function $\bar{f}$ (potentially different from $f$). The proof of the precompactness is deferred to Part (d).
    \item As \(j \to \infty\), we show that a suitable norm of \(f_{l_j} - \bar{f}_{l_j}\) approaches zero, allowing us to replace \(f_{l_j}\) with \(\bar{f}_{l_j}\) in \eqref{eq:lemma-compactness}. We then establish that at the limit, \(\bar{f}\) and \(f\) actually coincide, leading to our main identity \eqref{eq:lemma-compactness}.
    \item Next, we undertake the proof of \Cref{thm:Aubin-Lions-Linfty} and explore the identity \eqref{eq:lemma-compactness-non-linear}.
    \item We revisit the claim from part (a) to validate the precompactness of the sequence \((\bar{f}_{l_j})\).
    \item We conclude by showing that $\nabla f\in L^\infty(0,T;L^s(\Omega)) $ and $\partial_t f \in L^r(0,T;L^s(\Omega))$ if $s>1$.
\end{enumerate}
{\it \underline{Part~(a)}:} We recall that $\mathcal{K} = \mathcal{K}_{\ee_l}  \subset \mathbb{Z}^2$ represents the set of indices $k \in \mathbb{Z}^2$ for which $Y_k^{*,\ee_l} \subset \Omega$. 
For any $l\in \mathbb N$, we introduce the following function $\bar{f}_{l} \in \mathcal{C}([0,T], L^s(\Omega))$, which is piecewise constant at any time $t \in [0,T]$:
\begin{equation}\label{eq:def:barfl}
    \bar{f}_l(t,x):= 
    \sum_{k \in \mathcal{K}}
    \mathbbm{1}_{Y_k^{\ee_l}}(x)
    \bigg(
        \fint_{Y_k^{*,\ee_l}}
        f_l(t,\tilde x) d\tilde x
    \bigg),\qquad (t,x) \in [0,T]\times \Omega,
\end{equation}
where $\fint_{Y_k^{*,\ee_l}} f = \frac{1}{|Y_k^{*,\ee_l}|}\int_{Y_k^{*,\ee_l}} f$ denotes the average over the cell $Y_k^{*,\ee_l}$ (without the colloidal particle). The sequence $(\bar{f}_l)_{l\in \mathbb{N}}$ is well-defined in $(0,T)\times \Omega$ and it is uniformly bounded in $L^\infty(0,T;L^q(\Omega))$ (respectively $L^\infty((0,T)\times \Omega)$ under the assumptions of \Cref{thm:Aubin-Lions-Linfty}). Indeed
\begin{equation*}
\begin{aligned}
    \| \bar{f}_l \|_{L^\infty(0,T;L^q(\Omega))} 
    &= 
    \EssSup_{t \in (0,T)}
    \bigg(
    \int_\Omega 
        \bigg|   \sum_{k \in \mathcal{K}}
                 \mathbbm{1}_{Y_k^{\ee_l}}(x)
                \bigg(
                    \fint_{Y_k^{*,\ee_l}}
                    f_l(t,\tilde x) d\tilde x
                \bigg)
        \bigg|^q 
    dx
    \bigg)^\frac{1}{q}
    \\
    &= 
    \EssSup_{t \in (0,T)}
    \bigg(
    \sum_{k \in \mathcal{K}}
    | Y_k^{\ee_l}| 
    \bigg|   
       \frac{1}{|Y_k^{*,\ee_l}|}
       \int_{Y_k^{*,\ee_l}}
       f_l(t,\tilde x) d\tilde x
    \bigg|^q
    \bigg)^\frac{1}{q}.
\end{aligned}
\end{equation*}
Hence, applying the H\"older's inequality, we obtain
\begin{equation}\label{est:barfl-Lq}
\begin{aligned}
    \| \bar{f}_l \|_{L^\infty(0,T;L^q(\Omega))}&\leq  
    \EssSup_{t \in (0,T)}
    \bigg(
    \sum_{k \in \mathcal{K}}
       \frac{| Y_k^{\ee_l}| }{|Y_k^{*,\ee_l}|}
       \int_{Y_k^{*,\ee_l}}
       |f_l(t,\tilde x)|^q d\tilde x
    \bigg)^\frac{1}{q}
    =
    \EssSup_{t \in (0,T)}
    \bigg(
        \sum_{k \in \mathcal{K}}
       \frac{1}{\theta }
       \int_{Y_k^{*,\ee_l}}
       |f_l(t,\tilde x)|^q d\tilde x
    \bigg)^\frac{1}{q}
    \\
    & \leq 
    \EssSup_{t \in (0,T)}
    \bigg(
       \frac{1}{\theta }
       \int_{\Omega_l}
       |f_l(t,\tilde x)|^q d\tilde x
    \bigg)^\frac{1}{q}
    =
    \frac{1}{\theta^\frac{1}{q}} 
    \| f_l \|_{L^\infty(0,T;L^q(\Omega_l))}
    \leq 
     \frac{1}{\theta^\frac{1}{q}} 
    \mathbf C < \infty.
\end{aligned}
\end{equation}
Consequently, up to a subsequence, $(\bar{f}_{l_j})_{j \in \mathbb N}$ converges weakly-$\ast$ to a function $\bar{f} \in L^\infty(0,T; L^q(\Omega))$. Moreover, if the assumptions of \Cref{thm:Aubin-Lions-Linfty}) are satisfied, $(\bar{f}_{l_j})_{j \in \mathbb N}$   converges weakly-$\ast$ to $\bar f$ also in $L^\infty((0,T)\times \Omega)$.

\smallskip 
\noindent Next, we assert that for any compactly contained domain $\omega\ssubset \Omega$, the sequence $(\bar{f}_{l}|_{\omega})_{l\in \mathbb{N}}$ is precompact in $\mathcal{C}([0,T], L^s(\omega))$, for the chosen $s \in [1,q]$.  
The detailed proof of this assertion will be addressed in Part (d) as we now shift our focus to the remaining aspects of our main proof. Owing to this precompactness, we deduce that a representative of $\bar{f}|_{\omega}$ belongs to $\mathcal{C}([0,T], L^s(\omega))$, for any $\omega \ssubset \Omega$ and that, up to a subsequence, 
\begin{equation*}
    \bar{f}_{l_j}|_{\omega} \rightarrow \bar{f}|_{\omega}\quad 
    \text{strongly in}\quad \mathcal{C}([0,T], L^s(\omega)), \quad 
    \text{for any }\omega\ssubset \Omega.
\end{equation*}
Importantly, this subsequence can be chosen independently of $\omega$, given that the domain $\Omega$ can always be  written as the union of an increasing sequence of domains, each compactly contained within $\Omega$. Moreover, since $s \in [1, q)$, utilizing standard interpolations, $\bar{f}_{l_j} |_{\omega}$ strongly converges to $\bar{f} |_{\omega}$ in $L^\infty(0,T; L^{\alpha'}(\omega))$ for any $\alpha' \in [1, q)$ (or any $\alpha' \in [1, \infty)$, under the assumptions of \Cref{thm:Aubin-Lions-Linfty}).

\medskip 
\noindent 
{\it \underline{Part~(b)}:}  For a given $\phi \in \mathcal{C}^\infty_c([0,T]\times \Omega)$, we select a smooth domain $\omega \ssubset \Omega$ such that $\supp(\phi) \subset [0,T] \times \omega$. Moreover, for simplicity, we introduce the notation 
\begin{align*}
    \mathcal{C}_l   := \bigcup_{k \in \mathcal{K}} Y^{\ee_l}_k
    \qquad 
    \mathcal{C}_l^* := \bigcup_{k \in \mathcal{K}} Y^{*,\ee_l}_k,
\end{align*}
as the union of all cells (with and without colloids, respectively) that are included in $\Omega$. For a sufficiently large $L \in \mathbb{N}$ (ensuring $\ee_L$ is small enough), we have that $\omega  \subseteq \mathcal{C}_l \subseteq \Omega$, for all $l \in \mathbb N$ with $l \geq L$.

\noindent 
Next, we consider only $j\in \mathbb N$ such that $l_j \geq L$ and we proceed by dividing the integral on the left-hand side of \eqref{eq:lemma-compactness} into two major components:
\begin{equation}\label{eq:proof-lemma-spliting-int-fg}
        \int_0^T 
        \int_{\Omega_{l_j}}
        \phi 
        f_{l_{j}}
        g_{j}
        =
        \int_0^T 
        \int_{\Omega_{l_j}}
        \phi 
        (f_{l_{j}}-\bar{f}_{l_j})
        g_{j}
        +
        \int_0^T 
        \int_{\Omega_{l_j}}
        \phi 
        \bar{f}_{l_j}
        g_{j}
        =
        \int_0^T 
        \int_{\mathcal{C}_{l_j}^*}
        \phi 
        \big(
        f_{l_{j}} 
        -
        \bar{f}_{l_{j}}
        \big)
        g_j 
        +
        \int_0^T 
        \int_{\omega}
        \phi 
        \bar{f}_{l_{j}}
        \tilde 
        g_{j}.
\end{equation}
By assumption, $\tilde{g}_j \rightharpoonup \theta g$ weakly in $L^{1}(0,T; L^\alpha(\Omega))$ for a $\alpha \in (q', \infty]$ (weakly-$\ast$ if $\alpha = \infty$), while   $f_{l_j} |\omega \to \bar{f}|_{\omega}$ strongly in $L^\infty(0,T;L^{\alpha'}(\omega))$ from Part (a). Hence, we are in the condition to pass to the limit in the last integral of the right-hand side of \eqref{eq:proof-lemma-spliting-int-fg}. This transition yields:
\begin{equation}\label{eq:I-have-no-more-names}
    \lim_{j \to \infty }
    \int_0^T 
        \int_{\omega}
        \phi 
        \bar{f}_{l_{j}}
        \tilde 
        g_{j}
    = 
    \theta
    \int_0^T 
    \int_{\omega}
    \phi 
    \bar{f} 
    g
    = 
    \theta
    \int_0^T 
    \int_{\Omega}
    \phi 
    \bar{f} 
    g.
\end{equation}
Next, we show that the first integral on the right-hand side of \eqref{eq:proof-lemma-spliting-int-fg} converges towards $0$ as $j \to \infty$. 
Indeed, 
\begin{equation}\label{est:no-idea-how-to-call-it}
\begin{aligned}
    \bigg|
        \int_0^T 
        \int_{\mathcal{C}_{l_j}^*}
        \phi 
        \big(
        f_{l_{j}} 
        -
        \bar{f}_{l_{j}}
        \big)
        \tilde 
        g_j
    \bigg|
    &\leq 
    \| \phi \|_{L^\infty((0,T)\times \mathcal{C}_{l_j}^*)}
    \big\|
        f_{l_{j}} 
        -
        \bar{f}_{l_{j}}
    \big\|_{L^\infty(0,T ;L^{\alpha'}(\mathcal{C}_{l_j}^*))}
    \| \tilde 
        g_j 
    \|_{L^{1}(0,T;L^\alpha(\mathcal{C}_{l_j}^*))} \\
    &\leq 
    \| \phi \|_{L^\infty((0,T)\times \Omega)}
    \big\|
        f_{l_{j}} 
        -
        \bar{f}_{l_{j}}
    \big\|_{L^\infty(0,T ;L^{\alpha'}(\mathcal{C}_{l_j}^*))}
    \| \tilde 
        g_j 
    \|_{L^{1}(0,T;L^\alpha(\Omega))}.
\end{aligned}    
\end{equation}
To demonstrate that the right-hand side converges to $0$, it suffices to establish that $\|f_{l_j} - \bar{f}_{l_j}\|_{L^\infty(0,T; L^{\alpha'}(\mathcal{C}_{l_j}^*))}\to 0$, as $j \to \infty$. Recalling that $\alpha \in (q', \infty]\Rightarrow \alpha' \in [1, q)$, we apply a standard interpolation inequality between $L^1(\mathcal{C}_{l_j}^*)$ and $L^q(\mathcal{C}_{l_j}^*)$ to gather
\begin{equation*}
\begin{aligned}
    \|f_{l_j} &- \bar{f}_{l_j}\|_{L^\infty(0,T; L^{\alpha'}(\mathcal{C}_{l_j}^*))}
    \leq 
    \big\|
        f_{l_{j}} 
        -
        \bar{f}_{l_{j}}
    \big\|_{L^\infty(0,T ;L^1(\mathcal{C}_{l_j}^*))}^{1-\frac{q'}{\alpha}}
    \big\|
        f_{l_{j}} 
        -
        \bar{f}_{l_{j}}
    \big\|_{L^\infty(0,T ;L^q(\mathcal{C}_{l_j}^*))}^{\frac{q'}{\alpha}}
   \\
   &\leq 
    \Big( 
        |\mathcal{C}_{l_j}^*|^{1-\frac{1}{s}}
        \big\|
            f_{l_{j}} 
            -
            \bar{f}_{l_{j}}
        \big\|_{L^\infty(0,T ;L^s(\mathcal{C}_{l_j}^*))}
    \Big)^{1-\frac{q'}{\alpha}}
    \Big(
    \big\|
        f_{l_{j}} 
    \big\|_{L^\infty(0,T ;L^q(\mathcal{C}_{l_j}^*))}
    +
    \big\|
        \bar{f}_{l_{j}}
    \big\|_{L^\infty(0,T ;L^q(\Omega))}
    \Big)^{\frac{q'}{\alpha}}
\end{aligned}
\end{equation*}
Invoking the uniform estimates of \eqref{est:Thm-Aubin-Lions-uniform-estimates} and \eqref{est:barfl-Lq}, we obtain
\begin{equation}\label{ineq:lemma-second-term-vanishing-1}
\begin{aligned}    
   \|f_{l_j} - \bar{f}_{l_j}\|_{L^\infty(0,T; L^{\alpha'}(\mathcal{C}_{l_j}^*))}
    &\leq 
    \bigg(
        1+\frac{1}{\theta^\frac{1}{q}}
    \bigg)^{\frac{q'}{\alpha}}
    \mathbf C^{\frac{q'}{\alpha}}
    \big| \Omega \big|^{(1-\frac{1}{s} )(1-\frac{q'}{\alpha})}
    \big\|
        f_{l_{j}} 
        -
        \bar{f}_{l_{j}}
    \big\|_{L^\infty(0,T ;L^s(\mathcal{C}_{l_j}^*))}^{1-\frac{q'}{\alpha}}.
\end{aligned}    
\end{equation}
 Since $1-q'/\alpha >0$, it is enough to show that $\|f_{l_j} - \bar{f}_{l_j}\|_{L^\infty(0,T; L^s(\mathcal{C}_{l_j}^*))}\to 0$, as $j \to \infty$. To this end
\begin{align*}
    \| f_{l_j}- \bar{f}_{l_j} \|_{L^\infty(0,T;L^s(\mathcal C_{l_j}^*))}
    &
    =
    \EssSup_{t \in (0,T)}
    \bigg(
    \int_{\mathcal C_{l_j}^*} 
    \Big|
         f_{l_j}(t,x)
         -
        \sum_{k \in \mathcal K}
        \mathbbm{1}_{Y_k^{\ee_{l_j}}}(x)
        \fint_{Y_k^{*,\ee_l}}
         f_{l_j}(t,\tilde x) 
         d\tilde x
    \Big|^s 
    dx
    \bigg)^\frac{1}{s}.
\end{align*}
Hence, observing that $\mathbbm{1}_{Y_k^{\ee_{l_j}}}(x) = \mathbbm{1}_{Y_k^{\ee_{l_j},*}}(x)$ for any $x \in C_{l_j}^*$, we can recast the last term into
\begin{align*}
     \| f_{l_j}- \bar{f}_{l_j} \|_{L^\infty(0,T;L^s(\mathcal C_{l_j}^*))}
    &
    =
    \EssSup_{t \in (0,T)}
    \bigg(
    \sum_{k \in \mathcal K}
    \int_{Y_k^{*,\ee_{l_j}}}
    \Big|
        f_{l_j}(t,x)
         -
        \fint_{Y_k^{*,\ee_l}}
         f_{l_j}(t,\tilde x) d\tilde x
    \Big|^s
    dx
    \bigg)^\frac{1}{s}.
\end{align*}
We hence invoke the Poincar\'e-Wirtinger inequality of \Cref{lemma:Poincare-Wirtinger-inequality} for functions in $W^{1,s}(Y^{*,\ee_{l_j}}_{k})$. There exists a constant $ C_{s,Y^*}>0$ depending only on $s$ and $Y^*$ such that
\begin{equation}\label{ineq:lemma-second-term-vanishing-2}
\begin{aligned}
     \| f_{l_j}- \bar{f}_{l_j} \|_{L^\infty(0,T;L^s(\mathcal C_{l_j}^*))}
    &
    \leq 
    \EssSup_{t \in (0,T)}
    \bigg(
    \ee^{s}_{l_j} C_{s,Y^*}^s 
    \sum_{k \in \mathcal K}
    \int_{Y_k^{*,\ee_l}}
    \big|
        \nabla f_{l_j}(t,x)
    \big|^s
    dx
    \bigg)^\frac{1}{s}
    \\
    &
    \leq  
    \ee_{l_j} C_{s,Y^*}  \|\nabla f_{l_j}  \|_{L^\infty(0,T;L^s(\Omega_{{l_j}}))}
    \leq 
    C_{s,Y^*} \mathbf C \,\ee_{l_j},
\end{aligned}
\end{equation}
thanks to the uniform estimates of \eqref{est:Thm-Aubin-Lions-uniform-estimates}. Since $\ee_{l_j}\to 0$, we deduce that $ \| f_{l_j}- \bar{f}_{l_j} \|_{L^\infty(0,T;L^s(\mathcal C_{l_j}^*))}\to  0$, as $j\to \infty$. Coupling this result with \eqref{est:no-idea-how-to-call-it} and \eqref{ineq:lemma-second-term-vanishing-1}, we obtain that 
\begin{equation*}
    \lim_{j\to \infty}\int_0^T 
        \int_{\mathcal{C}^*_{l_j}}
        \phi 
        \big(
        \tilde 
        f_{l_{j}} 
        -
        \bar{f}_{l_{j}}
        \big)
        \tilde 
        g_j = 0,
\end{equation*}
which together with \eqref{eq:proof-lemma-spliting-int-fg} and  \eqref{eq:I-have-no-more-names}  implies 
\begin{equation}\label{eq:lemma-conv-towards-fbar}
    \lim_{j\to \infty}
    \int_0^T 
    \int_{\Omega_{l_j}}
      \phi 
      f_{l_{j}}
      g_{j} 
    =
    \theta
    \int_0^T 
    \int_{\Omega}
    \phi 
    \bar f
    g
\end{equation}
for any $\phi \in \mathcal{C}^\infty_c([0,T]\times \Omega)$. To conclude the proof of \eqref{eq:lemma-compactness}, we need to show that, in reality, the functions $\bar f $ and $f$ coincide in $(0,T)\times \Omega$. 
To this end, we consider the constant functions $g_j :=1 \in L^\infty((0,T)\times \Omega_{l_j})$, for any $j\in \mathbb N$. Hence $\tilde g_j\equiv  \mathbbm{1}_{\Omega_{l_j}}\stackrel{\ast}{\rightharpoonup} \theta \cdot 1 =: \theta g $ in $L^\infty((0,T)\times  \Omega)$ and thanks to \eqref{eq:lemma-conv-towards-fbar}
\begin{equation*}
    \lim_{j\to \infty}
    \int_0^T 
    \int_{\Omega_{l_j}}
    \phi 
    f_{l_j}
    g_{j}
    =
    \theta
    \int_0^T 
    \int_{\Omega}
    \phi
    \bar f.
\end{equation*}
Simultaneously, we recall that $\tilde f_{l_j} \stackrel{\ast}{\rightharpoonup}  \theta f$ in $L^\infty(0,T;L^q(\Omega))$, therefore
\begin{equation*}
    \theta \int_0^T 
    \int_{\Omega}
    \phi
    f
    =
    \lim_{j \to \infty}
    \int_0^T 
    \int_{\Omega}
    \phi
    \tilde f_{l_j}
    =
    \lim_{j\to \infty}
    \int_0^T 
    \int_{\Omega_{l_j}}
    \phi 
    f_{l_j}
    g_{j}
    =
    \theta
    \int_0^T 
    \int_{\Omega}
    \phi
    \bar f.
\end{equation*}
From the arbitrariness of $\phi \in \mathcal{C}^\infty_c([0,T]\times \Omega)$, we conclude that $\bar{f}$ and $f$ are indeed the same function.

\medskip 
\noindent
{\it \underline{Part (c)}}
The proof of \eqref{eq:lemma-compactness-non-linear} employs a similar approach to that used in Part (b). Notably, for any $m \in \mathbb{N}$,
\begin{equation}\label{eq:f^m-proof}
\begin{aligned}
    \int_0^T 
    \int_{\Omega_{l_j}}
    \phi 
    f_{l_j}^m
    g_{j} 
    &=
     \int_0^T 
    \int_{\Omega_{l_j}}
    \phi 
    \big( f_{l_j}-\bar f_{l_j}+ \bar f_{l_j}\big)^m
    g_{j} 
    =
    \sum_{n=0}^m
    \binom{m}{n}
    \int_0^T 
    \int_{\Omega_{l_j}}
    \phi 
    (
    f_{l_j}
    -
    \bar 
    f_{l_j}
    )^n
    \bar 
    f_{l_j}^{m-n}
    g_j  
    \\
    &= 
    \int_0^T 
    \int_{\omega}
    \phi 
    \bar
    f_{l_j}^m
    \tilde 
    g_j  
    +
    \sum_{n=1}^m
    \binom{m}{n}
    \int_0^T 
    \int_{\mathcal{C}_{l_j}^*}
    \phi 
    (
    f_{l_j}
    -
    \bar 
    f_{l_j}
    )^n
    \bar 
    f_{l_j}^{m-n}
    \tilde 
    g_j ,
\end{aligned}
\end{equation}
where we use the fact that $\supp(\phi) \subset [0,T]\times \omega$, $\tilde g \equiv 0$ in $[0,T] \times (\omega \setminus \Omega_{l_j})$ and $\omega  \subseteq \mathcal{C}_{l_j}^*$. Since $\bar f_{l_j}|\omega$   strongly converges to $\bar f|_{\omega} = f|_{\omega}$ in $L^\infty(0,T;L^{\beta}(\omega))$ for any $ \beta \in [1, \infty)$,  $\bar f_{l_j}^m|\omega$  also converges strongly towards $ f^m|_{\omega}$ in $L^\infty(0,T;L^{\alpha'}(\omega))$ with $ \alpha' \in [1, \infty)$. Given the weak convergence of $\tilde{g}_j \rightharpoonup \theta g$ weakly in $L^{1}(0,T; L^\alpha(\Omega))$, we deduce 
\begin{equation*}
    \lim_{j\to \infty} 
    \int_0^T 
    \int_{\omega}
    \phi 
    \bar
    f_{l_j}^m
    \tilde 
    g_j  
    =
    \theta
    \int_0^T 
    \int_{\omega}
    \phi 
    f^m
    g .
\end{equation*}
Furthermore, the sum for $n = 1,\dots, m$ on the right-hand side of \eqref{eq:f^m-proof} vanishes as $j \to \infty$. Indeed, we have
\begin{align*}
    \bigg| \int_0^T 
    \int_{\mathcal{C}_{l_j}^*}
    \phi 
    (
    \tilde 
    f_{l_j}
    -
    \bar 
    f_{l_j}
    )^n
    \bar 
    f_{l_j}^{m-n}
    \tilde 
    g_j\bigg|
    \leq 
    \| \phi \|_{L^\infty((0,T)\times \Omega)} 
    \| f_{l_j}
    -
    \bar 
    f_{l_j}
    \|_{L^\infty(0,T; L^{\alpha'}(\mathcal{C}_{l_j}))}^n 
    \| \bar  f_{l_j} \|_{L^\infty((0,T)\times \Omega)}^{m-n}
    \| \tilde g_j \|_{L^1(0,T;L^\alpha(\Omega))}.
\end{align*}
The sequence $(\tilde g_j)_{j \in \mathbb N}$ is uniformly bounded in $L^1(0,T;L^\alpha(\Omega))$, while the assumptions of \Cref{thm:Aubin-Lions-Linfty} implies that $(\bar  f_{l_j} )_{j \in \mathbb N}$ is uniformly bounded in $L^\infty((0,T)\times \Omega)$. Finally, we recall that we have shown in Part (a) that the norms $ \| f_{l_j}
- \bar  f_{l_j} \|_{L^\infty(0,T; L^{\alpha'}(\mathcal{C}_{l_j}^*))} \to 0$ as $j \to \infty$.

\medskip
\noindent 
{\it \underline{Part (d)}:}  
we show that the sequence $(\bar{f}_l)_{l \in \mathbb N}$ defined in \eqref{eq:def:barfl} is precompact in $\mathcal{C}([0,T], L^s(\omega))$. We make use of the Arzel\'a-Ascoli Theorem for Banach-value functions, which requires the following properties:
\begin{enumerate}[(i)]
    \item the sequence  $(\bar{f}_l|_\omega)_{l\in \mathbb N} $ is uniformly bounded in $L^\infty(0,T;L^s(\omega))$,
    \item the sequence $(\bar{f}_l|_\omega)_{l\in \mathbb N} $ is equicontinuous (in time) in $\mathcal{C}([0,T], L^s(\omega))$,
    \item the sequence $(\bar{f}_l(t,\cdot)|_\omega)_{l\in \mathbb N} $ is precompact in $L^s(\omega)$, for any time $t \in [0,T]$.
\end{enumerate}
Without loss of generality, we determine that the aforementioned properties hold for $l \in \mathbb{N}$ with $l \geq L$ sufficiently large. This is because for $l < L$, we are dealing with a finite set of functions, which is trivially precompact. For $l \geq L \in \mathbb{N}$ (ensuring $\ee_l$ is sufficiently small), we have that $\omega + B(x, 10\ee_l) \subset \Omega$. This ensures that any cell $Y_k^{\ee_l}$ intersecting $\omega$ will have its adjacent cells also included in $\mathcal{C}_l$.

\smallskip 
\noindent 
The relation described by $(i)$ stems from the estimate \eqref{est:barfl-Lq}, i.e.~from the uniform boundedness of the sequence $(\bar{f}_l)_{l\in \mathbb{N}}$ in $L^\infty(0,T; L^q(\Omega))\hookrightarrow L^\infty(0,T; L^s(\Omega))$. Turning our attention to the equicontinuity in $(ii)$, we first remark that each function $\bar{f}_l$ is in $W^{1, r}(0,T;L^s(\Omega))$ with $\partial_t \bar{f}_l = \overline{\partial_t f_l}$, namely
\begin{equation}\label{eq:def:dtbarfl}
    \partial_t \bar{f}_l(t,x):= 
    \sum_{k \in \mathcal{K}}
    \mathbbm{1}_{Y_k^{\ee_l}}(x)
    \bigg(
        \fint_{Y_k^{*,\ee_l}}
        \partial_t f_l(t,\tilde x) d\tilde x
    \bigg),\quad \text{a.e.}\quad (t,x) \in [0,T]\times \Omega.
\end{equation}
Indeed, for any $\psi \in \mathcal{C}^\infty_c(0,T)$, the following identities between Bochner integrals are satisfies
\begin{equation}
\begin{aligned}
    \int_0^T \psi'(t) \bar{f}_l(t, \cdot ) dt 
    &= 
    \sum_{k \in \mathcal K} \frac{1}{|Y_k^{*,\ee_l}|} \mathbbm{1}_{Y_k^{\ee_l}}(\cdot )
    \int_0^T \int_{Y_k^{*,\ee_l}} \psi'(t) f_l(t,\tilde x) d\tilde xdt \\
    &=
    -
    \sum_{k \in \mathcal K} \frac{1}{|Y_k^{*,\ee_l}|} \mathbbm{1}_{Y_k^{\ee_l}}(\cdot )
    \int_0^T \int_{Y_k^{*,\ee_l}} \psi(t) \partial_t f_l(t,\tilde x) d\tilde x dt
    = 
    -
    \int_0^T \psi(t) \overline{\partial_t f}_l(t, \cdot ) dt .
\end{aligned}
\end{equation}
Moreover, we recall that the representative of $f_l$ is considered in $\mathcal{C}([0,T], L^s(\omega))$; therefore we have that
\begin{equation*}
\begin{aligned}
    \| \bar{f}_l(t_2, \cdot ) - \bar{f}_l(t_1, \cdot ) \|_{L^s(\omega)}
    &= 
    \Big\| 
        \int_{t_1}^{t_2} \partial_t\bar{f}_l(t, \cdot )dt  
    \Big\|_{L^s(\omega)}dt 
    \leq 
    \int_{t_1}^{t_2} 
    \| \partial_t\bar{f}_l(t, \cdot ) \|_{L^s(\omega)}dt
    = 
    \int_{t_1}^{t_2} 
    \| \overline{\partial_t f_l}(t, \cdot ) \|_{L^s(\omega)}dt
\end{aligned}
\end{equation*}
for any $0\leq t_1\leq  t_2 \leq T $. We next infer that a similar inequality to the one in \eqref{est:barfl-Lq} ensures that, for a.e.~$t\in (0,T)$, $\| \overline{\partial_t f_l}(t, \cdot ) \|_{L^s(\omega)} \leq \frac{1}{\theta^\frac{1}{s}} 
\| \partial_t f_l(t, \cdot ) \|_{L^s(\Omega_l)}$. Thus
\begin{equation*}
\begin{aligned}
    \| \bar{f}_l(t_2, \cdot ) - \bar{f}_l(t_1, \cdot ) \|_{L^s(\omega)}
    &\leq  
    \frac{1}{\theta^\frac{1}{s}}
    \int_{t_1}^{t_2} 
    \|  \partial_t \bar f_l(t, \cdot ) \|_{L^s(\Omega_l)}dt\\
    &\leq 
    \frac{1}{\theta^\frac{1}{s}}
    |t_2-t_1|^{1-\frac{1}{r}} \|  \partial_t f_l \|_{L^r (0,T;L^s (\Omega_l))} \leq 
    \frac{\mathbf C}{\theta^\frac{1}{s}} |t_2-t_1|^{1-\frac{1}{r}},
\end{aligned}
\end{equation*}
where in the last step we have used the uniform estimate from \eqref{est:Thm-Aubin-Lions-uniform-estimates}. The last inequality implies in particular the equicontinuity described in $(ii)$. 

\noindent 
To apply the Arzel\`a-Ascoli Theorem, we now focus on establishing $(iii)$, which states that the sequence $(f_l(t,\cdot))_{l\in \mathbb{N}}$ is precompact in $L^s(\omega)$, for any fixed $t \in [0,T]$. 
Following a methodology similar to that in \cite{MR1225439}, we employ the Kolmogorov criterion, which relates the precompactness to the following uniform convergence:
\begin{equation*}
    \lim_{h \to 0}  \| f_l(t, \cdot + h e^n ) - f_l(t, \cdot) \|_{L^s (\omega)} = 0,\quad 
    \text{uniformly in}\quad l \in \mathbb N,
\end{equation*}
for any element $e^n$ of the canonical basis $\{ e^1, \, e^2 \}$ in $\mathbb{R}^2$. 

\noindent 
Let $\bar{h}>0$ be the largest positive value such that $x + h e^n$ belongs to $\mathcal{C}_l$, for any $x\in \omega$ and $h \in [0, \bar{h})$. 
We consider two contiguous cell $Y_k^{\ee_l}$ and $Y_j^{\ee_l}$ such that, $Y_k^{\ee_l}\cap \omega \neq \emptyset$ and $Y_j^{\ee_l} = 
\{ x+ \ee_l e^n,\,|\,x\in Y_k^{\ee_l}\}$. Remark that from our original assumption of $l \geq L$, the cell $Y_{j}^{\ee_l}$ is still in $ \mathcal{C}_l$ (but maybe not in $\omega$). In particular the function $\bar{f}_l(t,x) = \fint_{Y^{*, \ee_l}_j} f_l (t, \tilde x) d\tilde x$, for any $x \in Y_{j}^{\ee_l}$. Moreover, for a fixed time $t \in [0,T]$, we remark that
\begin{align*}
    \big|
        \bar{f}_l (t, x+ \ee_l e^n)
        -
        \bar{f}_l (t, x)
    \big|
    &= 
    \bigg|
        \fint_{Y_j^{*,\ee_l}}
        f_l(t,\tilde x) d\tilde x
        -
        \fint_{Y_k^{*,\ee_l}}
        f_l(t,\tilde x) d\tilde x
    \bigg|, \quad \forall x\in \mathcal C_l \quad \text{such that}\quad  x+ \ee_l e^n \in \mathcal{C}_l.
\end{align*}
In particular, the above identity holds true for any $x \in \omega$. Hence we assert the existence of a constant $\tilde C_s > 0$, dependent solely on $s \geq 1$ and the unit cell $Y^*$, satisfying the following inequality:
\begin{equation}\label{eq:Poincar-Wirtinger-to-be-cited-at-the-appendix}
    \begin{aligned}
    \big|
        \bar{f}_l (t, x+ \ee_l e^n)
        -
        \bar{f}_l (t, x)
    \big|
    &\leq 
    \tilde C_s \ee_l^{1-\frac{2}{s}} 
    \| \nabla f_l (t, \cdot ) \|_{L^s(Y_k^{*,\ee_l}\cup Y_j^{*,\ee_l} )}.
\end{aligned}
\end{equation}
We postpone the proof of this inequality to \Cref{lemma:mean-in-Z}. For a general $h \in [0, \bar{h})$, we have that $h \leq   \ee_l$ or $h >  \ee_l$. If $h \leq  \ee_l$, 
we introduce the sets $\mycal{A}_{l,k,h},\,\mycal{B}_{l,k,h}\subset Y^{\ee_l}_k$  
\begin{equation*}
\begin{aligned}
    \mycal{A}_{l,k,h} &= 
    \{x \in Y^{\ee_k}_k\text{ such that } -\ee_l/2  \leq  (x-x_k^{\ee_l} )\cdot e^n <  \ee_l/2-h\},\\
    \mycal{B}_{l,k,h} &= 
    \{x \in Y^{\ee_k}_k\text{ such that }  \ee_l/2-h \leq  (x-x_k^{\ee_l})\cdot e^n \leq  \ee_l/2\}.
\end{aligned}
\end{equation*}
Since $\bar{f}_l $ is constant in each cell, we obtain that
\begin{equation*}
    \left\{
    \begin{alignedat}{4}
        &|\bar{f}_l (t, x+h e^n) - \bar{f}_l (t, x) |
        = 0
        &&\text{for a.e. }x\in \mycal{A}_{l,k,h}
        ,
        \qquad 
        \\
        &|\bar{f}_l (t, x+h e^n ) - \bar{f}_l (t, x  ) |
        \leq 
        \tilde C_s\ee_l^{1-\frac{2}{s}}
        \| \nabla f_l (t, \cdot ) \|_{L^s(Y_k^{*,\ee_l}\cup Y_j^{*,\ee_l} )}
        \quad 
        &&\text{for a.e. }x\in \mycal{B}_{l,k,h}.
    \end{alignedat}
    \right.
\end{equation*}
Hence, remarking that $|\mycal{B}_{l,k,h}| =\ee_l \cdot h$, we obtain that
\begin{equation*}
\begin{aligned}
    \int_\omega 
    \big|
        \bar{f}_l (t, x+ h e^n)
        -
        \bar{f}_l (t, x)
    \big|^s 
    dx
    &= 
    \sum_{k\in \mathcal{K}} 
    \int_{\omega \cap  Y^{  \ee_l}_k}
    \big|
        \bar{f}_l (t, x+ h e^n)
        -
        \bar{f}_l (t, x)
    \big|^s
    dx
    \\
    &= 
    \sum_{k\in \mathcal{K}} 
    \int_{\omega \cap  \mycal{B}_{l,k,h}}
    \big|
        \bar{f}_l (t, x+ h e^n)
        -
        \bar{f}_l (t, x)
    \big|^s
    dx
    \\
    &\leq 
    \sum_{k \in \mathcal{K}} 
    | \mycal{B}_{l,k,h} |
    \tilde C_s^s \ee_l^{s-2}
    \| \nabla f_l (t, \cdot) \|_{L^s(Y_k^{*,\ee_l}\cup Y_j^{*,\ee_l} )}^s\\
    &\leq 
    \tilde C_s^s \ee_l^{s-1}h
    2\| \nabla f_l (t, \cdot) \|_{L^s(\Omega_l)}^s.
\end{aligned}
\end{equation*}
Denoting $C_s = 2^{1/s}\tilde C_s$, this in particular implies that for any $0\leq h \leq \ee_l$
\begin{equation}\label{est:first-estimate-h-aubin-lions}
\begin{aligned}
     \|  
        \bar{f}_l (t, \cdot + h e^n) -  \bar{f}_l (t, \cdot )
    \|_{L^s(\omega)}
    \leq 
    C_s \ee_l^{1-\frac{1}{s}}  h^\frac{1}{s}
    \| \nabla f_l (t, \cdot ) \|_{L^s(\Omega_l)}.
\end{aligned}
\end{equation}
We shall here remark that the terms on the right-hand side do not depend on the domain $\omega$. We make use of this aspect in the case of $h> \ee_l$. Indeed there exists $M\in \mathbb N$ and $0\leq h'<\ee_l$ such that $h = M\ee_l + h'$. 
We obtain 
\begin{equation*}
   \begin{aligned} 
    \|  
        &\bar{f}_l (t, \cdot + h e^n) -  \bar{f}_l (t, \cdot )
    \|_{L^s(\omega)} \\
    &= 
    \bigg\|  
        \bar{f}_k (t, \cdot + (M  \ee_l + h') e^n) -  \bar{f}_k (t, \cdot + M\ee_l e^n)+ 
        \sum_{ m = 1 }^M   \bar{f}_l (t, \cdot +  m \ee_l e^n) -  \bar{f}_l (t, \cdot +  (m-1) \ee_l e^n)
    \bigg\|_{L^s(\omega)} \\
    &\leq 
    \|  
        \bar{f}_k (t, \cdot + (M  \ee_l + h') e^n) -  \bar{f}_k (t, \cdot + M\ee_l e^n)
    \|_{L^s(\omega)}
    +
    \sum_{ m = 1 }^M 
    \|  
        \bar{f}_l (t, \cdot +  m \ee_l e^n) -  \bar{f}_l (t, \cdot +  (m-1) \ee_l e^n)
    \|_{L^p(\omega)}\\
    &\leq 
    \|  
        \bar{f}_k (t, \cdot + h') -  \bar{f}_k (t, \cdot )
    \|_{L^s(\omega + M  \ee_l e^n)}
    +
    \sum_{ m = 1 }^M 
    \|  
        \bar{f}_l (t, \cdot + \ee_l e^n) -  \bar{f}_l (t, \cdot)
    \|_{L^p(\omega+  (m-1)\ee_l e^n)}.
   \end{aligned}
\end{equation*}
Hence, applying \eqref{est:first-estimate-h-aubin-lions} with $\omega$ replaced by each domain $\omega+  \tilde m \ee_l e^n$, $\tilde m = 0,\dots,M$, we finally obtain
\begin{equation*}
   \begin{aligned} 
    \|  
        \bar{f}_l (t, \cdot + h e^n) -  \bar{f}_l (t, \cdot )
    \|_{L^s(\omega)} 
    &\leq 
    C_s
    \| \nabla f_l (t, \cdot) \|_{L^s(\Omega_l)}
    \bigg\{
    \ee_l^{1-\frac{1}{s}}(h')^\frac{1}{s}
    +
    \sum_{ m = 1 }^M 
     \ee_l^{1-\frac{1}{s}} \ee_l^\frac{1}{s}
     \bigg\}
     \\
    &\leq 
    C_s
    \| \nabla f_l  (t, \cdot) \|_{L^s(\Omega_l)}
    \Big\{
      h^{1-\frac{1}{s}}h ^\frac{1}{s}
     +
     M\ee_l
     \Big\}\\
    &\leq 
    C_s
    \| \nabla f_l (t, \cdot) \|_{L^s(\Omega_l)}
    2h.
    \end{aligned}
\end{equation*}
We hence invoke the uniform estimates in \eqref{est:Thm-Aubin-Lions-uniform-estimates} and recall that $(\ee_l)_{l\in \mathbb N}$ decrease monotonically. Thus
\begin{equation}\label{est:at-the-end-of-the-proof-of-AL}
\begin{alignedat}{4}
    \|  
        \bar{f}_l (t, \cdot + h e^n) -  \bar{f}_l (t, \cdot )
    \|_{L^s(\omega)}
    &\leq 
    \mathbf C C_s \ee_l^{1-\frac{1}{s}}  h^\frac{1}{s}
    \qquad
    &&\text{for any }h \in [0, \ee_l],
    \\
     \|  
        \bar{f}_l (t, \cdot + h e^n) -  \bar{f}_l (t, \cdot )
    \|_{L^s(\omega)}
    &\leq 
    2\mathbf C C_s h,
    \quad 
    &&\text{for any }h \in [\ee_l, \bar h).
\end{alignedat}
\end{equation}
for almost any $t \in (0,T)$. Thanks to the Kolmogorov criterion $(\bar{f}_l|_\omega(t, \cdot ))_{k\in \mathbb N}$ is precompact in $L^s(\omega)$ for almost any time $t\in (0,T)$. Since $\bar{f}_l|_\omega$ forms an equicontinouous sequence in  $\mathcal{C}([0,T],L^s(\omega))$ from part (ii), 
then $(\bar{f}_l|_\omega(t, \cdot ))_{k\in \mathbb N}$ is precompact in $L^s(\omega)$ for any time $t\in [0,T]$.

\medskip 
\noindent 
{\it \underline{Part (e)}} We conclude our proof by showing that $\nabla f\in L^\infty(0,T;L^s(\Omega)) $ and $\partial_t f \in L^r(0,T;L^s(\Omega))$, if $s>1$.  We recall that $\bar{f}_{l_j}|_{\omega} \to f|_{\omega} $ strongly in $\mathcal{C}([0,T], L^s(\omega))$, for any $\omega \ssubset \Omega$. Sending $\ee_{l_j}\to 0$ in \eqref{est:at-the-end-of-the-proof-of-AL}, we thus remark that 
\begin{equation*}
     \|  
        f(t, \cdot + h e^n) -  f(t, \cdot )
    \|_{L^s(\omega)}
    \leq 
    2\mathbf C C_s h,
    \quad 
    \text{for any }h \in [0, \bar h),\quad
    \text{any }\omega \ssubset \Omega\quad \text{and a.e.~}t\in(0,T).
\end{equation*}
The above relation implies in particular that $f(t, \cdot)$ belongs to $W^{1,s}(\Omega))$ and satisfies $\| \nabla f \|_{L^s(\Omega)} \leq  2\mathbf C C_s$, for a.e.~$t \in (0,T)$ (cf.~for instance Proposition 9.3 in \cite{MR2759829}). 
Finally, $\partial_t f$ belongs to $L^r(0,T;L^s(\Omega))$ since the sequence $(\bar{f}_{l_j})_{l\in \mathbb N}$ is uniformly bounded in $W^{1,r}(0,T;L^s(\Omega))$. From \eqref{eq:def:dtbarfl} and an argument of weak convergence, we have therefore
\begin{equation*}
     \| \partial_t f \|_{L^r(0,T;L^s(\Omega))}
     \leq 
     \liminf_{j\to \infty}
     \| \partial_t \bar{f}_{l_j} \|_{L^r(0,T; L^s(\Omega))} \leq 
     \liminf_{j\to \infty}
     \frac{1}{\theta^{\frac{1}{s}}}
     \| \partial_t f_{l_j} \|_{L^r(0,T; L^s(\Omega_{l_j}))} \leq \frac{1}{\theta^{\frac{1}{s}}} \mathbf C.
\end{equation*}
This concludes the proof of the \Cref{lemma:compactness_Omega_eps} and \Cref{thm:Aubin-Lions-Linfty}.
\end{proof}
\section{The extension of the pressure}\label{sec:pressure}
\noindent 
We begin our analysis by addressing the family of pressures $\pre_\ee \in L^2(0,T; \lbar^2_0(\Omega_\ee))$. In this section we introduce an appropriate ``extension'' $P_\ee$ (although, more precisely, $\nabla P_\ee$ extends $\nabla \pre_\ee$), which satisfies some uniform estimates in $L^2((0,T), \lbar^p_0(\Omega))$, for any $p\in [1,2)$. We build $P_\ee$ implicitly with a methodology reminiscent of Allaire's approach in \cite{MR1079189, MR1079190}, wherein $p=2$ was assumed. However, diverging from Allaire's framework, the term $(\nabla d_\ee)^T \Delta d_\ee$ on the right-hand side~of the Stokes equation \eqref{maineq1} is uniformly bounded only in $L^1((0,T)\times \Omega_\ee)$. Despite this weak norm, we transpose these estimates to the pressure $P_\ee$  using a norm that is less regular than the one used by Allaire, reflecting our assumption $p<2$.

\noindent 
Invoking the operator $\mathcal{R}_{q,\ee}$ of Lemma \ref{lemma:restriction_operator}, we determine $P_\ee$ through the following Proposition. 
\begin{prop}\label{prop:pressure}
    Consider the family of pressures $\pre_\ee \in L^2(0,T;\lbar^2_0(\Omega_\ee))$ of the Stokes equation \eqref{maineq1}. There exists a family of pressures $P_\ee:(0,T)\times \Omega\to \mathbb R$ which are in $L^2(0,T; \lbar^p_0(\Omega))$, for any $p\in (1,2)$, satisfying 
    $\nabla (P_{\ee|\Omega_\ee}) = \nabla \pre_\ee$ in $L^2(0,T;W^{-1, p}(\Omega_\ee))$, for any $\ee>0$. Additionally, the following estimate holds true
    \begin{equation}\label{est:Pee-proposition}
        \| P_\ee  \|_{L^2(0,T;L^p(\Omega))}
        \leq \frac{D_{p, \Omega}\mathcal{E}_0^\ee(1+\mathcal{E}_0^\ee)}{\ee}
        ,
    \end{equation} 
    for a suitable positive constant $D_{p, \Omega}$ that depends only on the exponent $p\in (1,2)$ and the domain $\Omega$. Moreover
	\begin{equation}\label{eq:dual-form-of-Pe-and-pe}
		\int_{0}^T\int_{\Omega}	P_\ee(t,x)\Div\, \psi(t,x) dxdt= 
		\int_{0}^T\int_{\Omega_\ee}	\pre_\ee(t,x) \Div\, \mathcal{R}_{p',\ee} \psi(t,x)dxdt.
	\end{equation}	 
    for any $\psi \in L^2(0,T;W^{1,p'}_0(\Omega))$, where $p'= p/(p-1)\in (2, \infty)$.
\end{prop}
\begin{remark}
    We highlight that $P_\ee$ in \eqref{prop:pressure} does not depend on the choice of $p\in(1,2)$. Moreover, as $p \to 2$, the constant $D_{p, \Omega}$ becomes unbounded, making it infeasible to extrapolate \eqref{est:Pee-proposition} into a strict $L^2$ framework. Moreover, while \eqref{est:Pee-proposition} technically holds also for $p=1$, this particular case offers no substantial insights. 
\end{remark}
\begin{proof}[Proof of \Cref{prop:pressure}]
        Thanks to \Cref{lemma:Bogovskii}, for any function $g\in L^2(0,T;\lbar^{p'}_0(\Omega))$, there exists a (non-unique) vector  function $\psi_g$ in  $L^2(0,T;W^{1,p'}_0(\Omega))^2$ satisfying
        \begin{equation}\label{Bogo}
            \Div\, \psi_g  = g\quad \text{and}\quad 
            \| \psi_g  \|_{L^2(0,T;W^{1, p'}_0(\Omega))}\leq C_{p', \Omega} \| g\|_{L^2(0,T;L^{p'}(\Omega))},
        \end{equation}
        where $C_{p', \Omega} $ depends only on $p'$ and $\Omega$. Moreover, from \Cref{lemma:restriction_operator}, $\mathcal{R}_{p',\ee}$ maps $W^{1, p'}_0(\Omega)$ into $W^{1, p'}_0(\Omega_\ee)$.  We thus infer that
        \begin{equation}\label{def:Fg}
            \mathcal{F}_\ee(g) =
            \int_0^T 
            \int_{\Omega_\ee} 
            \pre_\ee(t,x) 
            \Div\,\mathcal{R}_{p',\ee} (\psi_g) (t,x)
            dxdt
        \end{equation}
        defines a continuous linear functional $\mathcal{F}_\ee:L^2(0,T; \lbar^{p'}_0(\Omega))\rightarrow \RR$. The integral on the right-hand side~of \eqref{def:Fg} is finite, since $\pre_\ee \in L^2(0,T; L^p(\Omega_\ee))$ and $\Div\,\mathcal{R}_{p',\ee} (\psi_g)$ is in  $ L^2(0,T; L^{p'}(\Omega_\ee))$. However, $\mathcal{F}_\ee$ is a well-defined finite linear operator only if such integral does not depend on the particular choice of $\psi_g$ in \eqref{Bogo}. Consider an other function $\tilde \psi_g $ in $L^2(0,T;W^{1,p'}_0(\Omega))^2$ satisfying \eqref{Bogo}, then $\Div\, (\tilde \psi_g- \psi_g) 
        \equiv 0$, thus also 
        \begin{equation*}
            \Div\, \mathcal{R}_{p', \ee}\tilde \psi_g- \Div\, \mathcal{R}_{p', \ee}\psi_g=
            \Div\, \mathcal{R}_{p', \ee}(\tilde \psi_g- \psi_g)
         \equiv 0,
        \end{equation*}
        thanks to the third property of \Cref{lemma:restriction_operator}. In particular $\mathcal{F}_g$ in \eqref{def:Fg} does not depend on the choice of $\psi_g$. 

        \noindent 
        It follows trivially that $\mathcal{F}_g$ is linear and we shall only prove that it is continuous.
        To this end we invoke the Stokes equation \eqref{maineq1} and the no-slip boundary condition of $u_\ee$ in \eqref{bdy-cdt}. First, we have that
	\begin{equation*}
        \begin{aligned}
		\mathcal{F}_\ee( g ) 
		=
            \int_{\mathbb R_+}
            \int_{\Omega_\ee} 
            \nabla \uu_\ee(t,x) : 
            \nabla \mathcal{R}_{p',\ee }\psi_g (t,x)
            dx dt
            +
            \int_{\mathbb R_+}
    	\int_{\Omega_\ee} 
            (\nabla d_\ee(t,x))^T \Delta d_\ee(t,x) \cdot  \mathcal{R}_{p',\ee } 
            \psi_g(t,x)
            dxdt\\
            -
            \int_{\mathbb R_+}
            \int_{\Omega_\ee} 
            F_\ee(t,x) : 
            \mathcal{R}_{p',\ee }\psi_g (t,x)
            dx dt
            -
            \int_{\mathbb R_+}
    	\int_{\Omega_\ee} 
            \nabla H_\ee  :\nabla \mathcal{R}_{p',\ee } 
            \psi_g(t,x)
            dxdt.
        \end{aligned}
        \end{equation*}
        Thanks to the H\"older inequalities, we deduce therefore that 
        \begin{equation*}
        \begin{aligned}
            |\mathcal{F}_\ee
            &( g ) |
		\leq 
		\Big(
                \| \nabla \uu_\ee \|_{L^2((0,T)\times \Omega_\ee)}
                +
                \| \nabla H_\ee  \|_{L^2((0,T)\times \Omega_\ee)}
            \Big)
            \| \nabla \mathcal{R}_{p',\ee } \psi_g \|_{L^2((0,T)\times \Omega_\ee)}    
            +
            \\
            &+
            \| \nabla d_\ee   \|_{L^\infty(0,T;L^2(\Omega_\ee))}
            \| \Delta d_\ee   \|_{L^2((0,T)\times \Omega_\ee)}
		  \| \mathcal{R}_{p',\ee } \psi_g \|_{L^2(0,T; L^\infty(\Omega_\ee))}
            +
            \| F \|_{L^2((0,T)\times \Omega_\ee)}
            \| \mathcal{R}_{p',\ee } \psi_g \|_{L^2((0,T)\times \Omega_\ee)}.
        \end{aligned}
	\end{equation*}
         We hence recall that $p'>2$ and we apply the Poincar\'e inequality of \Cref{Poincare_ineq} to gather
        \begin{align*}
            \| \mathcal{R}_{p',\ee } \psi_g \|_{L^2((0,T)\times \Omega_\ee)}
            &=
            \| \mathcal{R}_{p',\ee } \psi_g \|_{L^2((0,T)\times \Omega)}
            \leq  C_{2, \Omega} \ee  
            \| \nabla \mathcal{R}_{p',\ee } \psi_g \|_{L^2((0,T)\times \Omega)}\\
            &\leq  C_{2, \Omega}|\Omega|^{\frac{1}{p}-\frac{1}{2}} 
            \ee  
            \| \nabla \mathcal{R}_{p',\ee } \psi_g \|_{L^2(0,T;L^{p'}(\Omega))}.
        \end{align*}
        Moreover, the Sobolev embedding $W^{1,p'}_0(\Omega)\hookrightarrow L^\infty(\Omega)$ holds true, therefore 
        \begin{equation*}
            \| \mathcal{R}_{p',\ee } \psi_g \|_{L^2(0,T;L^\infty(\Omega_\ee))}
            =
            \| \mathcal{R}_{p',\ee } \psi_g \|_{L^2(0,T;L^\infty(\Omega))}
            \leq 
            \tilde C_{p', \Omega}
            \| \nabla \mathcal{R}_{p',\ee } \psi_g \|_{L^2(0,T;L^{p'}(\Omega))},
        \end{equation*}
        where $\tilde C_{p', \Omega}>0$ depends only on $p'$ (i.e.~on $p$) and $\Omega$. Thanks to \Cref{rmk:uniform-estimates-norms-solutinos} and the fourth property of \Cref{lemma:restriction_operator}, we gather that there exists a positive constant $\tilde D_{p, \Omega}>0$ depending only on $p'$ (i.e.~on $p$) and $\Omega$ such that
        \begin{equation*}
        \begin{aligned}
            |\mathcal{F}_\ee( g ) |
		\leq 
		&
            \bigg(
            \| \nabla \uu_\ee \|_{L^2((0,T)\times \Omega_\ee)}
            +
            \| \nabla d_\ee   \|_{L^\infty(0,T; L^2(\Omega_\ee))}
            \| \Delta d_\ee   \|_{L^2((0,T)\times \Omega_\ee)}
            +
            \\
            &
            \hspace{2.75cm}+
            \ee 
            \| F_\ee \|_{L^2((0,T)\times \Omega)}
            +
            \| \nabla H_\ee \|_{L^2((0,T)\times \Omega)}
            \bigg)
		  \tilde D_{p, \Omega}
            \| \nabla \mathcal{R}_\ee \psi_g \|_{L^2(0,T; L^{p'}(\Omega_\ee))}
            \\
            &
            \leq 
            \tilde D_{p, \Omega}
            \mathcal{E}_0
            (
                1 
                +
                \mathcal{E}_0
            )
            C_{p', \Omega}
            \bigg(
                  \frac{1}{\ee}\|  \psi_g \|_{L^2(0,T; L^{p'}(\Omega))}  
                  +
                  \| \nabla \psi_g \|_{L^2(0,T; L^{p'}(\Omega))}
            \bigg)\\
            &\leq 
            \frac{\tilde D_{p, \Omega}\mathcal{E}_0
            (
                1 
                +
                \mathcal{E}_0
            )}{\ee}
             C_{p', \Omega}
             \bar{D}_{p, \Omega}
            \| \psi_g \|_{L^2(0,T; W^{1,p'}_0(\Omega))}
            ,
        \end{aligned}
	\end{equation*}
	for a suitable Poincar\'e constant $  \bar{D}_{p, \Omega}>0$ that depends 
        only on $\Omega$ and $p$. Thanks to \eqref{Bogo}, we finally obtain that
	\begin{equation*}
		|\mathcal{F}( g )| \leq 
            \frac{D_{p, \Omega}\mathcal{E}_0
            (
                1 
                +
                \mathcal{E}_0
            )}{\ee }
            \| g \|_{L^2(\mathbb R_+, L^{p'}(\Omega ))}.
        \end{equation*}
        for a suitable constant $  D_{p, \Omega}>0$ that depends 
        only on $\Omega$ and $p$ (which blows up as $p\to 2$). Hence $\mathcal F_\ee$ is linear and continuous in $L^2(0,T;\lbar^{p'}_0(\Omega))$, hence a standard duality argument implies that there exists a function 
	$P_\ee \in L^2(\mathbb R_+, \lbar^p_0(\Omega))$,
        such that
	\begin{equation*}
		\int_0^T\int_\Omega P_\ee(t,x) g(t,x)dxdt = \mathcal{F}_\ee(g)
            \quad 
            \text{and}
            \quad 
            \| P_{\ee} \|_{L^2(\mathbb R_+, L^p(\Omega))}
            \leq 
            \frac{D_{p, \Omega}\mathcal{E}_0
            (
                1 
                +
                \mathcal{E}_0
            )}{\ee }.
	\end{equation*}
        To conclude the proof of \Cref{prop:pressure}, it remains to prove that $\nabla P_\ee \equiv \nabla \pre_\ee $ in $L^2(0,T;W^{-1, p}(\Omega_\ee))$. This is a direct consequence of \Cref{lemma:restriction_operator} and the fact that $\mathcal{R}_{p',\ee} \psi = \psi$,  
	for any $\psi \in L^2(0,T;W^{1,p'}_0(\Omega_\ee))^2$. Finally, we remark that $P_\ee$ does not depend on the choice of $p\in (1,2)$, since $\mathcal{R}_{p', \ee}$ in  \eqref{def:Fg}  coincide in the dense subset $ \mathcal{C}^\infty_c(\Omega)^2$ (cf.~\Cref{rmk:restriction-operators-coincide}).
\end{proof}

\section{The Darcy law}\label{sec:Darcy}
\noindent
In this section we deal with the analysis of the flow $\uu_\ee$ and the proof of \eqref{part-i-main-thm} in \Cref{main_thm}. In order to prove that the functions $\tilde u_\ee$  converge strongly to zero in $L^2((0,T)\times \Omega)$, it is enough to observe that $\tilde \uu_\ee$ belongs to $L^2(0,T;H^1_0(\Omega))$ and it is identically null in $\Omega\setminus \Omega_\ee$. Hence, thanks to the Poincar\'e inequality of \Cref{Poincare_ineq} in $H^1_0(\Omega_\ee)$ and thanks to \Cref{rmk:uniform-estimates-norms-solutinos}, we observe that
\begin{equation}\label{eq:estimate-of-tildeu-ee-Poinc}
	\| \tilde  \uu_\ee \|_{L^2((0,T)\times \Omega)} =
	\| 		   \uu_\ee \|_{L^2((0,T)\times \Omega_\ee)}
	\leq C \ee  \| \nabla \uu_\ee \|_{L^2((0,T)\times \Omega_\ee)}
	\leq 
	C \ee 
	\mathcal{E}_0, 
\end{equation}
where $C =  C_{2, \Omega}$ depends only on $\Omega$. We hence deduce that the family of functions $\tilde \uu_\ee$ strongly converges to zero in $L^2((0,T)\times \Omega)$, as $\ee \to 0$. 

\noindent 
Our major goal in this section is therefore to prove that the rescaled profile $\tilde \uu_\ee/\ee \rightharpoonup u$ in $L^2((0,T)\times \Omega)$, where $u \in L^2((0,T)\times \Omega)$ is the unique solution of the Darcy system \eqref{eq:Darcy}.

\noindent
As underlying procedure, we shall consider a general sequence $(\ee_l)_{l\in \mathbb N}$ that is monotonically decreasing and converges to $0$ as $l\to \infty$. Given the uniform estimates of \Cref{rmk:uniform-estimates-norms-solutinos}, we have that
\begin{equation*}
    \bigg\| \frac{\tilde u }{\ee_{l}} \bigg\|_{L^2((0,T)\times \Omega)} \leq C\mathcal{E}_0,
\end{equation*}
thus there exists a subsequence $\ee_{l_k}$ for $k$ in $\mathbb N$ such that the function $\tilde u_{\ee_{l_k}}/\ee_{l_k}\rightharpoonup u$ in $L^2((0,T) \times \Omega)$. Here, $u$ is a function in $L^2((0,T)\times \Omega)$ that could initially depend on the subsequence $(\ee_{l_k})_{k\in \mathbb N}$. However, as we will elaborate in the upcoming proposition, $u$ does not dependent on it and is, in fact, the solution of the Darcy equations \eqref{eq:Darcy}. 
Therefore, from the arbitrariness of $(\ee_l)_{l\in \mathbb N}$, we conclude that $\tilde \uu_\ee/\ee\rightharpoonup u$, as $\ee \to 0$ in a general sense. 
\begin{prop}\label{prop:Darcy}
        For any $(\ee_l)_{l\in \mathbb N}$ with $\ee_l\to 0$, there exists a subsequence $(\ee_{l_k})_{k\in \mathbb N}$ such that $\tilde u_{\ee_{l_k}}/\ee_{l_k}\rightharpoonup u$ in $L^2((0,T) \times \Omega)$, where $u$ is the unique weak solution of
	\begin{equation*}
		\uu + \mathbb{B} \nabla \pre = G
            \quad \text{in}\quad (0,T)\times \Omega,\qquad
		\Div\,\uu = 0 \quad \text{in}\quad  (0,T)\times \Omega,\qquad
		\uu\cdot \nu  = 0\quad \text{on}\quad  (0,T)\times \partial \Omega,
  	\end{equation*}
	where the matrix $\mathbb{B}\in \mathbb{R}^{2\times 2}$ 
        is as  in \eqref{eq:matrices-AB} and $G$ is defined as in \Cref{main_thm}.
\end{prop}
\begin{proof}
For the sake of clarity, in subsequent discussions, we omit the index $l$ of the sequence $(\ee_l)_{l\in \mathbb N}$ (and any related subsequences), representing it simply as $\ee \to 0$. Nevertheless, our approach remains as previously detailed. 

\noindent 
We make use of a similar procedure as the one introduced by Mikelic in \cite{MR1131849}: we first set two meaningful functions $\omega^{i}$, with $i ={1,2}$, that are $Y^*$-periodic, where we recall that $Y^* = ]-1/2,1/2[^2\setminus \mathcal{T}$ represents the unitary cell configuration. Each $\omega^i$ is solution of the following Stokes problem:
\begin{equation}\label{eq:omega}
			- \Delta_y \omega^{i} + \nabla_y \pi^i = e^i \quad\text{in}\quad Y^*,\qquad
			\Div_y\, \omega^i = 0\quad \text{in}\quad Y^*,\qquad
			\omega^i = 0\quad\text{on}\quad \partial \mathcal{T},
\end{equation}
for $i = 1,2$, where $\{e^1, e^2\}$ is the canonical basis of $\mathbb R^2$. Thanks to Lemma IV.6.1 of \cite{Galdi2011}, the functions $\omega^i$ and the pressures $\pi^i$ belong to $H^2(Y^*)$ and $L^2(Y^*)$, respectively, thanks to the fact that the boundary $\partial T$ is of class $C^\infty$ (the class  $C^2$  would here suffice). Furthermore, also from Lemma IV.6.1 of \cite{Galdi2011}, there exists a constant $\smallc_{Y^*}>0$, that is uniquely determined by the geometry of $Y^*$, ensuring that
\begin{equation*}
	\| \omega^i \|_{H^2(Y^*)}+
	\| \pi^i \|_{L^2(Y^*)}
	\leq 
	\frac{{\smallc}_{Y^*} }{|Y^*|^{\frac{1}{2}}} 
        \| e^i\|_{L^2(Y^*)}
	=
	\smallc_{Y^*},
\end{equation*}
Thus, we consider the extensions $\tilde w^{i}$ and $\tilde \pi^i$ in $[-1/2,1/2]^2$ (which are identically null in $\mathcal{T}$) and we rescale them through  the relations $\omega^{i,\ee}(x):= \tilde \omega^i(x/\ee)$ and $\pi^{i,\ee}(x) = \tilde \pi(x/\ee)$, for any $x\in \mathbb R^2$ (for abbreviation we omit the \~{} in the notation of $\omega^{i,\ee}$ and $\pi^{i, \ee}$, the reader should keep in mind that these are functions defined almost everywhere in $\mathbb R^2$ and identically null in each particles periodically distributed).  In particular, for any cell $Y^{*,\ee}_k\subset \Omega_\ee$, the couple $(\omega^{i,\ee},\pi^\ee)$  satisfies  
\begin{equation}\label{eq:omega_eps}
		- \ee^2\Delta \omega^{i,\ee} + \ee \nabla \pi^{i,\ee} = e^i \quad\text{in}\quad 
        Y^{*,\ee}_k,\qquad
			\Div\, \omega^{i,\ee} = 0\quad \text{in}\quad 
        Y^{*,\ee}_k,\qquad
			\omega^{i,\ee} = 0\quad\text{on}\quad 
    \partial \mathcal{T}^\ee_k
\end{equation}
pointwise. We next introduce a general test function $\varphi \in \mathcal{C}_c^\infty((0,T)\times \Omega)$. Since $\supp \varphi $ is compact in $(0,T)\times \Omega$, its projection $\Sigma = \{ x \in \Omega,  \text{ such that }\exists t \in (0,T) \text{ with }(t,x) \in \supp \varphi \}$ is compact in $\Omega$. Given a sufficiently small $\ee>0$, any $x\in \Omega_\ee \cap \Sigma$ belongs to a closed cell $\overline{Y}^{\ee}_{k}=x_k^\ee + [-\ee/2, \ee/2]^2$, which is moreover included in $\Omega$ (and thus $\overline{Y}^{*,\ee}_{k}$ is also included in $\Omega_\ee$). We hence remark that the function $\ee\, \omega^{i,\ee}\varphi$ is well defined in  $L^2(0,T; H^1_0( \Omega_\ee))$ and we can therefore test it on the weak formulation of the Stokes equations \eqref{maineq1}:
\begin{equation}\label{to_control_thm:conv_of_u}
        \begin{aligned}           
        \underbrace{\ee\int_0^T\int_{\Omega_\ee} 
        \nabla (\varphi \omega^{i,\ee}): \nabla \uu_\ee}_{\I^\ee} -
		\underbrace{\ee\int_0^T\int_{\Omega_\ee} 
        \pre_\ee \Div\,(\varphi \omega^{i,\ee})}_{\I\I^\ee} = 
	-
        \underbrace{\ee
        \int_0^T
        \int_{\Omega_\ee} 
       (\nabla d_\ee)^T\Delta d_\ee: \omega^{i,\ee}\varphi}_{\I\I\I^\ee}
        + \\
        +
        \underbrace{
        \int_0^T
        \int_{\Omega_\ee}
        \ee F_\ee\cdot \omega^{i,\ee}\varphi
        }_{\I\mathcal{V}^\ee}
        +
         \underbrace{\ee
        \int_0^T
        \int_{\Omega_\ee}
        \nabla(\omega^{i,\ee}\varphi):
        \nabla H_\ee\cdot }_{\mathcal{V}^\ee}.
\end{aligned}
\end{equation}
We analyse each component $\I^\ee$, $\I\I^\ee$, $\I\I\I^\ee$, $\I\mathcal{V}^\ee$ and $\mathcal{V}^\ee$ separately. The components $\I^\ee$ and $\mathcal{V}^\ee$ are similar, allowing for a parallel treatment. We elaborate on the details specifically for $\I^\ee$, while the analysis for $\mathcal{V}^\ee$ follows with a similar approach, replacing $u_\ee$ with $H_\ee$. First
	\begin{equation}\label{Iee_Thm:conv_of_u}
	\begin{aligned}
	\I^\ee &= 
	    \ee\int_0^T\int_{\Omega_\ee} 
            \nabla \varphi\otimes \omega^{i,\ee}: \nabla \uu_\ee 
		+ \ee\int_0^T\int_{\Omega_\ee} \varphi \nabla \omega^{i,\ee}: \nabla \uu_\ee  	
		= \\&=
	\underbrace{
            \ee\int_0^T\int_{\Omega_\ee} \nabla \varphi\otimes \omega^{i,\ee}: \nabla \uu_\ee 
         }_{=:\I^\ee_1}
		+
	\underbrace{ 
                \ee\int_0^T\int_{\Omega_\ee} \nabla \omega^{i,\ee}: \nabla(\varphi \uu_\ee )
        }_{=:\I^\ee_2}
	\underbrace{
        - \ee\int_0^T\int_{\Omega_\ee} \nabla \omega^{i,\ee}: \nabla\varphi\otimes \uu_\ee
        }_{=:\I^\ee_3}.
\end{aligned}
\end{equation}
where $(\nabla \varphi\otimes \omega^{i,\ee})_{jk} = \partial_j \varphi \omega^{i,\ee}_k $, with $j,k =1,2$. We estimate the first term on the left-hand side~of \eqref{Iee_Thm:conv_of_u} as follows:
	\begin{equation*}
	\begin{aligned}
            |\I^\ee_1|=
		\Big|\ee\int_0^T\int_{\Omega_\ee} \nabla \varphi\otimes \omega^{i,\ee}: \nabla \uu_\ee 
		\Big|
		&\leq
            \ee 
		\| \nabla \varphi \|_{L^2(0,T; L^\infty(\Omega))}
		\| \omega^{i,\ee} \|_{L^2( \Omega_\ee)}
		\| \nabla \uu_\ee \|_{L^2((0,T)\times \Omega_\ee)}\\
            &\leq
            \ee 
		\| \nabla \varphi \|_{L^2(0,T; L^\infty(\Omega))}
		\| \omega^{i,\ee} \|_{L^2( \Omega_\ee)}
		\EE_0
            .
	\end{aligned}
	\end{equation*}
        Since $\Omega$ is bounded, there exists an integer $N\in      \mathbb N$ such that $\Omega\subseteq [-N-1/2,N+1/2]^2$. Consequently, by applying a substitution,  $\| \omega^{i,\ee} \|_{L^2( \Omega_\ee)}
        = \| \omega^{i,\ee} \|_{L^2( \Omega)}\leq
        \| \omega^{i,\ee} \|_{L^2([-N-1/2,N+1/2])}=(2N+1)^2\| \omega^i \|_{L^2(Y^*)}$, which allows us to conclude that the first term on the left-hand side of \eqref{Iee_Thm:conv_of_u} vanishes as $\ee$ goes to $0$.
	Moreover, since $\omega^{i,\ee}$ satisfies equation \eqref{eq:omega_eps} and $\uu_\ee $ is divergence free, 
	we observe that
	\begin{equation*}
	\begin{aligned}
            \I^\ee_2=
	 	 \ee\int_0^T\int_{\Omega_\ee} \nabla \omega^{i,\ee}: \nabla(\varphi \uu_\ee )
	 	 &= \int_0^T\int_{\Omega_\ee} \pi^\ee \Div\,(\varphi \uu_\ee) +\int_0^T \int_{\Omega_\ee} 
	 	 \frac{\uu_{\ee,i}}{\ee} \varphi
            =  \int_0^T\int_{\Omega_\ee} \pi^\ee \nabla \varphi \cdot \uu_\ee 
	   +\int_0^T \int_{\Omega} \frac{\tilde \uu_{\ee,i}}{\ee} \varphi.
	 \end{aligned}
      \end{equation*}
      We hence recall the Poincar\'e inequality of \Cref{Poincare_ineq}, which implies 
	$  \| \uu_\ee \|_{L^2((0,T)\times \Omega_\ee)} 
           \leq C_{2,\Omega} \mathcal{E}_0 \ee 
            $
         as in \eqref{eq:estimate-of-tildeu-ee-Poinc}, as well as we remark that $\| \pi^{i,\ee} \|_{L^2(\Omega_\ee)} \leq 
         \| \pi^{i,\ee} \|_{L^2([-N, N]^2)}
         \leq N \|  \pi^i \|_{L^2(Y^*)}$. Since any component $\tilde u_{\ee,i}$, with $i=1,2$, weakly converges towards the $i$-component $u_i$ of $u$ in $L^2((0,T)\times \Omega)$, as $\ee\to 0$, we deduce that
         \begin{equation*}
             \lim_{\ee \to 0}\I_2^\ee = \int_0^T \int_\Omega e^i \cdot u \varphi
             =\int_0^T \int_\Omega u_i \varphi.
         \end{equation*}
       To conclude the analysis of $\mathcal I^\ee$ in \eqref{Iee_Thm:conv_of_u}, it remains hence to estimate $\I^\ee_3$. We observe that
	\begin{equation*}
        \begin{aligned}
            |\I_3^\ee| 
            &=
		\Big|\ee\int_0^T\int_{\Omega_\ee} \nabla \omega^{i,\ee}: \nabla\varphi\otimes \uu_\ee \Big|		
		\leq T^\frac{1}{2}\ee \| \nabla \omega^{i,\ee} \|_{L^2( \Omega_\ee)}
		\| \nabla \varphi \|_{L^\infty((0,T)\times \Omega)}
		\| \uu_{\ee} \|_{L^2((0,T)\times \Omega_\ee)}
		\\
            &\leq C_{2,\Omega} T^\frac{1}{2} 
            \EE_0\ee^2 
             \| \nabla \omega^{i,\ee} \|_{L^2( \Omega_\ee)}
            \| \nabla \varphi \|_{L^\infty((0,T)\times \Omega)}
		.
        \end{aligned}
	\end{equation*}
        Once more, we have that $ \ee \| \nabla \omega^{i,\ee} \|_{L^2( \Omega_\ee)} \leq (2N+1)^2 \| \nabla_y \omega^{i} \|_{L^2(Y^*)}$ by substitution, thus also the term $\I_3^\ee$ vanishes as $\ee\to 0$.
        Summarizing, we have determined that 
        \begin{equation}\label{est:Darcy-Ie}
              \lim_{\ee \to 0}\I^\ee =
              \lim_{\ee \to 0}\I_2^\ee = \int_0^T \int_\Omega e^i \cdot u.
         \end{equation}
        Similarly, we claim that also $\mathcal{V}^\ee$ satisfies 
        \begin{equation}\label{est:Darcy-Ve}
              \lim_{\ee \to 0}\mathcal{V}^\ee  = \int_0^T \int_\Omega e^i \cdot H
              = \int_0^T \int_\Omega H_i.
         \end{equation}
        Next, before dealing with $\I\I^\ee$ and the 
        pressure $\pre_\ee$ in \eqref{to_control_thm:conv_of_u}, we first 
	estimate $\I\I\I^\ee$ and $\I\mathcal{V}^\ee$. 
        We remark that
	\begin{equation*}
	\begin{aligned}
		|\I\I\I^\ee| 
            &= 
		\ee \Big| \int_0^T\int_{\Omega_\ee} (\nabla d_\ee)^T \Delta d_\ee :
		\varphi \omega^{i,\ee}  \Big| \\ 
		&\leq 
		\ee 
		\| \nabla 	 d_\ee 	\|_{L^\infty(0,T;L^2(\Omega_\ee))} 
		\| \Delta	 d_\ee 	\|_{L^2((0,T)\times \Omega_\ee)}
		\| \varphi \|_{L^2(0,T; L^\infty( \Omega))}
		\| \omega^{i,\ee} \|_{L^\infty(\Omega_\ee)}\\
            &\leq 
		\ee 
		\EE_0^2
            \| \varphi  \|_{L^2(0,T; L^\infty( \Omega))}
		\| \omega^{i,\ee} \|_{L^\infty(\Omega_\ee)}.
    \end{aligned}
    \end{equation*}
    Given that $H^2(Y^*)$ is embedded into $L^\infty(Y^*)$ and that $\| \omega^{i,\ee} \|_{L^\infty(\Omega_\ee)} = \| \omega^{i} \|_{L^\infty(Y^*)}\leq C_{Y^*} \| \omega^{i} \|_{H^2(Y^*)}$, for a constant $C_{Y^*} $ depending only on $Y^*$,  we deduce that  
    \begin{equation}\label{est:Darcy-IIIe}
        \lim_{\ee \to 0} \I\I\I^\ee = 0
    \end{equation}
    In order to deal with $\I \mathcal{V}^\ee$ in \eqref{to_control_thm:conv_of_u}, we recall that $\ee F_\ee \to F$ in $L^2((0,T)\times \Omega)$, as $\ee \to 0$, as well as the rapidly oscillating functions $\omega^{i, \ee}$ are uniformly bounded in $L^\infty(\Omega)$. Employing a standard argument for the weak-* limit of rapidly oscillating periodic functions, we deduce that 
    $\omega^{i,\ee}\stackrel{\ast}{\rightharpoonup} \frac{1}{|Y|}\int_{Y^*}\omega^i(y) dy$ in $L^\infty(\Omega)$. We hence obtain
     \begin{equation}\label{est:Darcy-IVe}
              \lim_{\ee \to 0}\I\mathcal{V}^\ee =
              \frac{1}{|Y|}\int_{Y^*} \omega^i(y)dy \cdot \int_0^T \int_\Omega F(t,x) \varphi(t,x) dx dt.
    \end{equation} 
    We now focus on $\I\I^\ee$ 
    of \eqref{to_control_thm:conv_of_u} and we invoke 
    \Cref{prop:pressure} about the extensions $P_\ee $ of the pressures, which are in $L^2(0,T;\lbar^p_0(\Omega))$ for any $p\in (1,2)$. The gradients $\nabla \pre_\ee $ and $\nabla (P_{\ee}|_{\Omega_\ee})$ coincides in $L^2(0,T;W^{-1, p}(\Omega_\ee))^2$, for any
    $p\in (1,2)$.  Moreover, thanks to \eqref{est:Pee-proposition}, there exists a function $P\in L^2(0,T;\lbar^p_0(\Omega))$ such that $\ee P_\ee \rightharpoonup P$ in $L^2(0,T;L^p(\Omega))$ (up to a subsequence). Hence, we aim to show that
    \begin{equation}\label{est:Darcy-IIe}
        \lim_{\ee \to 0}\I\I^\ee = 
        - \int_0^T \int_{\Omega } 	P \big(\mathbb{B} \nabla \varphi \big)_i
        =
        - 
        \sum_{k =1}^2
        \int_0^T \int_{\Omega } 	P \,\mathbb{B}_{ik} \frac{\partial \varphi}{\partial x_k} ,
    \end{equation}
    where the matrix $\mathbb B$ defined in \eqref{eq:matrices-AB} has components 
    $\mathbb B_{ij} =  \frac{1}{|Y|}
   \int_{Y^*} 
   \nabla_y \omega^{i}(y) 
   :
   \nabla_y \omega^{j}(y) 
   dy$, with $i,j = 1,2$.

    \noindent 
    We remark that the function $\varphi \omega^{i,\ee}$
    belongs to $L^2(0,T;W^{1,p'}_0(\Omega_\ee))^2$ ($p' = p/(p-1)$), for any $p\in (1,2)$, since $\varphi \in \mathcal{D}((0,T)\times \Omega)$ and $\omega^{i,\ee} \in H^2(\Omega_\ee)$. Since $\omega^{i,\ee}$ is identically $0$ in $\Omega \setminus \Omega_\ee$ then $\mathcal{R}_{p',\ee} (\varphi  \omega^{i,\ee}) = \varphi  \omega^{i,\ee}$. Therefore we can recast $\I\I^\ee$ into
	\begin{equation*}
		\begin{aligned}
	       \I\I^\ee  
             &= 
             -\ee\int_0^T \int_{\Omega_\ee} \pre_\ee \Div(\varphi  \omega^{i,\ee})
             =
             -\ee\int_0^T \int_{\Omega_\ee} \pre_\ee \Div \mathcal{R}_{p',\ee} \big(\varphi  \omega^{i,\ee}\big)
             \\
             &=
             -\ee\int_0^T \int_{\Omega} P_\ee \Div(\varphi  \omega^{i,\ee})
             =
             -\ee\int_0^T \int_{\Omega} 
             P_\ee \nabla \varphi \cdot 
             \omega^{i,\ee}
             +
              P_\ee \varphi \underbrace{\Div\, \omega^{i,\ee}}_{=0}
		=
            -\ee \sum_{k=1}^2 
			 \int_0^T\int_{\Omega}  P_\ee
			 \frac{\partial \varphi}{\partial x_k}
               \omega^{i,\ee}_k.
		\end{aligned}
	\end{equation*}
        We next introduce the average-free functions $\chi^{i,\ee}_k(x) =   \omega^{i,\ee}_k(x) - \frac{1}{|\Omega|}
                \int_{\Omega}
                  \omega^{i,\ee}_k(z)dz$, which are still uniformly bounded in $L^\infty(\Omega)$ and (up to a subsequence) they weakly-$*$ converge to $0$, as $\ee \to 0$. Hence
        \begin{equation}\label{thm:II_thm:conv_of_u}
              \I\I^\ee 
              =
              \underbrace{-\ee \sum_{k=1}^2 
			 \int_0^T\int_{\Omega}  P_\ee
			 \frac{\partial \varphi}{\partial x_k}
              \chi^{i,\ee}_k
              }_{\I\I^\ee _1}
               \underbrace{-\ee \sum_{k=1}^2
              \bigg(
                \frac{1}{|\Omega|}
                \int_{\Omega}
                \omega^{i,\ee}_k(z)dz
              \bigg)
			 \int_0^T\int_{\Omega}  P_\ee
			 \frac{\partial \varphi}{\partial x_k}
        }_{\I\I^\ee _2}.
        \end{equation}               
	We first show that $\I\I^\ee _1$ vanishes as $\ee \to 0$, thus \eqref{est:Darcy-IIe} will be 
        determined by the limit of $\I\I^\ee_2$. 
        For a fixed $p\in (1,2)$, we invoke \Cref{lemma:Bogovskii} about the Bogovskii operator, in order to determine a (non-unique) function
	$\psi^{i,\ee}_k$ in $W^{1,p'}_0(\Omega)$ satisfying
	\begin{equation*}
			\begin{aligned}
                \Div\, \psi^{i,\ee}_k = \chi^{i,\ee}_k \quad \text{in}\quad\Omega
                \qquad \text{and}\qquad 
                \| \nabla  \psi^{i,\ee}_k \|_{L^{p'}(\Omega)} 
                &\leq 
                C_{p', \Omega} 
                \| \chi^{i,\ee}_k \|_{L^{p'}(\Omega)}
                \leq 
                C_{p', \Omega} |\Omega|^{1-\frac{1}{p}}
                \| \chi^{i,\ee}_k \|_{L^{\infty}(\Omega)}\\
                &
                \leq 
                2 C_{p', \Omega} |\Omega|^{1-\frac{1}{p}} \| \omega^i \|_{L^\infty(Y^*)}.
                \end{aligned}
	\end{equation*}
        We remark that the last term of the above inequalities does not depend on $\ee>0$. 
	Thus, up to a subsequence, there exists a function $ \psi^i_k\in W^{1, p'}_0(\Omega)$ 
	such that $\psi^{i,\ee}_k \to  \psi^i_k$ 
	strongly in $L^{p'}(\Omega)$ and $\nabla \psi^{i,\ee}_k \rightharpoonup \nabla  \psi^i_k$ weakly in $L^{p'}(\Omega)$. Moreover $\Div\, \psi^i_k = 0$, since $ \chi^{i,\ee}_k\stackrel{\ast}{\rightharpoonup} 0$. Next, we observe that
	\begin{equation*}
        \begin{aligned}
            \I \I^\ee_1
            &=
		-\ee
            \sum_{k=1}^2
            \int_0^T \int_{\Omega}  P_\ee  
		\frac{\partial \varphi}{\partial x_k}\Div\,\psi^{i,\ee}_k 
            =
		-\ee
            \sum_{k=1}^2
            \int_0^T \int_{\Omega}  P_\ee  
		\frac{\partial \varphi}{\partial x_k}\big( \Div\,\psi^{i,\ee}_k - \underbrace{\Div\,\psi^{i}_k}_{=0} 
            \big)
		\\
            & = 
            -\ee
            \sum_{k=1}^2
            \int_0^T   \int_{\Omega}
		P_\ee \,\Div\Big(\frac{\partial \varphi}{\partial x_k} \big(\psi^{i,\ee}_k-\psi^{i}_k\big)\Big)
		+
		\ee
            \sum_{k=1}^2
            \int_0^T \int_{\Omega}
		P_\ee  \frac{\partial \nabla \varphi}{\partial x_k}  
            \cdot 
            \big(\psi^{i,\ee}_k-\psi^{i}_k\big).
        \end{aligned}
        \end{equation*}
        Hence, by invoking to the relation  \eqref{eq:dual-form-of-Pe-and-pe} of \Cref{prop:pressure}, we obtain 
        \begin{equation*}
	\begin{aligned}		
             \I \I^\ee_1&=
		\underbrace{-
            \ee\sum_{k=1}^2
            \int_0^T   \int_{\Omega_\ee}
		\pre_\ee\,\Div\,\mathcal{R}_{p',\ee} \Big(\frac{\partial \varphi}{\partial x_k}       
             \big(\psi^{i,\ee}_k-\psi^{i}_k\big)
            \Big)}_{\I\I_{1,1}^\ee}
		+
            \underbrace{
		\ee\sum_{k=1}^2
            \int_0^T \int_{\Omega}
		P_\ee  \nabla  \frac{\partial \varphi}{\partial x_k}  
            \cdot 
            \big(\psi^{i,\ee}_k-\psi^{i}_k\big)
            }_{\I\I_{1,2}^\ee}.
	\end{aligned}
	\end{equation*}
        The term $\I \I_{1,2}^\ee$ vanishes as $\ee\to 0$ since $|\I \I_{1,2}^\ee| \leq 2\ee \| P_\ee \|_{L^2(0,T;L^p(\Omega))}\|\nabla^2 \varphi \|_{L^2(0,T;L^\infty(\Omega))} \|  \psi^{i,\ee}_k-\psi^{i}_k \|_{L^{p'}(\Omega)}$ and $\psi^{i,\ee}_k\to \psi^{i}_k$ in $L^{p'}(\Omega)$. 
        Additionally, thanks to the Stokes equation \eqref{maineq1}, we have that
	\begin{equation*}
	\begin{aligned} 
		  \I\I_{1,1}^\ee
            =   
		\ee
            \sum_{k=1}^2
            \bigg\{
		-
            \int_0^T   \int_{\Omega}
		\nabla \uu_\ee 
		:\nabla  \mathcal{R}_{p',\ee} 
		\Big(\frac{\partial \varphi}{\partial x_k} \big(\psi^{i,\ee}_k-\psi^{i}_k\big)\Big)
		-
            \int_0^T   \int_{\Omega}
		(\nabla d_\ee)^T \Delta d_\ee 
		\cdot \mathcal{R}_{p',\ee} 
		\Big(\frac{\partial \varphi}{\partial x_k}
		  \big(\psi^{i,\ee}_k-\psi^{i}_k\big)
            \Big) +\\
            +
            \int_0^T   \int_{\Omega}
		F_\ee 
		\cdot 
            \mathcal{R}_{p',\ee} 
		\Big(\frac{\partial \varphi}{\partial x_k} \big(\psi^{i,\ee}_k-\psi^{i}_k\big)\Big)
            +
            \int_0^T   \int_{\Omega}
		\nabla H_\ee 
		:\nabla  \mathcal{R}_{p',\ee} 
		\Big(\frac{\partial \varphi}{\partial x_k} \big(\psi^{i,\ee}_k-\psi^{i}_k\big)\Big)
            \bigg\},
	\end{aligned} 
        \end{equation*}
        therefore
        \begin{align*}
            \big|\I \I_{1,1}^\ee\big| \leq 
            \sum_{k=1}^2
            \bigg\{
            \Big( 
                \| \nabla \uu_\ee  \|_{L^2((0,T)\times \Omega_\ee )} + 
                \| \nabla H_\ee  \|_{L^2((0,T)\times \Omega_\ee )} 
            \Big)
            \ee 
            |\Omega |^{\frac{1}{p}-\frac{1}{2}}
            \Big\| 
                \nabla 
                \mathcal{R}_{p',\ee} 
		\Big(\frac{\partial \varphi}{\partial x_k} \big(\psi^{i,\ee}_k-\psi^{i}_k\big)\Big)
            \Big\|_{L^2(0,T;L^{p'}(\Omega_\ee))} 
            +
            \\
            +
             \| \nabla d_\ee  \|_{L^\infty(0,T;L^2(\Omega_\ee))}
             \| \Delta d_\ee  \|_{L^2(0,T;L^2(\Omega_\ee))}
             \ee
             \Big\| 
                \mathcal{R}_{p',\ee} 
		\Big(\frac{\partial \varphi}{\partial x_k} \big(\psi^{i,\ee}_k-\psi^{i}_k\big)\Big)
            \Big\|_{L^2(0,T;L^\infty(\Omega_\ee))} 
            +
            \\
            +
            \ee \| F_\ee \|_{L^2((0,T)\times \Omega_\ee)}
            |\Omega |^{\frac{1}{p}-\frac{1}{2}}
            \Big\| 
                \mathcal{R}_{p',\ee} 
		\Big(\frac{\partial \varphi}{\partial x_k} \big(\psi^{i,\ee}_k-\psi^{i}_k\big)\Big)
            \Big\|_{L^2(0,T;L^{p'}(\Omega_\ee))} \bigg\}
        \end{align*}
        Since $p'>2$, $W^{1,p'}(\Omega) \hookrightarrow L^\infty(\Omega)$, we observe that
        \begin{align*}
             \Big\| 
                \mathcal{R}_{p',\ee} 
		&
            \Big(\frac{\partial \varphi}{\partial x_k} 
            \big(\psi^{i,\ee}_k-\psi^{i}_k\big)\Big)
            \Big\|_{L^2(0,T;L^\infty(\Omega_\ee))} 
             =
            \Big\|  
               \widetilde{ \mathcal{R}_{p',\ee} }
		\Big(\frac{\partial \varphi}{\partial x_k} \big(\psi^{i,\ee}_k-\psi^{i}_k\big)\Big)
            \Big\|_{L^2(0,T;L^\infty(\Omega))} \\
            &\leq 
            \tilde C_{p', \Omega}
             \Big\| 
               \nabla  \widetilde{\mathcal{R}_{p',\ee} }
		\Big(\frac{\partial \varphi}{\partial x_k} \big(\psi^{i,\ee}_k-\psi^{i}_k\big)\Big)
            \Big\|_{L^2(0,T;L^{p'}(\Omega))}
            =
            \tilde C_{p', \Omega}
             \Big\| 
               \nabla  \mathcal{R}_{p',\ee} 
		\Big(\frac{\partial \varphi}{\partial x_k} \big(\psi^{i,\ee}_k-\psi^{i}_k\big)\Big)
            \Big\|_{L^2(0,T;L^{p'}(\Omega_\ee))},
        \end{align*}
        where $\tilde C_{p', \Omega}>0$ depends only on $p'$ (i.e. on $p$) and $\Omega$. The notation $\widetilde{ \mathcal{R}_{p',\ee} }$ means that we apply the operator $\mathcal{R}_{p',\ee}$ and then we extend the resulting function by $0$ within each particle. 
        Hence, we invoke the uniform estimates in \Cref{rmk:uniform-estimates-norms-solutinos} and the ones in \Cref{lemma:restriction_operator} to gather
        \begin{align*}
            \big|\I\I_{1,1}^\ee\big| \leq 
            C_{p',\Omega}
            \bigg(
                \EE_0
                |\Omega |^{\frac{1}{p}-\frac{1}{2}}
                +
                \tilde C_{p',\Omega}
                \EE_0^2
            \bigg)
            \sum_{k=1}^2
            \bigg(
            \Big\| 
                    \Big(\frac{\partial \varphi}{\partial x_k} \big(\psi^{i,\ee}_k-\psi^{i}_k\big)\Big)
            \Big\|_{L^2(0,T;L^{p'}(\Omega_\ee))} 
            + \\
            +
            \ee
            \Big\| 
                    \nabla 
                    \Big(\frac{\partial \varphi}{\partial x_k} \big(\psi^{i,\ee}_k-\psi^{i}_k\big)\Big)
            \Big\|_{L^2(0,T;L^{p'}(\Omega_\ee))}
            \bigg).
        \end{align*}
        Since $\psi^{i,\ee}_k\to \psi^{i}_k$ in $L^{p'}(\Omega)$ and $\nabla (\psi^{i,\ee}_k-\psi^{i}_k)$ is uniformly bounded in $L^{p'}(\Omega)$, we finally deduce that $\I\I_{1,1}^\ee\to 0$, as $\ee \to 0$.
    
	\noindent 
        To conclude the proof of \eqref{est:Darcy-IIe}, we shall finally address the limit of $\I\I^\ee_2$ in \eqref{thm:II_thm:conv_of_u}. We recall once more that $\ee P_\ee \rightharpoonup P$ in $L^2(0,T;L^p(\Omega))$ and $\omega^{i,\ee}\stackrel{\ast}{\rightharpoonup} \frac{1}{|Y|}\int_{Y^*}\omega^i(y) dy$. Thus
        \begin{align*}
            \lim_{\ee \rightarrow 0} \I\I^\ee_2 
            &= 
            -
            \lim_{\ee \rightarrow 0} 
            \sum_{k=1}^2
              \bigg(
                \frac{1}{|\Omega|}
                \int_{\Omega}
                \omega^{i,\ee}_k(z)dz
              \bigg)
			 \int_0^T\int_{\Omega}  
              \ee 
              P_\ee
              (t,x)
			 \frac{\partial \varphi}{\partial x_k}
            (t,x)
             dx dt\\
              & = 
              -
              \sum_{k=1}^2
              \bigg(
                \frac{1}{|Y|}
                \int_{Y^*}
                \omega^{i}_k(y)dy
              \bigg)
			 \int_0^T
              \int_{\Omega}  
              P(t,x)
			 \frac{\partial \varphi}{\partial x_k}(t,x)dxdt\\
              & = 
              -
			 \int_0^T
              \int_{\Omega}  
              P(t,x)
              \sum_{k=1}^2
              \bigg\{
              \bigg(
                \frac{1}{|Y|}
                \int_{Y^*}
                \omega^{i}(y)\cdot e^k dy
              \bigg)
			 \frac{\partial \varphi}{\partial x_k}
              (t,x)
              \bigg\}
              dx dt.
        \end{align*}
        Hence, recalling that $\omega^i$ satisfies \eqref{eq:omega}  we finally obtain
        \begin{align*}
             \lim_{\ee \rightarrow 0} \I\I^\ee_2 
             &= 
             -
			 \int_0^T
              \int_{\Omega}  
              P(t,x)
              \sum_{k=1}^2
              \bigg\{
              \bigg(
                \frac{1}{|Y|}
                \int_{Y^*}
                \nabla_y\omega^{i}(y): \nabla_y\omega^{k}(y) dy
              \bigg)
			 \frac{\partial \varphi}{\partial x_k}
              (t,x)
              \bigg\}
              dx dt\\
              &=
              -
			 \int_0^T
              \int_{\Omega}  
              P(t,x)
               \big( \mathbb B \nabla \varphi(t,x)\big)_i
              dxdt.
        \end{align*}
        This concludes the proof of \eqref{est:Darcy-IIe}.

        \noindent 
	Summarizing, we couple the identities in \eqref{est:Darcy-Ie}, \eqref{est:Darcy-Ve}, \eqref{est:Darcy-IIIe},
        \eqref{est:Darcy-IVe},  and
        \eqref{est:Darcy-IIe} to deduce that $(\uu,P)$ satisfies 
	\begin{equation}\label{eq:equation-for-ui-proof}
		\uu_i + \sum_{k=1}^2 \mathbb{B}_{ik} \frac{\partial P}{\partial x_k} = 
            \frac{1}{|Y|}
            \int_{Y^*}\omega^i(y)dy \cdot F + H_i,\qquad \Div\,\uu = 0
	\end{equation}
        in 	$\mathcal{D}'((0,T)\times \Omega)$, for $i = 1,2$. 
        Moreover, since both $u$ and the right-hand side~of \eqref{eq:equation-for-ui-proof} belong to $L^2((0,T)\times \Omega)$, 
        also $\mathbb B \nabla P$ is in $ L^2((0,T)\times \Omega)$. Thus $P$ belongs to $ L^2(0,T;L^p( \Omega))$ with $ \nabla P$ in $ L^2((0,T)\times \Omega)$, which implies in particular that $P$ is also in $L^2((0,T)\times \Omega)$ (although $\ee P_\ee$ does not necessarily weakly converge towards $P$ in this space).
	  Finally
	\begin{equation*}
			\int_0^T \int_{\Omega} \uu \cdot \nabla \psi  = 
			\lim_{\ee\rightarrow 0}\frac 1\ee \underbrace{
			\int_0^T \int_{\Omega}\tilde \uu_\ee \cdot \nabla \psi}_{=0} = 0,\qquad 
			\text{for any}\quad \psi\in \mathcal{C}^\infty([0,T]\times \overline{\Omega}). 
	\end{equation*}	 
	hence the boundary condition $\uu(t,\cdot) \cdot \nu = 0$ on 
        $\partial \Omega$ 
        is satisfied in $H^{-1/2}(\partial \Omega)$, for a.e.~$t\in (0,T)$. 
        This concludes the proof of \Cref{prop:Darcy}.
\end{proof}

\section{Homogenisation of the director equation}\label{sec:director}

\noindent 
In this section, we address the directors $d_{\varepsilon}$ and conclude the proof of \Cref{main_thm}, specifically focusing on part (iii). Our approach mirrors that of \Cref{sec:Darcy}. We consider a general sequence $(\ee_l)_{l\in \mathbb{N}}$, characterized by its monotonic decrease and convergence to zero as $l \to \infty$. We demonstrate the existence of a subsequence $(\ee_{l_k})_{k \in \mathbb{N}}$, for which the function $\tilde{d}_{\varepsilon_{l_k}}$ weakly-$\ast$ converges to a vector function $d$ in $L^\infty((0,T) \times \Omega)$. 
Notably, $d$ may initially depend on the specific subsequence $(\ee_{l_k})_{k \in \mathbb{N}}$. Nevertheless, as detailed in the forthcoming proposition, $d$ is ultimately identified as the solution to equation \eqref{eq:d}. Thus, the arbitrariness of the sequence $(\ee_l)_{l\in \mathbb{N}}$ leads us to the general conclusion that $\tilde{d}_\ee \rightharpoonupstar d$, as $\varepsilon \to 0$. 

\noindent 
We recall from the assumptions of \Cref{main_thm} that the initial data $\tilde{d}_{\ee, \mathrm{in}}$ weakly-$\ast$ converges to $d_{\mathrm{in}}$ in $L^\infty(\Omega)$ and that $d_{\mathrm{in}}$ belongs to $H^1(\Omega)\cap L^\infty(\Omega)$ from \Cref{rmk:d-in-is-also-in-H1}. We prove the following result.
\begin{prop}\label{prop:convergence-of-d}
    For any sequence $(\ee_l)_{l\in \mathbb N}\subset ]0, \infty[$ with $\ee_l \to 0$, there exists a subsequence  $(\ee_{l_k})_{k\in \mathbb N}$ such that $ \tilde d_{l_k} \rightharpoonupstar d$ in $L^\infty((0,T)\times \Omega)$. The vector $d$ satisfies
    \begin{equation*}
        d\in L^\infty((0,T)\times \Omega) \cap L^\infty(0,T;H^1(\Omega))\quad \text{with}\quad  
        \partial_t d,\,\Delta d \in L^2((0,T)\times \Omega)
    \end{equation*}
    and is the unique weak solution of 
    \begin{equation}\label{eq:d-equation-in-proposition}
        \partial_t d - \Div\Big( \frac{1}{\theta }\mathbb A \nabla d \Big) = - (|d|^2-1)d\quad 
        \text{in }(0,T)\times \Omega,
        \qquad 
        \nu \cdot \mathbb A \nabla d = 0 \quad \text{on }(0,T)\times \partial \Omega,
        \qquad 
        d_{|t = 0} = d_{\rm in}\quad\text{in } \Omega,
    \end{equation}
    where $\nu$ is the normal vector on $\partial \Omega$ and $\mathbb{A} \in \mathbb{R}^{2 \times 2}$ is the matrix defined in \eqref{eq:matrices-AB}.
\end{prop}
\begin{proof}
With an abuse of notation, we omit the index $l\in \mathbb N$ from the sequence $(\ee_l)_{l \in \mathbb{N}}$ and any associated subsequences, and instead represent $\ee_l \to 0$ as $l\to \infty$ simply by $\ee \to 0$.  The functions $\tilde d_\ee$ have norms in $L^\infty((0,T)\times \Omega)$ uniformly bounded by $1$, in accordance with the maximum principle satisfied by \Cref{maineq3} and with $\| d_{\ee,\rm in}\|_{L^\infty(\Omega_\ee)}\leq 1$, for any $\ee>0$. Similarly, the extensions of the gradients $\widetilde{\nabla d_\varepsilon}$ are uniformly bounded in $L^\infty(0,T; L^2(\Omega))$, thanks to the uniform estimates of \Cref{rmk:uniform-estimates-norms-solutinos}. Consequently, there exist a vector function $d \in L^\infty((0,T) \times \Omega)$ and a tensor function $\Xi \in L^\infty(0,T; L^2(\Omega))$ such that 
\begin{equation}\label{eq:weak-convergence-of-de-and-nablade}
    \tilde d_{\ee} \stackrel{\ast}{\rightharpoonup} \theta\, d,
    \quad 
    \text{weakly-}\ast \quad \text{in}\quad  L^\infty((0,T)\times \Omega),
    \qquad 
    \widetilde{\nabla d_{\ee}}
    \stackrel{\ast}{\rightharpoonup}
    \theta\,
    \Xi 
    \quad 
    \text{weakly-}\ast \quad \text{in}\quad  L^\infty(0,T; L^2(\Omega)),
\end{equation}
up to a subsquence, as $\ee \to 0$. We first assert that also the following weakly-$\ast$ convergence holds true:
\begin{equation}\label{eq:weak-convergence-d^3-d}
    (|\tilde d_\ee|^2-1)\tilde d_\ee \rightharpoonupstar \theta (|d|^2-1) d 
    \qquad \text{weakly-}\ast \quad \text{in}\quad  L^\infty((0,T)\times \Omega),
\end{equation}
up to a subsequence, as $\ee \to 0$. To this end, we invoke \Cref{lemma:compactness_Omega_eps} and \Cref{thm:Aubin-Lions-Linfty} and we observe that \Cref{maineq3} implies
\begin{equation*}
\begin{aligned}
	&\| \partial_t d_{\ee} \|_{L^2(0,T;L^\frac{4}{3}(\Omega_{\ee}))}
        =
        \| 
            - u_\ee \cdot \nabla d_\ee + \Delta d_\ee -  (| d_{\ee} |^2-1)d_{\ee}
        \|_{L^2(0,T;L^\frac{4}{3}(\Omega_{\ee}))}
        \\
	&\leq 
	\| \uu_{\ee} 
        \|_{L^2(0,T;L^4(\Omega_{\ee}))}
        \| \nabla d_{\ee} 
        \|_{L^\infty(0,T; L^2(\Omega_{\ee}))}
	\!+\! 
        \| 
            \Delta  d_{\ee} 
        \|_{L^2(0,T;L^\frac{4}{3}(\Omega_{\ee}))}
	\!\!+
        \big( 
            \| d_\ee \|_{L^\infty((0,T)\times \Omega_\ee)}^3 
            \!+\!
            \| d_\ee \|_{L^\infty((0,T)\times \Omega_\ee)}
        \big)
	T^{\frac{1}{2}}
        |\Omega_\ee|^{\frac{3}{4}}\\
	&\leq 
	\| \tilde \uu_{\ee} \|_{L^2(0,T;L^4(\Omega))}
	\| \nabla d_{\ee} 
        \|_{L^\infty(0,T; L^2(\Omega_{\ee}))}
	+ 
        |\Omega |^{\frac{1}{4}}
        \| \Delta  d_{\ee} \|_{L^2(0,T;L^2(\Omega_{\ee}))}
	+
        2 
        T^{\frac{1}{2}}
        |\Omega|^{\frac{3}{4}}
        .
\end{aligned}
\end{equation*}
Since $H^1_0(\Omega)$ is continuously embedded into $L^4(\Omega)$ thanks to Rellich-Kondrachov Theorem, there exists a constant $C_\Omega>0$, depending only on $\Omega$, such that $ \| \tilde \uu_{\ee} \|_{L^2(0,T;L^4(\Omega))} \leq C_\Omega \| \nabla \tilde \uu_{\ee} \|_{L^2((0,T)\times \Omega)} = 
C_\Omega \| \nabla \uu_\ee \|_{L^2( (0,T) \times \Omega_\ee)}$. Hence, invoking the uniform estimates of \Cref{rmk:uniform-estimates-norms-solutinos}, we gather that
\begin{equation*}
\begin{aligned}
    \| \partial_t d_{\ee} 
    \|_{L^2(0,T;L^\frac{4}{3}(\Omega_{\ee}))}
    &\leq 
        C_\Omega\| \nabla \uu_{\ee} \|_{L^2((0,T)\times \Omega_\ee)}
	\| \nabla d_{\ee} 
        \|_{L^\infty(0,T; L^2(\Omega_{\ee}))}
	+ 
        |\Omega |^{\frac{1}{4}}
        \| \Delta  d_{\ee} \|_{L^2(0,T;L^2(\Omega_{\ee}))}
	+
        2 
        T^{\frac{1}{2}}
        |\Omega|^{\frac{3}{4}}
    \\
    &\leq 
    C_\Omega \EE_0^2 + |\Omega|^\frac{1}{4} \EE_0 + 
    2
    T^{\frac{1}{2}}
    |\Omega|^{\frac{3}{4}}.
\end{aligned}
\end{equation*}
This implies  that the  $L^2(0,T;L^\frac{4}{3}(\Omega_{\ee}))$-norms of $\partial_t d_\ee$ are uniformly bounded  in $\ee>0$. In particular, each component $d_{\ee,i}$, $i = 1,2$, of $d_\ee = (d_{\ee,1}, d_{\ee,2})^T$ satisfies the assumptions of \Cref{lemma:compactness_Omega_eps} and \Cref{thm:Aubin-Lions-Linfty} with $f_l = d_{\ee,i}$ (i.e.~$f_l = d_{\ee_l,i}$), $q = 2$, $s = 4/3$ and $r = 2$. Next, we consider a general vector function $\varphi $ with components in $\mathcal{C}^\infty([0,T]\times \bar{\Omega})$ and we set $g_l := \varphi \cdot d_\ee =  \varphi \cdot d_{\ee_l}$. Thanks to \eqref{eq:weak-convergence-of-de-and-nablade} and up to a subsequence, $g_l \rightharpoonupstar \theta \,d\cdot \varphi $ in $L^\infty((0,T)\times \Omega)$ (hence also weakly-$\ast$ in $L^1(0,T;L^\infty(\Omega))$). A direct application of \Cref{thm:Aubin-Lions-Linfty} with $m = 2$ and of  \Cref{lemma:compactness_Omega_eps} yields
\begin{equation*}
    \lim_{\ee \to 0} 
    \int_0^T \int_\Omega (|\tilde d_\ee|^2 -1) \tilde  d_\ee \cdot  \varphi 
    = 
    \lim_{\ee \to 0} 
    \int_0^T \int_{\Omega_\ee} 
    \Big(
     \sum_{i = 1}^2 | d_{\ee,i}|^2 -1
     \Big)  d_\ee \cdot  \varphi
    =
    \theta     
    \int_0^T \int_{\Omega} \Big(  \sum_{i = 1}^2 |d_i |^2 -1\Big)  d \cdot  \varphi,
\end{equation*}
up to a subsequence. The arbitrariness of the function $\varphi$ in $\mathcal{C}^\infty((0,T)\times \Omega)$ leads in particular to the weak-$\ast$ convergence in \eqref{eq:weak-convergence-d^3-d}. Additionally, always from \Cref{lemma:compactness_Omega_eps}, the director $d \in L^\infty((0,T)\times \Omega)$ satisfies also $\nabla d \in L^\infty(0,T;L^\frac{4}{3}(\Omega))$ and $\partial_t d \in L^2(0,T;;L^\frac{4}{3}(\Omega))$.

\noindent
Next we consider a vector function $\varphi \in \mathcal{C}^\infty([0,T]\times \overline{\Omega})$ with additionally $\supp \varphi \subset [0,T)\times \overline{\Omega}$ and we apply the weak formulation of \Cref{maineq3} to the test function $\varphi_{|\Omega_{\ee}}$: 
\begin{equation}\label{eq_depsk}
	\begin{aligned}
		-\int_0^T\int_{\Omega}  \tilde d_{\ee} \cdot 
              \partial_t \varphi - 
		\int_{\Omega}   \tilde d_{\ee ,\rm in} \cdot 
            \varphi
		&+ 
            \int_0^T\int_{\Omega} 
            \big(
                \tilde u_{\ee}
                \cdot 
                \widetilde{\nabla d_{\ee}}
            \big)\cdot 
            \varphi
		+
            \int_0^T\int_{\Omega} 
            \widetilde{\nabla d_{\ee}}
            :
            \nabla \varphi 
		= 
		-
            \int_0^T\int_{\Omega} 
		(|\tilde d_{\ee}|^2-1) \tilde d_{\ee}      
            \cdot 
            \varphi.
	\end{aligned}
\end{equation}
We hence proceed to send $\ee \to 0$ (up to a subsequence) and we first observe that the third integral at the left-hand side of \eqref{eq_depsk} converges towards $0$:
\begin{equation*}
    \begin{aligned}
        \bigg|
        \int_0^T\int_{\Omega} 
        \big(
            \tilde u_\ee\cdot 
            \widetilde{\nabla d_\ee}
        \big)\cdot 
        \varphi        
        \bigg|
        &\leq 
        \| \tilde u_\ee \|_{L^2((0,T)\times \Omega)}
        \Big\| \widetilde{\nabla d_\ee} 
        \Big\|_{L^\infty(0,T; L^2(\Omega))}
        \| \varphi \|_{L^2(0,T; L^\infty(\Omega))}
        \\
        &\leq
        \| u_\ee \|_{L^2((0,T)\times \Omega_\ee)}
        \| \nabla d_\ee
        \|_{L^\infty(0,T; L^2(\Omega_\ee))}
        \| \varphi \|_{L^2(0,T; L^\infty(\Omega))}.
    \end{aligned}
\end{equation*}
Hence, invoking the Poincar\'e's inequality of \Cref{Poincare_ineq}, there exists a constant $C_{2, \Omega}>0$ which depends only on $\Omega$ such that
\begin{align*}
        \bigg|
        \int_0^T\int_{\Omega} 
        \big(
            \tilde u_\ee\cdot 
            \widetilde{\nabla d_\ee}
        \big)\cdot 
        \varphi        
        \bigg|
        &\leq
        C_{2, \Omega}\ee 
        \|  \nabla u_\ee \|_{L^2((0,T)\times \Omega_\ee)}
        \| \nabla d_\ee 
        \|_{L^\infty(0,T; L^2(\Omega_\ee))}
        \| \varphi \|_{L^2(0,T; L^\infty(\Omega))}
        \\
        &\leq
        C_{2, \Omega}\ee 
        \EE_0^2
        \| \varphi \|_{L^2(0,T; L^\infty(\Omega))} \to 0\quad \text{as} \quad \ee \to 0.
\end{align*}
We are therefore in the condition to pass to the limit in \eqref{eq_depsk} as $\ee \to 0$. 
Invoking once more \eqref{eq:weak-convergence-of-de-and-nablade} and \eqref{eq:weak-convergence-d^3-d}, we obtain
\begin{equation}\label{limit-of-eq_depsk}
	\begin{aligned}
		-
            \theta 
            \int_0^T\int_{\Omega}  d \cdot 
              \partial_t \varphi - 
		\theta
            \int_{\Omega}   d_{\rm in} 
            \cdot 
            \varphi
		+
           \theta
            \int_0^T\int_{\Omega} 
            \Xi
            :
            \nabla \varphi 
		= 
		-
            \theta
            \int_0^T\int_{\Omega} 
		(|d|^2-1)
            d
            \cdot 
            \varphi.
	\end{aligned}
\end{equation}
Next, we claim that $\Xi = \mathbb A \nabla d$ in $L^\infty(0,T;L^\frac{4}{3}(\Omega))$. 
To this end, we consider two scalar functions $\chi_1,\,\chi_2 $ in  $H^2(Y^*)$ that are solutions of the following equations: 
\begin{equation*}
        \left\{
        \begin{aligned}
	  &\begin{alignedat}{4}
		&- \Delta \chi_i = 0\qquad 
            &&
            \text{in }Y^*,\\
		&\mu \cdot \nabla_y \chi_i = -e^i \cdot \mu
            \qquad 
            &&
            \text{on }\partial \mathcal{T},
	 \end{alignedat}
      \\
      &\;\tilde 
            \chi_i \text{ is periodic in }Y\text{ and is average free},
        \end{aligned}
        \right.
\end{equation*}
for $i = 1, 2$, where $\mu = \mu(y)$ denotes the inner normal vector on the boundary $\partial \mathcal{T}$. Given that $\tilde{\chi_1}$ and $\tilde{\chi_2}$ are periodic in $Y$, we extend them periodically across $\mathbb{R}^2$. We thus define the function $\psi_i(y) = y \cdot e^i + \tilde \chi_i(y)$ for all $y \in \mathbb R^2$, ensuring that $\mu \cdot \nabla_y \psi_i = 0$ at any boundary $\partial \mathcal{T}+ k$, with $k \in \mathbb Z^2$. For any $\ee>0$, we then introduce the rescaled function $\psi_{\varepsilon,i}: \mathbb R^2 \mapsto \mathbb R$ as follows:
\begin{equation*}
    \psi_{\ee,i}(x) := \ee \, \psi_i \left( \frac{x}{\ee}\right)
    = 
    x\cdot e_i
    +
    \ee \tilde \chi_i \left( \frac{x}{\ee}\right),\qquad x \in \mathbb R^2.
\end{equation*}
The functions $\psi_{\ee,i}$ are uniformly bounded in $H^1(\Omega_\ee)\cap L^\infty (\Omega_\ee)$, for $\ee>0$. Furthermore, making use of a standard two-scale method, 
\begin{equation}\label{weak-conv-of-psi-nabla-psi}
	\tilde \psi_{\ee,i} 
        \stackrel{\ast}{\rightharpoonup} \theta \,x \cdot e^i \quad \text{weakly-$\ast$ in } L^\infty(\Omega ),
	\qquad
	\widetilde{\nabla \psi_{\ee,i}} 
        \rightharpoonup  
        \theta  
        \fint_{Y^*} 
        \nabla_y \psi_{i}(y)dy
	  \quad
        \text{weakly in }L^2(\Omega).
\end{equation}
Additionally, we remark that for any $i,j \in \{ 1,2\}$
\begin{align*}
    e^j\cdot \int_{Y^*} \nabla_y \psi_i 
    &=
    \int_{Y^*} \partial_{y_j} \psi_i 
    =
    \int_{Y^*}  \partial_{y_j} \psi_i   - 
    \int_{Y^*}  \underbrace{\Delta_y\chi_i }_{=0} \chi_j \\
    & =
    \int_{Y^*}  \partial_{y_j} \psi_i   -
    \int_{\partial \mathcal{T}} (\mu \cdot \nabla_y \chi_i ) \chi_j   +
    \int_{Y^*} \nabla_y \chi_i  \cdot \nabla_y \chi_j  \\
    & =
    \int_{Y^*}  \partial_{y_j} \psi_i   +
    \int_{\partial \mathcal{T}} (\mu \cdot e^i) \chi_j  +
    \int_{Y^*} \nabla_y \chi_i \cdot \nabla_y \psi_j -
    \int_{Y^*} \nabla_y \chi_i\cdot e^j
    \\
    &= \int_{Y^*}  \partial_{y_j} \psi_i   +
    \int_{\partial \mathcal{T}} (\mu \cdot e^i) \chi_j  +
    \int_{Y^*} \nabla_y \psi_i \cdot \nabla_y \psi_j - 
    \int_{Y^*} e^i \cdot \nabla_y \psi_j
    -
    \int_{Y^*}\partial_{y_j} \psi_i 
    +
    \int_{Y^*} e^i \cdot e^j\\
    &=  
    \int_{\partial \mathcal{T}} (\mu \cdot e^i) \chi_j  +
    \int_{Y^*} \nabla_y \psi_i \cdot \nabla_y \psi_j - 
    \int_{Y^*} \partial_{y_i} \psi_j
    +
    \int_{Y^*} e^i \cdot e^j\\
    &=  
    \int_{\partial \mathcal{T}} (\mu \cdot e^i) \chi_j  +
    \int_{Y^*} \nabla_y \psi_i \cdot \nabla_y \psi_j - 
    \int_{Y^*} \partial_{y_i} \chi_j 
    =
    \int_{Y^*} \nabla_y \psi_i \cdot \nabla_y \psi_j.
\end{align*}
Therefore, for any $i,j\in \{1,2\}$, we deduce that
\begin{equation}\label{weak-conv-of-psi-nabla-psi-2}
    e^j\cdot \widetilde{\nabla \psi_{\ee,i}} 
    \rightharpoonup 
    \theta  e^j\cdot 
    \fint_{Y^*} 
    \nabla_y \psi_{i}
    =
    \theta 
    \fint_{Y^*} 
    \nabla_y \psi_i \cdot \nabla_y \psi_j 
    = 
    \frac{1}{|Y|}
    \int_{Y^*} 
    \nabla_y \psi_i \cdot \nabla_y \psi_j
    =
    \mathbb A_{ij}
    \quad \text{weakly in }
    L^2(\Omega).
\end{equation}
Next, we consider a vector function $\phi$ with components in $\mathcal{C}^\infty_c([0,T)\times \Omega)$ and we apply the weak formulation in \eqref{eq_depsk} to $\varphi = \phi \psi_{\ee, i}$: 
\begin{equation}\label{eq:d-equation-applied-to-psi-phi-1}
\begin{aligned}
	-
        \int_0^T\int_{\Omega_\ee} 
        d_\ee           
        \cdot \partial_t \phi \, \psi_{\ee,i}
        -
        \int_{\Omega_\ee} 
        d_{\ee, \rm in} \cdot \phi \psi_{\ee,i}  
        &+
	\int_0^T\int_{\Omega_\ee}
        (\uu_\ee\cdot \nabla d_{\ee})\cdot \phi \psi_{\ee,i} 
        +
        \\
        &+
	\int_0^T\int_{\Omega_\ee} 
        \nabla d_\ee: \nabla 
        (\phi \psi_{\ee,i}) = 
        -
	\int_0^T\int_{\Omega_\ee} 
        (|d_\ee|^2-1)d_\ee
        \cdot 
        \phi
        \psi_{\ee,i}
        .
\end{aligned}
\end{equation}
The third integral can be decomposed into 
\begin{equation*}
\begin{aligned}
    \int_0^T\int_{\Omega_\ee} 
    \nabla d_\ee: 
    &\nabla (\phi \psi_{\ee,i})
    = 
    \int_0^T\int_{\Omega_\ee} 
    \nabla d_\ee: 
    \nabla \phi 
    \,
    \psi_{\ee,i} 
    +
    \int_0^T\int_{\Omega_\ee} 
    \nabla d_\ee: 
    \nabla  \psi_{\ee,i}
    \otimes 
    \phi
    \\
    &=
    \int_0^T\int_{\Omega_\ee} 
     \nabla d_\ee: 
    \nabla \phi 
    \,
    \psi_{\ee,i} 
    +
    \int_0^T
    \int_{\Omega_\ee} 
    \nabla\big( \phi \cdot  d_\ee\big)
    \cdot 
    \nabla  \psi_{\ee,i}
    -
    \int_0^T
    \int_{\Omega_\ee} 
    \nabla \phi: 
    \nabla  \psi_{\ee,i}
    \otimes 
    d_\ee
\end{aligned}
\end{equation*}
The second integral on the right-hand side is null, since $-\Delta \psi_{\ee,i} = 0$ in $\Omega_\ee$, $\nu \cdot \nabla \psi_{\ee,i} = 0$ on $ \partial \Omega_\ee\setminus \partial \Omega$ and $d_\ee \cdot \phi $ is identically null on a neighborhood of $\partial \Omega$. Hence, we obtain
\begin{equation}\label{eq:nablad:nabla(phipsi)}
    \int_0^T\int_{\Omega_\ee} 
    \nabla d_\ee: 
    \nabla 
    (\phi \psi_{\ee,i})
    =
    \int_0^T\int_{\Omega_\ee} 
    \nabla d_\ee: 
    \nabla \phi
    \,
    \psi_{\ee,i} 
    -
    \int_0^T\int_{\Omega_\ee} 
    \nabla \phi 
    :
    \nabla  \psi_{\ee,i}\otimes d_\ee.
\end{equation}
Coupling \eqref{eq:d-equation-applied-to-psi-phi-1} together with  \eqref{eq:nablad:nabla(phipsi)} and integrating by parts in time, 
we deduce that
\begin{equation}\label{eq:d-equation-applied-to-psi-phi-2}
\begin{aligned}
	-
        \int_0^T\int_{\Omega_\ee} 
        d_\ee         
        \cdot\partial_t \phi 
        \,
        \psi_{\ee,i}
        &-
        \int_{\Omega_\ee} 
        d_{\ee,\rm in}           
        \cdot 
        \phi_{|t = 0}
        \psi_{\ee,i}
        +
	\int_0^T
        \int_{\Omega_\ee}
        \big(
            \uu_\ee\cdot 
            \nabla d_\ee
        \big)\cdot 
        \phi 
        \psi_{\ee,i} 
        +
        \\
        &
        +
        \int_0^T\int_{\Omega_\ee} 
        \nabla d_\ee: 
        \nabla \phi 
        \,
        \psi_{\ee,i}
        -
        \int_0^T\int_{\Omega_\ee} 
        \nabla \phi
        :
        \nabla
        \psi_{\ee,i}
        \otimes 
        d_\ee
        = 
        -
	\int_0^T
        \int_{\Omega_\ee} 
        (|d_\ee|^2-1)
        d_\ee
        \cdot 
        \phi
        \psi_{\ee,i}
        .
\end{aligned}
\end{equation}
We next analyse each term of \eqref{eq:d-equation-applied-to-psi-phi-2}, as $\ee \to 0$. By invoking \Cref{Poincare_ineq} and \Cref{rmk:uniform-estimates-norms-solutinos}, we first remark that the convective term vanishes:
\begin{equation*}
\begin{aligned}
    \Big|
    \int_0^T\int_{\Omega}
    &\big(
        \uu_\ee\cdot 
         \nabla d_\ee 
    \big)\cdot \phi 
     \psi_{\ee,i} 
    \Big|
    \leq
    \| u_\ee \|_{L^2((0,T)\times \Omega_\ee)}
    \| \nabla d_\ee \|_{L^\infty(0,T;L^2(\Omega_\ee))}
    \| \phi \|_{L^2(0,T;L^\infty(\Omega))}
    \| \psi_{\ee,i} \|_{L^\infty(\Omega_\ee)}
    \\
    &\leq
    C_{2,\Omega} \ee
    \| \nabla u_\ee     \|_{L^2((0,T)\times \Omega_\ee)}
    \| \nabla d_\ee     \|_{L^\infty(0,T;L^2(\Omega_\ee))}
    \| \phi             \|_{L^2(0,T;L^\infty(\Omega))}
    \Big(
        \| x\cdot e^i \|_{L^\infty(\Omega_\ee)} 
        +
        \ee 
        \| \tilde \chi_{\ee,i}  \|_{L^\infty(\Omega_\ee)} 
    \Big)
    \\
    &\leq
    C_{2,\Omega} \ee
    \EE_0^2
    \| \phi             \|_{L^2(0,T;L^\infty(\Omega))}
    \Big(
        |\Omega| \max_{x \in \Omega} |x| 
        +
        \ee 
        \| \chi_{i}  \|_{L^\infty(Y^*)} 
    \Big)
    \to 0 \qquad \text{as }\ee \to 0.
\end{aligned}
\end{equation*}
We therefore turn our attention to the remaining terms in \eqref{eq:d-equation-applied-to-psi-phi-2}, employing \Cref{lemma:compactness_Omega_eps} repeatedly.

\noindent 
For the first integral in \eqref{eq:d-equation-applied-to-psi-phi-2} we employ \Cref{lemma:compactness_Omega_eps} with $q = 2$, $ s = \frac{4}{3} $, $r = 2$, $f_l = d_{\ee} = d_{\ee_l}$ and $ g_l = \partial_t \phi \psi_{\ee,i} = \partial_t \phi \psi_{\ee_l,i}$. Hence, recalling \eqref{eq:weak-convergence-of-de-and-nablade} and \eqref{weak-conv-of-psi-nabla-psi},
\begin{equation*}
    \lim_{\ee \to 0}
    \int_0^T
    \int_{\Omega_\ee} 
    d_\ee           
    \cdot
    \partial_t \phi 
    \psi_{\ee,i}
    = 
    \theta 
    \int_0^T\int_{\Omega} 
    d  \cdot\partial_t \phi (x\cdot e^i).
\end{equation*}
up to a subsequence. Next, we consider the second integral in \eqref{eq:d-equation-applied-to-psi-phi-2} and we remark that the $H^1(\Omega_\ee)$-norms of $d_{\ee, \rm in}$ are uniformly bounded by the assumptions of \Cref{main_thm}, while $\psi_{i, \ee} \rightharpoonup  \theta\, x\cdot e_i$ weakly in $L^2(\Omega)$ thanks \eqref{weak-conv-of-psi-nabla-psi}. Applying \Cref{lemma:compactness_Omega_eps} with $f_l = d_{\ee, \rm in} = d_{\ee_l, \rm in}$, $g_l = \psi_{i, \ee_l}$ (considering them as constant functions in time), $q = 2$, $ s= 2$ and $r = 2$, we gather
\begin{equation*}
    \lim_{\ee \to 0}
    \int_{\Omega_\ee} 
    d_{\ee, \rm in} \cdot \phi_{|t = 0}\,\psi_{\ee, i} 
    =
    \lim_{\ee \to 0}
    \frac{1}{T}
    \int_0^T
    \int_{\Omega_\ee} 
    d_{\ee, \rm in} \cdot \phi_{|t = 0}\,\psi_{\ee, i} 
    =
    \frac{\theta}{T}
    \int_0^T
    \int_{\Omega} 
    d_{\rm in} \cdot \phi_{|t = 0}\,(x\cdot e_i) 
    =
    \int_{\Omega} 
    d_{\rm in} \cdot \phi_{|t = 0}\,(x\cdot e_i).
\end{equation*}
Next, we address the fourth integral on the left-hand side of  \eqref{eq:d-equation-applied-to-psi-phi-2} and we apply once more \Cref{lemma:compactness_Omega_eps}. This time we set $f_l = \psi_{\ee,i} = \psi_{\ee_l,i}$ and $g_l$ each component of the tensor $\nabla d_{\ee} = \nabla d_{\ee_l}$ to gather
\begin{equation*}
    \lim_{\ee \to 0 }
    \int_0^T
    \int_{\Omega_\ee} 
    \nabla d_\ee: 
    \nabla \phi 
    \,
    \psi_{\ee,i}
    = 
    \theta 
    \int_0^T
    \int_{\Omega} 
    \Xi :
    \nabla \phi 
    \,
    (x\cdot e^i)
    .
\end{equation*}
A similar argument implies moreover that
\begin{equation*}
    \lim_{\ee \to 0}
    \int_0^T
    \int_{\Omega_\ee} 
    \nabla \phi: 
    \nabla 
    \psi_{\ee,i}
    \otimes 
    d_\ee 
    = 
    \theta 
    \int_0^T\int_{\Omega} 
    \nabla \phi  :
    \Big(\fint_{Y^*} \nabla_y \psi_y\Big)
    \otimes d.
\end{equation*}
Finally, recalling the weak-$\ast$ convergence in \eqref{eq:weak-convergence-d^3-d}, we apply \Cref{lemma:compactness_Omega_eps} with $f_l = \psi_{\ee,i} = \psi_{\ee_l,i}$ and $g_l = (|d_\ee|^2-1)d_\ee$ to obtain
\begin{equation*}
    \lim_{\ee \to 0}
    \int_0^T
    \int_{\Omega_\ee}
        (|d_\ee|^2-1)
        d_\ee
        \cdot 
        \phi
        \,
        \psi_{\ee,i}
        =
        \theta 
        \int_0^T\int_{\Omega}
        (| d |^2-1)
        d 
        \cdot 
        (x\cdot e_i)
        \phi.
\end{equation*}
Summarising, sending $\ee \to 0$ in \eqref{eq:d-equation-applied-to-psi-phi-2} and dividing the result by $\theta$, we have that the following identity is satisfied for any vector function $\phi$ with components in $\mathcal{C}^\infty_c([0,T)\times \Omega)$:
\begin{equation}\label{eq:almost-there1}
\begin{aligned}
	-
        \int_0^T\int_{\Omega} 
        d           
        \cdot 
        (x\cdot e_i)
        \partial_t \phi
        &-
        \int_{\Omega} 
        d_{\rm in}          
        \cdot 
        (x\cdot e_i)
        \phi_{|t = 0} 
        +
        \int_0^T\int_{\Omega} 
        \Xi
        :
        \nabla \phi 
        (x\cdot e_i)
        +
        \\
        &-
        \int_0^T\int_{\Omega} 
        \nabla \phi
        :
        \Big(
        \fint_{Y^*} \nabla_y \psi_i
        \Big)
        \otimes 
        d
        = 
	-
        \int_0^T\int_{\Omega} 
        (|d|^2-1)
        d
        \cdot 
        (x\cdot e_i)
        \phi
        .
\end{aligned}
\end{equation}
Moreover, applying \Cref{limit-of-eq_depsk} to $\varphi  = (x\cdot e_i)\phi$, we have that 
\begin{equation}\label{eq:almost-there2}
\begin{aligned}
	-
        \int_0^T\int_{\Omega} 
        d           
        \cdot 
        (x\cdot e_i)
        \partial_t \phi
        &-
        \int_{\Omega} 
        d_{\rm in}          
        \cdot 
        (x\cdot e_i)
        \phi_{|t = 0} 
        +
        \\
        &+
        \int_0^T\int_{\Omega} 
        \Xi:
        \nabla 
        \big(
            (x\cdot e_i)
            \phi
        \big)
        = 
	-
        \int_0^T\int_{\Omega} 
         (|d|^2-1) 
        d
        \cdot 
        (x\cdot e_i)
        \phi
        .
\end{aligned}
\end{equation}
Subtracting \eqref{eq:almost-there1} to \eqref{eq:almost-there2}, we hence deduce that
\begin{equation*}
    \int_0^T\int_{\Omega} 
    \Xi:
    \big[\nabla 
      (x\cdot e_i)
    \otimes 
    \phi 
    \big]
    +
    \int_0^T\int_{\Omega} 
    \nabla \phi 
    :
    \Big[
    \Big( 
        \fint_{Y^*}\nabla_y \psi_i
    \Big)
    \otimes 
    d
    \Big]
    = 0.
\end{equation*}
Integrating by parts the last integral, we obtain 
\begin{equation*}
    \int_0^T\int_{\Omega} 
    (\Xi^Te_i)
    \cdot 
    \phi 
    -
    \int_0^T
    \int_{\Omega} 
    \phi
    \cdot 
    \Big(
    \Big(
        \fint_{Y^*} \nabla_y \psi_i
    \Big)
    \cdot \nabla 
    \Big)d
    = 0.
\end{equation*}
From the arbitrariness of $\phi$ and recalling \eqref{weak-conv-of-psi-nabla-psi-2}, we deduce that
\begin{equation*}
    \Xi^Te^i = 
    \Big(
    \Big(
        \fint_{Y^*} \nabla_y \psi_i
    \Big)
    \cdot \nabla 
    \Big) 
    d 
    =
    \sum_{j = 1}^2
    \Big(
        e_j \cdot \fint_{Y^*} \nabla_y \psi_i
    \Big)
    \partial_{x_j} d = \sum_{j = 1}^2 \frac{1}{\theta} \mathbb A_{ij} \partial_{x_j} d,
\end{equation*}
which finally implies that 
\begin{equation*}
    \Xi =\frac{1}{\theta}  \mathbb A \nabla d.
\end{equation*}
Replacing this relation in \Cref{limit-of-eq_depsk} implies that 
\begin{equation*}
	\begin{aligned}
		-
            \int_0^T\int_{\Omega}  d \cdot 
              \partial_t \varphi - 
		\theta
            \int_{\Omega}   d_{\rm in} 
            \cdot 
            \varphi
		+
            \int_0^T\int_{\Omega} 
            \frac{1}{\theta}  \mathbb A \nabla d
            :
            \nabla \varphi 
		= 
		-
            \int_0^T\int_{\Omega} 
		(|d|^2-1)
            d
            \cdot 
            \varphi,
	\end{aligned}
\end{equation*}
for any $\varphi \in \mathcal{C}^\infty([0,T]\times \overline{\Omega})$ with $\supp \varphi \subset [0,T) \times \overline{\Omega}$. Thus $d$ is weak solution of \eqref{eq:d-equation-in-proposition}, which concludes the proof of \Cref{prop:convergence-of-d} and \Cref{main_thm}.
\end{proof}

\appendix
\section{A compactness-like result}
\noindent 
In this Appendix, we provide a detailed proof of a Sobolev inequality tailored for homogenisation and averaged functions, which has been instrumental in our earlier arguments for \eqref{eq:Poincar-Wirtinger-to-be-cited-at-the-appendix} and the proof of \Cref{lemma:compactness_Omega_eps} and \Cref{thm:Aubin-Lions-Linfty}.
\begin{lemma}\label{lemma:mean-in-Z}
    Let $s\in [1,\infty]$ and $Y^{*, \ee}_k,\, Y^{*, \ee}_j$ be two contiguous cells (i.e. two cells which share a common side). Denote by $Z^{*, \ee}_{kj} :=Y^{*, \ee}_k\cup Y^{*, \ee}_j$. There exists a  constant $C_s\geq 0$, which depends only on $Y^*$ and $s$, such that
    \begin{equation}\label{est:mean-in-Z}
        \Big|
            \fint_{Y^{*, \ee}_k} f(x) dx - 
            \fint_{Y^{*, \ee}_j} f(x) dx
        \Big|
        \leq 
        C_s \ee^{1-\frac{2}{s}} \| \nabla f \|_{L^s(Z^{*, \ee}_{kj})},
    \end{equation}
    for any function $f \in W^{1,s}(Z^{*, \ee}_{kj})$.
\end{lemma}
\begin{proof}
    We denote by $\tilde Y^*$ the unitary cell $\{ (x-x_k^\ee)/\ee,\, x \in Y^{*, \ee}_j\}$ 
    which is contiguous to $Y^*$ and by $Z^*$ the union $Y^*\cup \tilde Y^*$. Hence $Z^\ee_{kj}= \{ x_k^\ee + \ee z,\text{ with }z \in Z^*\}$. We claim that there exists a constant $C_s> 0$, which depends only on $Y^*$, such that 
    \begin{equation*}
        \Big|
            \fint_{Y^{*}} g(z) dz - 
            \fint_{\tilde Y^{*}} g(z) dz
        \Big|
        \leq 
        C  \| \nabla_z g \|_{L^p(Z^{*})},
    \end{equation*}
    for any function $g \in W^{1,s}(Z^{*})$. The inequality \eqref{est:mean-in-Z} follows thus by rescaling the variable $x = x_k^\ee + \ee z$, with $z \in Z^*$. Assume by contradiction that there is a non-trivial sequence $(g_n)_{n\in \mathbb N}$ in $W^{1,s}(Z^{*})$ satisfying 
    \begin{equation*}
        \Big|
            \fint_{Y^{*}} g_n(z) dz - 
            \fint_{\tilde Y^{*}} g_n(z) dz
        \Big|
        >
        n
        \| \nabla_z g_n \|_{L^s(Z^{*})}.
    \end{equation*}
    We consider the sequence $(\omega_n)_{n\in \mathbb N}$ in $W^{1,s}(Z^{*})$ defined by
    \begin{equation*}
        \omega_n(z) := 
        \frac{g_n(z) - \fint_{\tilde Z^{*}} g_n(y) dy}{\Big\| g_n - \fint_{\tilde Z^{*}} g_n(y) dy \Big\|_{L^p(Z^*)}}\quad \Rightarrow \quad  
         \Big|
            \fint_{Y^{*}} \omega_n(z) dz - 
            \fint_{\tilde Y^{*}} \omega_n(z) dz
        \Big|
        >
        n
        \| \nabla_z \omega_n \|_{L^s(Z^{*})},
    \end{equation*}
    so that $ \|\omega_n \|_{L^s(Z^*)}= 1$, as well as $\fint_{Z^*}\omega_n(z)dz = 0$. Moreover $\| \nabla \omega_n \|_{L^s(Z^*)}\leq \frac{2}{n|Y^*|}\| \omega_n \|_{L^1(Z^*)}\to 0 $ as $n \to \infty$, thus there exist a subsequence $(\omega_{n_m})_{m\in \mathbb N}$ and function $\omega$ in $W^{1,s}(Z^*)$, such that $\omega_{n_m}\to \omega $ strongly in $L^s(Z^*)$ and $\nabla_z \omega_{n_m} \rightharpoonup \nabla_z \omega$ weakly in $L^s(Z^*)$. Because of the strong convergence, $\omega$ is average free in $Z^*$ and its norm $\| \omega \|_{L^s(Z^*)} = 1$. Additionally, from the weak convergence, $\| \nabla_z \omega \|_{L^s(Z^*)} \leq \liminf_{m \to \infty} \| \nabla_z \omega_{n_m} \|_{L^s(Z^*)} = 0$, thus $\omega$ is constant in $Z^*$. Being average free, $\omega = 0$ everywhere in $Z^*$, which contradicts $\| \omega \|_{L^s(Z^*)} = 1$.
\end{proof}

\subsection*{Acknowledgment} During the course of this project, F.D.A.~was partially supported by the Bavarian Funding Programme for the Initiation of International Projects (BayIntAn\_UWUE\_2022\_139).
\\
A.Z.~has been partially supported by the Basque Government through the BERC 2022-2025 program. Additionally, he was supported by the Spanish State Research Agency through Severo Ochoa CEX2021-001142 and through project PID2020-114189RB-I00 funded by Agencia Estatal de Investigaci\'on
(PID2020-114189RB-I00 / AEI / 10.13039/501100011033). A.Z. was also partially supported  by a grant of the Ministry of Research, Innovation and Digitization, CNCS - UEFISCDI, project number PN-III-P4-PCE-2021-0921, within PNCDI III.  A.S. and A.Z. acknowledge the hospitality in Hausdorff Research Institute for Mathematics funded by the Deutsche Forschungsgemeinschaft (DFG, German Research Foundation) under Germany's Excellence Strategy EXC-2047/1-390685813. The authors would like to thank the Isaac Newton Institut for Mathematical Sciences for support and hospitality during the programme "The design of new materials" when work on this paper was undertaken. This work was supported by EPSRC grant numbers EP/K032208/1 and EP/R014604/1.

\bibliographystyle{abbrv}
\bibliography{literature}

\end{document}